\documentclass[10pt, a4paper]{amsart}

\usepackage{hyperref,cleveref,graphics,mathrsfs,bbm,mathtools}
\usepackage{texmac}
\crefname{condition}{condition}{conditions}
\Crefname{condition}{Condition}{Conditions}
\usepackage{enumitem}
\usepackage{bookmark}
\usepackage[all,cmtip]{xy}
\usepackage[alphabetic]{amsrefs}
\usepackage[left=3cm, right=3cm, top=3.5cm, bottom=3.5cm,footskip=1cm,headsep=1.5cm]{geometry}

\makeatletter
\newcommand{\leqnomode}{\tagsleft@true}
\newcommand{\reqnomode}{\tagsleft@false}
\makeatother

\calsymbols{c}{A,B,D,F,G,H,I,K,L,O,P,S,T,U,X,Y,C}
\bbsymbols{}{G,T,F}
\bbsymbols{b}{R,A,S}
\fraksymbols{f}{p,m,g,k,P,N,M,n,q,c,A}
\scrsymbols{s}{A,B,K,L,S,T,W,X,Y,k}
\mathbfsymbols{bf}{u,v,w,c,A,x,y,h}

\operators{PSL,nrd,Ad,vol,disc,free,Ann,ss,MaxSpec,split,ram,SO,SU,sign,AI,PS,O,im,temp}
\operators{ind,Cl,Art,gcd,cyc,tors,PR,Norm,Ext,N,im,Pell,nd,an,ur,cts,Frob,inv,res,div,prob,Alt,Map,Br,Li,Inj,nd,Sym,Surj}
\NewDocumentCommand\FbL{O{b} O{L}}{\cF^{#1}_{#2}}

\newcommand{\rou}{\boldsymbol\mu}
\usepackage[OT2,T1]{fontenc}

\DeclareSymbolFont{cyrletters}{OT2}{wncyr}{m}{n}
\DeclareMathSymbol{\Sha}{\mathalpha}{cyrletters}{"58}

\newcommand{\Pmax}[1]{P^+\!(#1)}
\newcommand{\ImageModSquares}{\bS}

\newcommand{\idelesmodsquaresset}[2]{\prod_{v\in #2}#1_v^\times/#1_v^{\times 2}}
\AtBeginDocument{%
   \def\MR#1{}
}
\begin{document}
\title{Hasse principle in twist families and distributions of $2$-Selmer groups}
 
\author{Alex Bartel}
\address{School of Mathematics and Statistics, University of Glasgow,
University Place, Glasgow G12 8SQ, United Kingdom}
\author{Adam Morgan}
\address{Gonville and Caius College, Trinity Street, Cambridge, CB2 1TA, United Kingdom}

\email{alex.bartel@glasgow.ac.uk, ajm269@cam.ac.uk}

\begin{abstract} 
We investigate the Hasse principle  in quadratic twist families
of genus $1$ hyperelliptic curves and more generally of $2$-coverings of a principally polarised abelian variety.  As a consequence, we obtain results on the existence of rational points on Kummer varieties.
The main technical ingredient is a result on the distribution
of $2$-Selmer ranks in certain sparse families of quadratic twists of a given abelian variety.
\end{abstract}
\maketitle
\setcounter{tocdepth}{1}
\tableofcontents

\section{Introduction}\label{sec:intro}
\subsection{Selmer group distribution}
  Let $A/\Q$ be a principally polarised abelian variety. In the breakthrough work \cite{Smith1, Smith2}, Smith determines, in great generality,
the distribution of the dimensions of the $2$-Selmer groups $\Sel_2(A_d/\Q)$ of quadratic twists $A_d$ of $A$, as $d$ varies in suitable families of square-free integers.
One situation that his results do not cover involves families of $d$ in which the Selmer groups have ``universal'' elements --- see \cite{Smith2}*{Definition 2.10}.
Motivated by applications to the Hasse principle for hyperelliptic curves and for Kummer varieties, and to statistics of Galois module structures, we consider families of quadratic twists defined by certain Frobenian conditions which typically have such universal elements.  
In this paper, we prove the main technical result on Selmer group distributions and the application to the Hasse principle.
The application to Galois module structures will be carried out in a separate paper.

For a Galois number field $L$, a finite set $\Sigma$ of places of $\Q$, and $b\in \idelesmodsquaresset{\Q}{\Sigma}$, define
\[
  \FbL = \{d\in \Z \text{ squarefree}:  d\equiv b \in \idelesmodsquaresset{\Q}{\Sigma}, (\text{prime }p\not\in \Sigma \text{ and }p|d)\Rightarrow p\text{ totally split in }L\}.
\]

As usual, by identifying the Galois modules $A_d[2]$ and $A[2]$, we view $\Sel_2(A_d/\Q)$ as a subgroup of $H^1(\Q,A[2])$.
Given $L$ and $b$ as above, we define in Section \ref{sec:syst_subspace} 
the \emph{systematic subspace} $\cS_b\subset H^1(\Q,A[2])$. It has the property 
  that we have $\cS_b\subset \Sel_2(A_d/\Q)$ for all $d\in \FbL$.

Our main technical result, proven as Theorem \ref{thm:main_distr}, is the following.
\begin{theorem}\label{thm:intro_distr}
  Let $A/\Q$ be a principally polarised abelian variety with full rational $2$-torsion and
  satisfying \Cref{assumption:simple_gamma_mod}, let $g$ be its dimension,
  and let $L$ be a Galois number field. Let $\Sigma$ be a finite set of places of $\Q$ containing
  $2$, $\infty$, all places at which $A$ has bad reduction, and all places that are ramified
  in $L$. For $r\in \Z_{\geq 0}$, define
  \[
    \alpha(r) = \prod_{j\geq 1}(1+2^{-j})^{-1}\prod_{j=1}^{r}\frac{2}{2^{j}-1}.
  \]
  Let $b\in \idelesmodsquaresset{\Q}{\Sigma}$ be such that $\FbL\neq \emptyset$.
 %
  Let $n_b\in \Z_{\geq 0}$ be the dimension of the systematic subspace $\cS_b$. Then there exists $m_b\in \{0,1\}$ such that
  for all $d\in \FbL$ we have $\dim\Sel_2(A_{d}/\Q)\equiv n_b+m_b\pmod 2$, and
  for all $r\in\Z_{\geq 0}$ we have
  \[
    \lim_{X\to \infty} \frac{\#\{d\in \FbL: |d| < X, \dim\Sel_2(A_d/\Q)=2g+n_b+m_b+2r\}}{\#\{d\in \FbL: |d| < X\}}=\alpha(2r+m_b).
  \]
\end{theorem}
When $L=\Q$, i.e. when there are no Frobenian conditions on the prime divisors of
$d\in \FbL$, the systematic subspace is trivial and \Cref{assumption:simple_gamma_mod} is equivalent, in the case of full rational $2$-torsion,
to the hypotheses in \cite{Smith2}*{Theorem 2.14} -- see Remark \ref{rmrk:favored}. In particular, in the case $L=\Q$ our theorem
reduces to a special case of \cite{Smith2}*{Theorem 2.14}.
To our knowledge, the theorem has not been known for any other field~$L$. We note however that Smith's main results apply over a general base number field.

\subsection{Hasse principle for genus $1$ hyperelliptic curves}
In \cite{BhargavaGenus1} Bhargava investigated how many genus $1$
hyperelliptic curves with affine model $y^2=f(x)$ have a rational point,
as $f$ runs over the set of separable quartic polynomials with integer coefficients
bounded by a given constant.
Here we investigate the analogous question for quadratic twist families.
Of course that question is only interesting if $f$ has no rational roots.
We assume, from now on, that $f$ has no rational roots and that, moreover,
the Galois group of $f$ is contained in the Klein $4$-group $V_4$, equivalently
that the Jacobian $A$ of the hyperelliptic curve $C:y^2=f(x)$ has full rational $2$-torsion.

Note that unlike in \cite{BhargavaGenus1}, 100\% of quadratic twists $C_d:dy^2=f(x)$ of such a genus
$1$ curve fail to be everywhere locally soluble, see e.g. \cite{Novak}.
Thus, the interesting question in this context is: for what proportion of square-free integers $d$ for which $C_d$ has points over all completions
of $\Q$ does it have a point over $\Q$?
In Theorem \ref{thm:globalPointsLowerBound} we deduce the following result from Theorem \ref{thm:intro_distr}.
\begin{corollary}\label{cor:intro_hyperell}
  Let $C/\Q: y^2=f(x)$ be a hyperelliptic curve, where $f$ is a monic quartic without rational roots and with Galois group contained in $V_4$,
  and let $A$ be its Jacobian.
  Assume that $A$ satisfies the hypotheses of \Cref{thm:intro_distr}.
  Let $L$ be the splitting field of $f$.
  With $\Sigma$ as in the theorem, let
  $b\in\idelesmodsquaresset{\Q}{\Sigma}$ be such that $n_b=1$, such that $\FbL\neq \emptyset$, and such that
  for all $d\in \FbL$, the curve $C_d$ has points everywhere locally
  and the $2^\infty$-Selmer corank $\rk_{2}(A_d/\Q)$ of $A_d/\Q$ is odd, i.e.,
  in the notation of \Cref{thm:intro_distr}, we have $m_b=0$.
  Assume that for all $d\in \FbL$ for which $\rk_{2}(A_d/\Q)=1$, the $2$-part of
  the Shafarevich--Tate group of $A_d/\Q$ is finite. Then we have
  $$
    \liminf_{X\to \infty}\frac{\#\{d\in \FbL: |d| < X, C_d(\Q)\neq \emptyset\}}{\#\{d\in \FbL: |d| < X\}}\geq 
      \prod_{j\geq 1}(1+2^{-j})^{-1}>0.41.
  $$
\end{corollary}

A recent AI-generated preprint \cite{OAIBSD} claims to prove that whenever $\rk_{2}(A_d/\Q)=1$, the
Shafarevich--Tate group of $A_d/\Q$ is finite, so the finiteness assumption in Corollary \ref{cor:intro_hyperell}
may be removable.

\subsection{Hasse principle for Kummer varieties}
Our final application of Theorem \ref{thm:intro_distr} concerns the Hasse principle
for Kummer varieties. Let $A/\Q$ be a principally polarised abelian variety and let $0\neq \zeta \in H^1(\Q,A[2])$.
We can associate with $\zeta$ a generalized Kummer variety $\sX_\zeta$. The existence
of rational points on classes of such Kummer varieties is the subject of several
works \cite{MR2183385,MR3519097,MR3932783,MR4934780}. These all use versions of
Swinnerton-Dyer's descent-fibration method to establish, conditional on finiteness
of relevant Shafarevich--Tate groups,  either the Hasse principle or the sufficiency of the Brauer--Manin obstruction.

For each squarefree integer $d$, denote by $Y_{\zeta,d}$ the $2$-covering of $A_d$
associated to $\zeta$. This is a torsor under $A_d$ equipped with an involution
$\iota_d$ compatible with multiplication by $-1$ on $A_d$. Denote by
$\widetilde{Y}_{\zeta,d}$ the blow up of $Y_{\zeta,d}$ at the fixed points of
$\iota_d$. There is a natural degree $2$ morphism $f_d\colon \widetilde{Y}_{\zeta,d}\to \sX_{\zeta}$
(the quotient by the unique extension of $\iota_d$). In each of the cited works,
one establishes the existence of a rational point on $\sX_{\zeta}$  by first proving
the existence of a squarefree integer $d$ for which  $\widetilde{Y}_{\zeta,d}(\Q)\neq \emptyset$.
As a result of Theorem \ref{thm:intro_distr}, in certain cases we are able not only
to establish the existence of such a $d$, but to quantify how many such $d$ there are. 

Let $L/\Q$ be the smallest Galois extension trivialising $\zeta$, and let $\Sigma$
be a finite set of places containing $2,\infty$, all places ramified in $L/\Q$,
and all places of bad reduction for $A$.
For an integer $d$, let $Y_{\zeta,d}(\bA_\Q)^{\Br(Y_{\zeta,d})}$
denote the Brauer--Manin set, see e.g. \cite{CT-Skoro}*{\S 13.3.1}. 

For the statement of Theorem \ref{thm:kummer_intro} below, which we will prove
as \Cref{prop:noBrauerManinObs},
recall that Condition \ref{assumption:simple_gamma_mod} asserts that the Galois
action on $A[4]$ is suitably large, subject to being trivial on $A[2]$.
In the terminology of the statement of Theorem \ref{thm:intro_distr},
Condition \ref{cond_E_new} is the assertion that there exists
$b\in \prod_{v\in \Sigma}\Q_v^{\times}/\Q_v^{\times 2}$ such that both the
systematic subspace $\cS_b$ is generated by $\zeta$, and the parity $m_b+n_b \pmod 2$
is equal to $1$. This condition is closely related to \cite[Condition E]{MR2183385};
see Remark \ref{rmk:condition_E} for more details.

%

\begin{theorem} \label{thm:kummer_intro}
Suppose that $A$ has full rational $2$-torsion. Let $0\neq \zeta \in H^1(\Q,A[2])$,
let $L/\Q$ be as above, and suppose  that Conditions \ref{assumption:simple_gamma_mod}
and \ref{cond_E_new} are satisfied.
Then we have
\[\#\{d\in \Z  ~~:~~d\textup{ squarefree},  ~|d|<X,~  Y_{\zeta,d}(\bA_\Q)^{\Br(Y_{\zeta,d})}\neq \emptyset\} \asymp  X(\log X)^{\frac{1}{[L:\Q]}-1}.\]
\end{theorem}

By a celebrated theorem of Manin, finiteness of Shafarevich--Tate groups of abelian varieties
implies that the Brauer--Manin obstruction is the only one to existence
of rational points on torsors under abelian varieties. The next result is deduced from \Cref{thm:kummer_intro} using Manin's argument.
\begin{corollary}\label{cor:kummer_intro}
Assume the notation and hypotheses of \Cref{thm:kummer_intro}, and
assume further that the $2$-primary part of the Shafarevich--Tate group of each
quadratic twist of $A$ is finite. Then we have
\[\#\{d\in \Z  ~~:~~d\textup{ squarefree},  ~|d|<X,~  f_d(\widetilde{Y}_{\zeta,d}\big(\Q)\big)\neq \emptyset\} \asymp  X(\log X)^{\frac{1}{[L:\Q]}-1}.\]
\end{corollary}

An immediate consequence is that, under the hypotheses of Corollary \ref{cor:kummer_intro},
one has $\sX_\zeta(\Q)\neq \emptyset$. To the best of our knowledge, this existence statement is already new.
In Proposition \ref{cond_Z_prop}, we show that Condition \ref{assumption:simple_gamma_mod}
is implied by a version of \cite[Conditions $\textup{Z}_1$ and $\textup{Z}_2$]{MR2183385}.
In particular, one can deduce from Theorem \ref{thm:kummer_intro} a similar result
to \cite[Theorem 1.1]{MR2183385} which, though working over $\Q$, allows abelian
varieties of arbitrary dimension; see Corollary \ref{cor:kummer_finite} and the
following remark. As alluded to in \cite[Remark 2.10]{MR3932783}, such a result
(i.e. Corollary \ref{cor:kummer_finite} minus the quantitative aspect) can almost certainly
be established via a more direct application of Swinnerton-Dyer's method, and should,
roughly speaking, be viewed as being weaker than the main result of Harpaz's work
\cite{MR3932783} (in particular, unlike that work, we include no second descent information).
Nevertheless, we believe there is much potential for proving new results about
rational points on Kummer varieties via analytic methods. Indeed, fully adapting
Smith's work \cite{Smith1,Smith2} to twist families defined by suitable Frobenian
conditions should  yield  a version of Theorem \ref{thm:kummer_intro} that applies
over general number fields, allows more general Galois structures on $A[2]$, and
incorporates information from higher descents.   


\subsection{Strategy of proof of Theorem \ref{thm:intro_distr}}\label{sec:introProof}
We broadly follow the strategy of Heath--Brown \cite{MR1292115}, but with some modifications,
and in several places take inspiration from the more recent works of Kane \cite{MR3101079}
and Smith \cites{Smith1,BSDGoldfeld}.
To determine the distribution of $\dim\Sel_2(A_d/\Q)$ over $d\in \FbL$ we
determine, for all $r\in \Z_{\geq 1}$, the $r$-th moments. For $X\in \R_{\geq 0}$ write $\FbL(X)=\{d\in \FbL: |d|<X\}$.
By the $r$-th moment one usually means
$\lim_{X\to\infty}\frac{1}{\#\FbL(X)}\sum_{d\in \FbL(X)}\#\Sel_2(A_d/\Q)^r$,
equivalently the average number of homomorphisms from $\F_2^r$ to $\Sel_2(A_d/\Q)$,
although we will in fact compute something slightly different, see below.

To determine the $r$-th moment, it is convenient to work with
the abelian variety $B=A^r$. One clearly has $\Sel_2(A_d/\Q)^r=\Sel_2(B_d/\Q)$. 

In common with the works of Kane and Smith, we begin by using the local Tate pairing to express, in \Cref{lem:kane_formula},
the characteristic function of $\Sel_2(B_d/\Q)\subset H^1(\Q,B[2])$
in terms of quadratic characters on local cohomology groups.

Write $D$ for the ``part of $d$ coprime to $\Sigma$''.
We parametrise both local and global cohomology
groups in terms of factorisations of $D$ into squarefree integers indexed by $B[2]$.
The expression for $\#\Sel_2(B_d/\Q)=\#\Hom(\F_2,\Sel_2(B_d/\Q))$
becomes a sum of Jacobi symbols indexed by factorisations of $D=\prod_{\bfu\in B[2]^2}D_{\bfu}$ -- see
\Cref{prop:explicit_selmer_formula_non-varying}. 
One wants to evaluate the sum of such expressions over $d<X$, in the limit as $X\to \infty$.
It turns out to be fruitful to partition that latter sum into subsums $\cS^T(X)$
according to the subsets $T = \{\bfu\in B[2]^2:D_{\bfu}\neq 1\}$. Thus, for example,
there are $\#B[2]^2$ subsums corresponding to singletons $T$, i.e. to trivial
factorisations of $D$ into just a single factor. This is explained in detail in
\Cref{def:support}.

At this point, a central role is played by the notion of linked indices. We define
in \Cref{quad_form_q_defi} what it means for two elements of $B[2]^2$ to be linked.
The role of \Cref{prop:error_term_F_K}
is to show that in the contribution to the $r$-th moment, subsums $\cS^T(X)$ for
subsets $T$ that are not maximal unlinked will go into the error term.

It remains to estimate the sum $\cS^T(X)$ over maximal unlinked subsets $T$.
In fact, we determine the average number of \emph{injective} homomorphisms
from $\F_2^r$ to $\Sel_2(A_d/\Q)$, or to be even more precise, from $\F_2^r$ to
the quotient $\Sel_2(A_d/\Q)/(\delta_d(A_d[2])+\cS_b)$ by the image of $2$-torsion and the systematic subspace,
which naturally produces the shift in the distribution by $2g+n_b$ appearing in the statement of Theorem \ref{thm:intro_distr}.

The systematic subspace is defined in Section \ref{sec:syst_subspace}.
\Cref{thm:inj_moments} explains that the passage from all homomorphisms
$\F_2^r\to \Sel_2(B_d/\Q)$ to those that are injections into $\Sel_2(B_d/\Q)/\cS_b$
corresponds to summing only over particular maximal unlinked subsets $T$.
That restriction turns out to be very convenient for the final count.

Section \ref{sec:combinatorics} is devoted to computing the main term
of the injection-count and deducing the distribution.
Maximal unlinked subsets turn out to be cosets of maximal isotropic subspaces
for a certain quadratic form, as in \cite{MR1292115}. One novelty
of our approach is that by shifting Selmer elements by the images
of $2$-torsion, we manage to reduce the analysis to only maximal
isotropic subspaces.

It is shown in \Cref{desc_of_max_iso}
that those maximal isotropic subspaces admit a very explicit linear-algebraic description:
they are parametrised by pairs $(U,P)$ consisting of a subspace $U$ of $B[2]$
and an alternating pairing $P$ on $U$. Moreover, it is shown in \Cref{main_invariant_gamma_computation_prop} that
the only such pairs $(U,P)$ contributing to the final count are those for which
$U$ is fixed by a subgroup  $\Gamma$ of $\End(B[2])$ arising from the natural Galois
action on $B[4]$, as described in Section \ref{sec:GalAction4Torsion}.
It is here that our assumption on the Galois image on $B[4]$ comes in. Under that assumption,
one can explicitly describe the $\Gamma$-invariant subspaces $U\subset B[2]$ -- see \Cref{simple_subspace_lemma}.

This is where the strategy of counting injections pays off. It turns out that for injections, the only relevant maximal isotropic
subspaces are those for which $U=B[2]$, so that the sum becomes only a sum over the alternating
pairings $P$. As above, only pairings that are, in a precise sense, compatible
with $\Gamma$ are shown to contribute to the sum. Writing $B[2]$ as $A[2]\otimes \F_2^r$,
the relevant pairings are exactly the ones of the form $e_2\otimes \langle,\rangle$,
where $e_2$ is the Weil pairing on $A[2]$ and $\langle,\rangle$ is an arbitrary symmetric
pairing on $\F_2^r$. In particular, the number of such pairings is $2^{r(r+1)/2}$, which is
exactly the $r$-th injection-moment -- see \Cref{thm:inj_moments_main}.

\begin{acknowledgements}
We would like to thank Alexei Skorobogatov for helpful discussions. Parts
of this research were funded by EPSRC Fellowships EP/P019188/1,
`Cohen–Lenstra heuristics, Brauer relations, and low-dimensional manifolds' and 
EP/V006541/1 `Selmer groups, Arithmetic Statistics and Parity Conjectures'.
During part of the period in which this work was carried out, AM was supported
by the Max-Planck-Institut f\"ur Mathematik in Bonn and the Leibniz fellow programme
at the Mathematisches Forschungsinstitut Oberwolfach.
\end{acknowledgements}

\subsection{Notation}
By $\mu$ we will denote the M\"obius function $\Z\to \{-1,0,1\}$.
If $k$ is a field, then we will denote by $\bar{k}$ a fixed algebraic closure
of $k$, by $k^s$ the separable closure inside $\bar{k}$, and by $G_k$ the absolute
Galois group of $k$, i.e. the Galois group of $k^s/k$. We denote by $k^{\times 2}$
the subgroup of squares in the unit group $k^\times$ of $k$. If $k$ is a global field,
then by $\cO_k$ we denote its ring of integers, and, for a place $v$ of $k$, 
by $k_v$ we denote the completion of $k$ at $v$.
If $U$ is a set, then $\charfunc_U$ will denote the characteristic function
of $U$ (on some superset of $U$). For a finite group, we will denote by $\triv$ the
trivial character on that group.

\section{Preliminaries on \texorpdfstring{$2$}{2}-Selmer groups}\label{sec:ellcurves}
In this section, let $k$ be a number field, and let $A/k$ be a
principally polarised abelian variety with full rational $2$-torsion. We denote
by $\Sigma$ a finite set of places of $k$ containing all places dividing $2$,
all archimedean places, and all places at which $A$ has bad reduction. Let $g=\dim A$. 
\subsection{Weil pairings}

For $n\in \Z_{\geq 1}$, let  
$
  e_n\colon A[n]\otimes A[n]\to \rou_n
$
denote the Weil pairing on $A[n]$ associated to the given principal polarisation,
where $\rou_n$ denotes the group of $n$-th roots of unity in $\bar{k}$.
We also write $e$ for $e_2$ but taking values in $\F_2$ instead of $\{\pm 1\}$,
so that we have $e_2(x,y) = (-1)^{e(x,y)}$ for all $x,y\in A[2]$.

\subsection{Quadratic twists}\label{sec:SelmerBackg}
Fix $d\in k^\times$, 
let $K=k(\sqrt{d})$  and let
$\psi_d\colon G_k\rightarrow \F_2$ be the associated quadratic character with kernel $G_K$.
We allow for the possibility that $\psi_d$ is trivial, i.e. that $d\in k^{\times 2}$.

Let $A_d/k$ denote the quadratic twist of $A$ by $K/k$, and
fix a $K$-isomorphism $\xi\colon A_d\to A$  
such that, for all $\sigma\in G_k$, one has $\xi^{-1}\xi^{\sigma}=(-1)^{\psi_d(\sigma)}$. Via the isomorphism $\xi$, the principal polarisation on $A$ descends to a principal polarisation on $A_d$ (\cite[Lemma 4.16]{MR3951582}). In this way, we view $A_d$ as a principally polarised abelian variety also. 
Note that $\xi$ induces a $G_k$-isomorphism $A_d[2]\cong A[2]$. This identifies the Weil pairing on $A[2]$ with the corresponding pairing on $A_d[2]$ (to see this, e.g.
combine \cite[Lemmas 4.5(ii) and 4.17]{MR3951582}).

Denote by $\delta\colon A(k)/2A(k)\hookrightarrow H^1(k,A[2])$ the coboundary 
homomorphism associated with the multiplication-by-$2$ Kummer sequence
\[
  0\longrightarrow A[2]\longrightarrow A(k^s)\stackrel{[2]}{\longrightarrow}A(k^s)\longrightarrow 0.
\]
 In what follows, use the map $\xi$ to identify $H^1(k,A_d[2])$ with $H^1(k,A[2])$. 
In this way, we view the corresponding coboundary homomorphism associated to  $A_d$ as taking values in $H^1(k,A[2])$. We denote the resulting homomorphism by $\delta_d\colon A_d(k)/2A_d(k)\hookrightarrow H^1(k,A[2])$.

\subsection{Selmer conditions}

For each place $v$ of $k$, denote by
\begin{equation*}
  \delta_v \colon A(k_v)/2A(k_v) \hookrightarrow  H^1(k_v,A[2])\quad\text{ and}\quad
  \delta_{d,v} \colon A_d(k_v)/2A_d(k_v)   \hookrightarrow   H^1(k_v,A[2])
\end{equation*}
the local versions of the maps $\delta$ and $\delta_d$, respectively.
We denote by $\sS_{d,v}$ the image of $\delta_{d,v}$. When $d\in k^{\times 2}$, we drop it from the notation and write $\sS_v=\im(\delta_v)$.  
For a nonarchimedean place $v$ of
$k$ we denote by $H^1_{\ur}(k_v,A[2])$ the set of classes in $H^1(k_v,A[2])$
that restrict trivially to the inertia group $I_v$ at $v$.

\begin{lemma} \label{lem:kummer_image_basic} 
Let $v\notin \Sigma$ be a place of $k$.  Then one has
\[\sS_{d,v}
= \begin{cases} \delta_{d,v}(A[2]),&v \textup{ ramifies in } k(\sqrt{d});\\
H^1_{\ur}(k_v,A[2]),&\textup{otherwise.} \end{cases}\]
\end{lemma}

\begin{proof}
When $v$ is unramified in $k(\sqrt{d})$, the abelian variety $A_d$ has good reduction at $v$ and the result is standard (see e.g.
  \cite[Remark 4.11 and Proposition 4.12]{MR2833483}). 
  Suppose that $v$ ramifies in $k(\sqrt{d})$. Arguing as in
  \cite[Lemma 5.2]{MR4400944}, we see that $A_d$ has no points of order $4$ defined over the maximal unramified extension of $k_v$. In particular, $A_d(k_v)$ contains no points of order $4$, hence
  $\delta_{d,v}$ is injective on $A[2]$. Thus we have $\dim_{\F_2}\delta_{d,v}(A[2])=\dim A[2]$.
  Since $v\nmid 2$, we also have 
  \[\dim A[2]=\dim A_d(k_v)/2A_d(k_v)=\dim \sS_{d,v},\] 
where the first equality follows e.g. from \cite[Proposition 3.9]{MR1370197}.
\end{proof}

\begin{lemma}\label{prop:KummerIntersection} 
  Let $v\not\in\Sigma$ be a place of $k$. Then:
  \begin{enumerate}[leftmargin=*, label=\upshape{(\arabic*)}]
    \item\label{item:normImUnram} if $v$ is unramified in $k(\sqrt{d})$, then we have
      $\sS_{v}=\sS_{d,v}$;
    \item\label{item:normImRam} if $v$ is ramified in $k(\sqrt{d})$, then we have
      $\sS_v\cap \sS_{d,v} = 0$.   
  \end{enumerate}
\end{lemma}
\begin{proof}
  Part \ref{item:normImUnram} follows from \cite{MR3519097}*{Lemma 4.1}
   and \cite{MR0444670}*{Corollary 4.4}. For part \ref{item:normImRam}
  see \cite{MR3519097}*{Lemma 4.3}.
\end{proof}

\subsection{Parity of Selmer ranks}\label{sec:ParityRanks}
Let $\rk_2(A/k)$ denote the $2^\infty$-Selmer rank of $A/k$,
\[
  \rk_2(A/k)=\rk_{\Z_2}\Hom\Big(\varinjlim_{n}\Sel_{2^n}(A/k),\Q_2/\Z_2\Big).
\]

\begin{notation}
For a place $v\in \Sigma$ and an element $b\in k_v^{\times}/k_v^{\times 2}$, define $\kappa_v(b)\in \Z/2\Z$ as
\[\kappa_v(b)=\dim_{\F_2} \sS_{v}/\sS_{v}\cap \sS_{b,v} \pmod 2.\]
We write $\kappa_v(d)$ for the result of evaluating $\kappa_v$ at the class of $d$ in  $k_v^{\times}/k_v^{\times 2}$. 
\end{notation}

\begin{theorem}\label{thm:2infty-Selmer rank}
We have $\rk_2(A_d/k)\equiv \dim_{\F_2}\Sel_2(A_d/k) \pmod 2$. Moreover, we have  
$$
    \rk_2(A_d/k) \equiv \rk_2(A/k) + \sum_{v\in \Sigma} \kappa_v(d)\pmod 2.
  $$
\end{theorem}
\begin{proof}
Let $\sha_{\nd}(A_d/k)$ denote the quotient of $\sha(A_d/k)$ by
its maximal divisible subgroup. We have 
 $\dim_{\F_2}\Sel_2(A_d/k) = \rk_2(A_d/k)+\dim_{\F_2}A_d(k)[2]+\dim_{\F_2}\sha_{\nd}(A_d/k)[2]$.
 By \cite{MR2915483}*{Proposition 2.12(a)}, the assumption that $A[2]$ is defined over $k$
  implies that $\dim_{\F_2}\sha_{\nd}(A_d/k)[2]$ is even. Since  $\dim_{\F_2}A_d(k)[2]=2g$,
  we see that $\rk_2(A_d/k)\equiv \dim_{\F_2}\Sel_2(A_d/k) \pmod 2$.
  
  The second assertion follows from \cite{MR3951582}*{Theorem 10.12 and Lemma 10.8}.
\end{proof}

\begin{corollary}\label{cor:2-Selmer rank}  
  Let $d'\in k^\times $ be such that $d$ and $d'$
  have the same class in $\idelesmodsquaresset{k}{\Sigma}$.
  Then we have $\rk_2(A_d/k) \equiv \rk_2(A_{d'}/k) \pmod 2$
  and $$\dim_{\F_2}\Sel_2(A_d/k)\equiv \dim_{\F_2}\Sel_2(A_{d'}/k)\pmod 2.$$
\end{corollary}
\begin{proof}
Immediate from  \Cref{thm:2infty-Selmer rank}.   
\end{proof}

 \subsection{Local Tate pairings} \label{sec:local_tate}

For each place $v$ of $k$, composition of cup-product associated to the Weil pairing $e$
and the local invariant map gives a non-degenerate pairing 
\begin{equation} \label{eq:local_tate_pairing_E}
  \left \langle ~,~\right \rangle_v: H^1(k_v,A[2])\times H^1(k_v,A[2])\longrightarrow \tfrac{1}{2}\Z/\Z\cong \F_2,
\end{equation}
with respect to which the subspace $\sS_v$   is maximal isotropic (see \cite[Proposition 4.10]{MR2833483}).

We denote by $\left \langle ~,~\right \rangle_\Sigma$ the sum of the local Tate
pairings \eqref{eq:local_tate_pairing_E} over $v\in \Sigma$.
It is a non-degenerate $\F_2$-valued bilinear pairing on $\oplus_{v\in \Sigma}H^1(k_v,A[2])$.
 
We denote by $H_v$ the  $\F_2$-valued pairing on $H^1(k_v,\F_2)$ given similarly by composition of cup-product and the local invariant map, and denote by $H_\Sigma$ its sum over $v\in \Sigma$.  For $d_1, d_2 \in k^{\times}$, we have $(d_1,d_2)_v=(-1)^{H_v(\psi_{d_1},\psi_{d_2})}$, where $(~,~)_v$ denotes the quadratic Hilbert symbol.

Since the action of $G_k$ on $A[2]$ is trivial, the natural cup-product map
\[H^1(k, \F_2)\otimes A[2] \longrightarrow H^1(k,A[2])=\Hom_{\cts}(G_k,A[2])\]  
is an isomorphism. We will frequently identify these groups, via this isomorphism, in what follows. By associativity of the cup-product, the analogous isomorphism with $k$ replaced by $k_v$ identifies the local Tate pairing $\left \langle ~,~\right \rangle_v$ with the pairing $H_v\otimes e$, and $\left \langle ~,~\right \rangle_\Sigma$ with $H_\Sigma \otimes e$.

\subsection{Tate quadratic forms}\label{tate_quad_form_subsec}
 
Let $\phi$ be an endomorphism of $A$ defined over $k$ that is invariant under the Rosati involution associated to the given principal polarisation.
Denote by $P_{\phi,2}\colon A[2]\otimes A[2]\rightarrow \rou_2$ the pairing defined by
$P_{\phi,2}(x,y)=e_2(x,\phi(y))$. This pairing is alternating and $G_k$-equivariant.
The same is true of the pairing $P_{\phi,4}\colon A[4]\otimes A[4] \rightarrow \rou_4$
defined by $P_{\phi,4}(x,y)=e_4(x,\phi(y))$.  

From these pairings, we can follow \cite[Definition 5.19, Example 5.20]{MS2021} to define
a Theta group $\mathcal{H}_\phi$ for $A[2]$ in the sense of \cite[Definition 5.1]{MS2021}. 
 Indeed, denote by $U_{\phi}$ the $G_k$-group with underlying $G_k$-set $\rou_4\times A[4]$, and with group structure given by $(a, x)\cdot (b,y)=(abP_{\phi,4}(x,y),x+y).$
This has a normal subgroup $N=\{(1,x)~~\colon~~x\in A[2]\}$, and we define  $\cH_{\phi}=U_{\phi}/N$. 
We have a commutative diagram of $G_k$-groups with exact rows 
\begin{equation} \label{theta_group_diag}
\xymatrix{1\ar[r]&\rou_4 \ar[r]^{}\ar@{=}[d]&U_{\phi}\ar[r]\ar[d]& A[4]\ar[r]\ar[d]^2&1\\
1\ar[r]&\rou_4\ar[r]^{} &\cH_{\phi}\ar[r]^{}& A[2]\ar[r]&1,}
\end{equation}
realising $\cH_{\phi}$ as a central extension of $A[2]$ by $\rou_4$.
The associated commutator pairing $A[2]\otimes A[2]\rightarrow \rou_4$ is $P_{\phi,2}$.   
For each place $v$ of $k$, define
$
  q_{\phi,v}\colon H^1(k_v,A[2])\rightarrow \Q/\Z 
$
to be the composition of the connecting map associated to the bottom row of \eqref{theta_group_diag} with the local invariant map,
and let
\[
  q_{\phi,\Sigma}\colon \oplus_{v\in \Sigma}H^1(k_v,A[2])\rightarrow \Q/\Z
\]
be the sum over $v\in \Sigma$ of the maps $q_{\phi,v}$.
We deduce from \cite[Proposition 2.9]{MR2915483} that $q_{\phi,v}$ and
$q_{\phi,\Sigma}$ are quadratic forms whose associated bilinear pairings send $(f,g)$ to
$\langle f, \phi(g)\rangle_v$, respectively $\langle f, \phi(g)\rangle_\Sigma$.
Since $\phi$ is an endomorphism of $A$, it follows from  \cite[Proposition 4.10]{MR2833483} that, for every $v\in \Sigma$, the isotropy condition \cite[Equation (5.18)]{MS2021}
is satisfied for the images $\sW_v$ of the multiplication-by-$4$ Kummer maps
\[
  \delta^{(4)}_v\colon A(k_v)/4A(k_v)\to H^1(k_v,A[4])
.\]
Since, for every $v\in \Sigma$,  we have $\sS_v=2\sW_v$,
it follows from \cite[Example 5.20]{MS2021} that we have
\[
  q_{\phi,v}(\sS_v)=0\text{ for }v\in \Sigma,\quad\text{ and }\quad
  q_{\phi,\Sigma}\big(\oplus_{v\in \Sigma}\sS_v\big)=0.
\]
By reciprocity for the Brauer group of $k$, the form $q_{\phi,\Sigma}$ also vanishes
on the image of the group $H^1_{\Sigma}(k,A[2])$ of global cocycles that are unramified
outside $\Sigma$. Moreover, if in the construction of $q_{\phi,v}$ the Galois module $A[4]$ is
replaced by $A_d[4]$ for $d\in k^\times$, a direct computation shows that one obtains the same quadratic form.
Thus, $q_{\phi,v}$ also vanishes on the image of $\sS_{d,v}$ for every $d\in k^\times$.

\begin{lemma}\label{boundary_map_theta}
Let $x\in A[2]$, let $v$ be a place of $k$, and let $\chi \in H^1(k_v,\rou_2)$.
Then the connecting maps associated to the bottom row of \eqref{theta_group_diag} send:
\begin{enumerate}[leftmargin=*,label=\upshape{(\roman*)}]
  \item\label{item:u} $ x\in H^0(k_v,A[2])$ to the image of $\delta(x)\cup_e \phi(x)$ in $H^1(k_v,\rou_4)$,
  \item\label{item:chicupu} $\chi \cup x \in H^1(k_v,A[2])$ to the image of $\delta(x)\cup_e \phi(x)\cup \chi$ in $H^2(k_v,\rou_4)$.
\end{enumerate}
Here $\cup_e$ denotes the cup-product with respect to the Weil pairing $e$ on $A[2]$.
\end{lemma}

\begin{proof}
\ref{item:u} Let $\overline{x}\in A[4]$ be such that $2\overline{x}=x$.
Then $(1,\overline{x})$ is a lift of $x$ to $U_{\phi}$ under the maps in \eqref{theta_group_diag}.
We then compute, for $\sigma \in G_{k_v}$, 
\[(1,\overline{x})^{-1}\cdot \sigma(1,\overline{x})=\big(P_{ \phi,4}(\overline{x},\sigma \overline{x}),\sigma \overline{x}-\overline{x}\big).\]
Since $\sigma \overline{x}-\overline{x}$ lies in $A[2]$, reducing the above expression modulo $N$
we see from commutativity of \eqref{theta_group_diag} that the image of $x$ under the connecting
map associated to the bottom row of \eqref{theta_group_diag} is the $1$-cocycle $\rho\colon G_{k_v}\rightarrow \rou_4$ defined by
\[\rho(\sigma)=P_{\phi,4}(\overline{x},\sigma \overline{x})=P_{\phi,4}(\overline{x}-\sigma \overline{x},\sigma \overline{x}),\]
the last equality following from the fact that $P_{\phi,4}$ is alternating.
Again using that $\sigma \overline{x}-\overline{x}$ lies in $A[2]$, this is equal to 
$e_2(\overline{x}-\sigma \overline{x}, \phi(x)) =\big(\delta(x)\cup \phi(x)\big)(\sigma),$
as desired.

\ref{item:chicupu} Let $\omega\colon G_{k_v}\rightarrow \{0,1\}\subseteq \mathbb{Z}$ be such that, for all $\sigma \in G_{k_v}$,
we have $\chi(\sigma)=(-1)^{\omega(\sigma)}$. Let $\theta \in \cH_{\phi}$ be any lift of $x$.
Then the function $\rho\colon G_{k_v}\rightarrow \cH_{\phi}$ sending $\sigma \in G_{k_v}$ to
$\rho(\sigma)=\theta^{\omega(\sigma)}$ gives a lift of $\chi \cup x$ to $\cH_{\phi}$.
We then compute, for $\sigma, \tau \in G_{k_v}$,
\begin{align*}
  \rho(\sigma)\cdot \sigma& \rho(\tau)\cdot \rho(\sigma \tau)^{-1}= 
  \theta^{\omega(\sigma)}\cdot \sigma(\theta)^{\omega(\tau)}\cdot \theta^{-\omega(\sigma \tau)}\\
  &= \theta^{\omega(\sigma)}\cdot \sigma(\theta)^{\omega(\tau)}\cdot \theta^{-\omega(\sigma)}\cdot \theta^{-\omega(\tau)}\cdot \theta^{(d\omega)(\sigma,\tau)}\\
  &= [\theta^{\omega(\sigma)},\sigma(\theta)^{\omega(\tau)}]\cdot (\sigma(\theta)\cdot \theta^{-1})^{\omega(\tau)}\cdot \theta^{(d\omega)(\sigma,\tau)}\\
  &=  [\theta^{\omega(\sigma)},\sigma(\theta)^{\omega(\tau)}]\cdot   [\sigma(\theta),\theta^{-1}]^{\omega(\tau)}\cdot (\theta^{-1}\cdot \sigma(\theta))^{\omega(\tau)}\cdot \theta^{(d\omega)(\sigma,\tau)},
\end{align*}
where $(d\omega)(\sigma,\tau)$ is equal to $2$ if $\omega(\sigma)=\omega(\tau)=1$, and is $0$ otherwise.
Since the commutator pairing associated to the bottom row of \eqref{theta_group_diag} is equal
to $P_{\phi,2}$, the first two terms vanish. With $\overline{x}$ as in the proof of \ref{item:u},
we can take $\theta$ to be the image of $(1,\overline{x})$ in $\cH_{\phi}$.
Then $\theta^2$ is the image of $(1,\overline{x})^2=(1,x)$ in $\cH_{\phi}$, which is trivial.
From this we conclude that the final term vanishes also. Finally, the computation in part \ref{item:u}
shows that  $(\theta^{-1}\cdot \sigma(\theta))^{\omega(\tau)}$ is a $2$-cocycle representing the class
of $\delta(x)\cup \phi(x)\cup \chi$. 
\end{proof}

\subsection{Galois action on the $4$-torsion}\label{sec:GalAction4Torsion}
Denote by $\Aut(A[4]/A[2])$ the subgroup of 
$\Aut A[4]$ consisting of elements that restrict to the trivial automorphism of $A[2]$. 

\begin{lemma} \label{lem:gamma_is_iso}
For $\phi \in \End(A[2])$, the map $f(\phi)\colon A[4] \to A[4]$ defined by \[f(\phi)(x)=x+\phi(2x)\]
 is an element of $\Aut(A[4]/A[2])$. The resulting map $f\colon\End(A[2]) \to \Aut(A[4]/A[2])$ is an isomorphism of groups.
\end{lemma}

\begin{proof}
Let $\phi \in \End(A[2])$ and let $f(\phi)$ be as in the statement. It is clear that $f(\phi)$ is an endomorphism of $A[4]$ which restricts to the identity on $A[2]$.  Let $x\in A[4]$ be such that $f(\phi)(x)=0$. Then $x=-\phi(2x)$ is an element of $A[2]$, hence $0=-\phi(2x)=x$. Thus $f(\phi)$ is injective, hence is an automorphism of $A[4]$ by counting. A straightforward computation shows that $f$ is a homomorphism. To complete the proof, we note that the inverse of $f$ is given by sending $\sigma \in \Aut(A[4]/A[2])$ to the endomorphism of $A[2]$ defined by 
\[y\longmapsto \sigma x - x \quad\text{ for any }x\in A[4]\text{ such that }2x=y.\]
Here the choice of $x$ is irrelevant, since $\sigma$ acts trivially on $A[2]$.
\end{proof}

\begin{definition} \label{not:gamma_def} \label{def:definition_of_gamma}
Since $A/k$ has full rational $2$-torsion, the absolute Galois group $G_k$ acts on $A[4]$ through a homomorphism $G_k \to \Aut(A[4]/A[2])$.  We define 
$\gamma\colon G_k \to \End(A[2])$
as the composition of this map with the isomorphism $\Aut(A[4]/A[2])\cong \End(A[2])$ afforded by Lemma \ref{lem:gamma_is_iso}. Explicitly, for $\sigma \in G_k$ and $x\in A[2]$, we have 
$\gamma(\sigma)(x)=\delta(x)(\sigma).$
\end{definition}

Note that $\gamma$ factors through $k(A[4])/k$, and as such is unramified at primes outside $\Sigma$.
It realises $\Gal(k(A[4])/k)$ as a subgroup of $\End(A[2])$.  

\begin{lemma} \label{lem:twisted_boundary_map}
Let $v$ be a place of $k$ and $d\in k^\times$. Then, for all $x\in A[2]$, we have $\delta_{d,v}(x)=\delta_{v}(x)+\psi_d \cup x$. For all $\sigma \in G_{k_{v}}$, we have 
\[\delta_{d,v}(x)(\sigma)=\gamma(\sigma)(x)+\psi_d(\sigma)x.\]
\end{lemma} 

\begin{proof}
Via the fixed identification of $A$ and $A_d$ over $k(\sqrt{d})$ we identify $A[4]$ and $A_d[4]$.
Unlike in the case of $2$-torsion, this identification does not respect the Galois action;
instead, $A_d[4]$ corresponds to $A[4]$ with the new action 
$\sigma \cdot y=(-1)^{\psi_d(\sigma)}\sigma(y).$
Now with $\sigma$ and $x$ fixed as in the statement, take $y\in A[4]$ such that $2y=x$. We then have 
\begin{eqnarray*}
\delta_{d,v}(x)(\sigma)&=& (-1)^{\psi_d(\sigma)}\sigma(y)-y\\
&=& (\sigma(y)-y)+((-1)^{\psi_d(\sigma)}-1)\sigma(y)\\
&=& \gamma(\sigma)(x)+\psi_d(\sigma)x.
 \end{eqnarray*}
This proves the second part of the statement. For the first,   note that the right hand side of the above formula is the result of evaluating $\delta_{v}(x)+\psi_d \cup x$ at $\sigma$.
\end{proof}

\begin{lemma} \label{gamma_and_weil_lemma}
Let $\sigma \in G_k$. Then for all $x,y \in A[2]$ we have
\[e(x,\gamma(\sigma)y)=e(\gamma(\sigma)x,y)+\psi_{-1}(\sigma)e(x,y).\] 
\end{lemma}

\begin{proof}
From the definition of $\gamma$, each $\sigma \in G_k$ acts on $A[4]$ as
\[x \longmapsto x+\gamma(\sigma)2x.\]
Moreover, for all $x,y \in A[4]$, we have $e_4(\sigma x, \sigma y)=\sigma e_4(x,y)$.
The result follows from substituting $\sigma x=x+\gamma(\sigma)2x$ and $\sigma y=y+\gamma(\sigma)2y$ into this identity. 
\end{proof}

\begin{definition}
Let $\End^+(A[2])$ denote the additive subgroup of $\End(A[2])$ consisting of elements $\phi$ such that
$e(x,\phi(y))=e(\phi(x),y)$
for all $x,y\in A[2]$ (i.e. consisting of endomorphisms that are self-adjoint for the Weil pairing). 
\end{definition}
 
 \begin{remark}
   By Lemma \ref{gamma_and_weil_lemma}, the image of $\Gal\big(k(A[4])/k(\sqrt{-1})\big)$ under $\gamma$ is contained in $\End^+(A[2])$. 
 \end{remark}
 
Note that $\End^+(A[2])$ is a sub-Lie algebra of $\End(A[2])$ (where we view the latter as a Lie algebra over $\F_2$ with Lie bracket given by the commutator).

\begin{proposition} \label{transvections_proposition}
For $x\in A[2]$, denote by $\phi_x\in \End(A[2])$ the endomorphism of $A[2]$ defined by $\phi_x(y)=e(y,x)x$.  
Then the following assertions hold.
\begin{enumerate}[leftmargin=*,label=\upshape{(\arabic*)}]
\item\label{item:generateEndPlus} The maps $\phi_x$, for $x\in A[2]$,
  lie in $ \End^+(A[2])$.
  \item \label{item:generateEndPlusLie} Let $x_1,...,x_{2g}$ be a basis for $A[2]$ as an $\F_2$-vector space. Suppose that, for all $i\neq j$, we have $e(x_i,x_j)=1$. Then the maps $\{\phi_{x_i}\}_{i=1}^{2g}$ generate $\End^+(A[2])$ as a Lie algebra over $\F_2$.
\item\label{item:EndPlusSimple} As an $\End^+(A[2])$-module, $A[2]$ is simple.
\item\label{item:endosScalars} The ring of $\End^+(A[2])$-equivariant endomorphisms of $A[2]$ is $\F_2$.
\end{enumerate}
\end{proposition}

\begin{proof}
\ref{item:generateEndPlus} Straightforward computation. 

\ref{item:generateEndPlusLie} This is essentially \cite[Lemma 3.2]{MR4316889}. Consider the  isomorphism 
\begin{equation} \label{eq:tensor_identification}
\End(A[2])\cong A[2]\otimes \Hom(A[2],\F_2)\cong A[2]\otimes A[2]
\end{equation}
 of $\F_2$-vector spaces,
where the first map is the standard isomorphism, and the second map is induced by the
Weil pairing on the second factor. Under this map, $\End^+(A[2])$ corresponds to the
subgroup $\Sym_2(A[2])$ of symmetric tensors. Moreover, for $x,y\in A[2]$, one computes that the map $\phi_x$ corresponds to the pure tensor $x\otimes x$, and the commutator $\phi_x\phi_y-\phi_y\phi_x$ corresponds to $e(x,y)(x\otimes y+y\otimes x).$
In particular, our assumptions mean that the image under  \eqref{eq:tensor_identification}  of the Lie algebra generated by $\{\phi_{x_i}\}_{i=1}^{2g}$  contains the elements $x_i\otimes x_i$, for $i\in \{1,....,2g\}$, and $x_i\otimes x_j+x_j\otimes x_i$  for all $i,j\in \{1,...,2g\}$ with $i<j$. Such elements span $\Sym_2(A[2])$ as an $\F_2$-vector space.

\ref{item:EndPlusSimple} Let $V$ be a subspace of $A[2]$ which is invariant under the action of $\End^+(A[2])$,
and suppose that a non-zero vector $y$ lies in $V$.  Given any $x\in A[2]$ satisfying $e(x,y)=1$,
we see that $\phi_x(y)=e(x,y)x=x$ lies in $V$. In particular, $V$ contains both $0$ and the
unique non-trivial coset of $\left \langle y\right \rangle ^\perp$. This is only possible if $V=A[2]$. 

\ref{item:endosScalars} Suppose $\alpha\in \End(A[2])$ commutes with all elements of $\End^+(A[2])$. Then for any $x,y\in A[2]$,  we have 
\[0=(\alpha \phi_x-\phi_x\alpha)(y)=e(x,y)\alpha(x)-e(\alpha(y),x)x.\]
Take $x\neq 0$ and take $y$ to be any element of $A[2]$ with $e(x,y)=1$. Then the above identity shows that $x\in A[2]$ is an eigenvector for $\alpha$. That is, every $x\in A[2]$ is an eigenvector for $\alpha$. This is only possible if $\alpha$ is a scalar multiple of the identity (if $v_1$ and $v_2$ are linearly independent, and if $\alpha(v_i)=\lambda_iv_i$ for $i=1,2$, then considering $\alpha(v_1+v_2)$ we see that $\lambda_1=\lambda_2$).
\end{proof}


\section{Analytic preliminaries}\label{sec:analytic}
Throughout the section we fix the following data.
Let $L$ be a Galois number field, let $n_L$ be its degree, let $\Sigma$ be a set
of places of $\Q$ containing $2$, $\infty$, and all places that ramify in $L$,
let $N$ be the product of the prime numbers corresponding to the finite places in $\Sigma$,
and let $\cP_L$ be the set of prime numbers $p\not\in \Sigma$ that are totally split in $L/\Q$.

Given a subset $U\subset (\Z/4NZ)^\times$, we will be interested in the family of squarefree integers
$$
\cF_{U,L} = \{n\in \Z_{\geq 1}: n\text{ squarefree, } p|n\Rightarrow p\in \cP_L, n\hspace{-0.5em}\mod{4N}\in U\}.
$$
For all $X\in \R$, let $\cF_{U,L}(X) = \{n\in \cF_{U,L}: n < X\}$. 

Note that if a prime number $p\not\in \Sigma$ is totally split in $L$, then the $p$-Frobenius automorphism
in $\Gal(\Q(\zeta_{4N})/\Q)$ is contained in the subgroup $S=\Gal(\Q(\zeta_{4N})/L\cap \Q(\zeta_{4N}))$.
Thus, if we identify $S$ with its image in $(\Z/4N\Z)^\times$ under the usual isomorphism $\Gal(\Q(\zeta_{4N})/\Q)\cong (\Z/4NZ)^\times$,
then we have $\cF_{U,L} = \cF_{U\cap S,L}$. 

If $U=\{c\}$ is a
singleton, we will abbreviate $\{c\}$ to $c$ in the subscript. If we have $U=(\Z/4N\Z)^\times$, then we will
drop it from the notation and just write $\cF_L$, respectively $\cF_L(X)$.
In particular, we also have $\cF_L = \cF_{S,L}$.

\subsection{Asymptotics for the families $\cF_{U,L}(X)$}\label{sec:family_densities}

\begin{lemma}\label{lem:average_chi}
  Let $\chi$ be a Dirichlet character modulo $4N$. Then we have
  $$
    \sum_{n\in \cF_L(X)}\chi(n) = c_\chi X (\log X)^{m(\chi)-1} + O(X(\log X)^{m(\chi)-2}),
  $$
  where
  $$
    m(\chi) = \leftchoice{1/n_L}{\chi|_S = \triv;}{0}{\text{otherwise},}
    $$
  and $c_{\chi}$ is a constant (possibly $0$).
\end{lemma}
\begin{proof}
  Define the function $\rho_{\chi}\colon \Z_{\geq 1}\to \C$ by
  $$
    \rho_{\chi}(n) = \mu^2(n)\chi(n)\prod_{p|n}\charfunc_{\cP_L}(p),
  $$
  where recall that $\mu$ denotes the M\"obius function, and the product runs over the prime
  divisors of $n$. This is
  a Frobenian multiplicative function in the sense of \cite{DanLilian}*{Definition 2.7},
  with corresponding Galois field $\tilde{L}=L\cdot\Q(\zeta_{4N})$. More precisely,
  for every prime number $p$ that is unramified in $\tilde{L}$ we have
  $$
  \rho_{\chi}(p) = \chi(\Frob_p)\charfunc_{\Gal(\tilde{L}/L)}(\Frob_p),
  $$
  where $\Frob_p\in \Gal(\tilde{L}/\Q)$ denotes any choice of Frobenius element at $p$,
  and the character $\chi$ is viewed as a character of $\Gal(\tilde{L}/\Q)$ lifted from
  the quotient $\Gal(\Q(\zeta_{4N})/\Q)\cong (\Z/4N\Z)^\times$.
  The sum in the statement of the lemma is equal to $\sum_{n=1}^X\rho_{\chi}(n)$.
  By \cite{DanLilian}*{Lemma 2.8} we have
  $$
  \sum_{n=1}^X\rho_{\chi}(n) = c_{\chi} X(\log X)^{m(\chi)-1} + O(X(\log X)^{m(\chi)-2}),
  $$
  where
  \begin{eqnarray*}
    m(\chi) & = & \frac{1}{[\tilde{L}:\Q]}\sum_{\sigma\in \Gal(\tilde{L}/\Q)}\chi(\sigma)\charfunc_{\Gal(\tilde{L}/L)}(\sigma)\\
                   & = & \frac{1}{[\tilde{L}:\Q]}\sum_{\sigma\in \Gal(\tilde{L}/L)}\chi(\sigma)\\
                   & = & \frac{1}{n_L}\langle\chi|_{\Gal(\tilde{L}/L)},\triv\rangle
  \end{eqnarray*}
  and $c_{\chi}$ is a constant.

  Here $\langle\; ,\; \rangle$ denotes the inner product of group characters, and $\triv$, recall, denotes the trivial
  character. Now considering $\chi$ as a character of $\Q(\zeta_{4N})/\Q$, that inner product
  is equal to $\langle \chi|_S,\triv\rangle$, which is equal to $1$ if $\chi$ restricts to the trivial character
  of $S$ and $0$ otherwise. This proves the lemma.
\end{proof}

\begin{proposition}\label{prop:asymptotic_cF}
  There exists a constant $c\in \R_{>0}$ such that for all $U\subset S$ one has
  $$
  \#\cF_{U,L}(X) = \#U\cdot cX(\log X)^{\frac{1}{n_L}-1} + O\left(\frac{X}{\log X}\right).
  $$
\end{proposition}
\begin{proof}
  For every $u\in S$, define $\rho_u\colon \Z_{\geq 1}\to \C$ by
  $$
  \rho_u(n) = \mu^2(n)\charfunc_{\{u\}}(n\hspace{-0.5em}\mod 4N)\prod_{p|n}\charfunc_{\cP_L}(p).
  $$
  Then we have
  $$
  \#\cF_{U,L}(X) = \sum_{u\in U}\sum_{n=1}^X\rho_u(n).
  $$
  Let $u\in S$. We have $\charfunc_{\{u\}} = \tfrac{1}{\varphi(4N)}\sum_{\chi} \overline{\chi(u)}\chi$,
  with the sum running over all characters of $(\Z/4N\Z)^\times$, and where $\varphi$ denotes Euler's totient function. Accordingly,
  we have
  \begin{eqnarray}\label{eq:sum_rho}
    \varphi(4N)\sum_{n=1}^X \rho_u(n) & = & \sum_{\chi}\overline{\chi(u)}\sum_{n=1}^X \mu^2(n)\chi(n)\prod_{p|n}\charfunc_{\cP_L}(p)\nonumber\\
                                      & = & \sum_{\chi}\overline{\chi(u)}\sum_{n\in \cF_L(X)}\chi(n),
  \end{eqnarray}
  where the summands $\chi$ run over all Dirichlet characters modulo $4N$.
  We now apply Lemma \ref{lem:average_chi}. Note that the sum \eqref{eq:sum_rho}
  is dominated by the summands of the outer sum corresponding to those $\chi$
  for which $m(\chi)$ is maximal, which are those that restrict to the trivial character on $S$ (assuming that the sum of the leading constants is non-zero, which we shall argue later).
  These characters $\chi$, in particular, satisfy $\chi(u)=1$ for all $u\in U$.
  For each of those, the inner sum is $c_{\chi}X(\log X)^{\frac{1}{n_L}-1}+O(X(\log X)^{\frac{1}{n_L}-2})$,
  while for the others it is $O(X(\log X)^{-1})$.
  It follows that we have
  $$
    \#\cF_{U,L}(X) = \#U\cdot cX(\log X)^{\frac{1}{n_L}-1}+O(X(\log X)^{-1})
  $$
  with $c=\tfrac{1}{\varphi(4N)}\sum_{\genfrac{}{}{0pt}{}{\chi:}{\chi|_{S}=\charfunc}}c_{\chi}$, where the sum runs
  over the characters of $(\Z/4N\Z)^\times$.

  It remains to show that $c$ is non-zero. Now, the function $\rho\colon \Z_{\geq 1}\to \C$
  given by $\rho(n) = \mu^2(n)\prod_{p|n}\charfunc_{\cP_L}(p)$ is a Frobenian multiplicative
  function: for every prime number $p\not\in \Sigma$ we have 
  $\rho(p) = \charfunc_{\{1\}}(\Frob_p)$, where $\Frob_p$ is a
  Frobenius element at $p$ in $\Gal(L/\Q)$. Moreover, we have
  $$
    \#\cF_L(X) = \sum_{n=1}^X \rho(n).
  $$
  By \cite{DanLilian}*{Lemma 2.8} this is equal to $c_{\rho}X(\log X)^{m(\rho)-1} + O(X(\log X)^{m(\rho)-2})$,
  where $m(\rho) = \frac{1}{n_L}>0$ and $c_{\rho}$ is real and positive.
  Since we have $c_{\rho} = \#S\cdot c$, this proves the proposition.
\end{proof}

\subsection{Auxiliary analytic results}
In  the rest of this section, we collect some additional analytic results that will be useful later. We begin by recalling some well-known results. 

For a positive integer $n$, we denote by $\omega(n)$ the number of distinct prime factors of $n$.
\begin{lemma}[\cite{NormalNumberDivisors}*{\S2.2, Lemma A}]\label{lem:hardy_ramanujan}
There is an absolute constant $B_0$ such that, for every real number $X\geq 3$
and for every integer $l\geq 1$, we have
\[
  \#\{n\leq X~~\big \vert~~\omega(n)=l,~\mu^2(n)=1\}\leq \frac{B_0X}{\log(X)}\cdot \frac{(\log\log(X)+B_0)^{l-1}}{(l-1)!}.
\]
\end{lemma}

\begin{proposition} \label{number of prime factors}
Let $A>0$ be a real number. Then there is a constant $c(A)>0$ such that, setting
$\Omega=c(A)e\cdot(\log\log(X)+B_0)$ for $B_0$ as in \Cref{lem:hardy_ramanujan}, we have
\begin{eqnarray}\label{eq:many_divisors}
  \#\{n\leq X ~~\big \vert~~\mu^2(n)=1,~\omega(n)>\Omega\}\ll X\log(X)^{-A}.
\end{eqnarray}
\end{proposition}
\begin{proof}
Fix some $c(A)>0$ to be determined later, and let $\Omega$ be as in the statement.
By \Cref{lem:hardy_ramanujan} and Stirling's formula, the left hand side of
\eqref{eq:many_divisors} is bounded above by 
\[
  \frac{B_0X}{\log(X)}\sum_{l>\Omega}\frac{(\log\log(X)+B_0)^{l-1}}{(l-1)!}\ll \frac{X}{\log(X)}\sum_{l>\Omega}\left(\frac{e\cdot(\log\log(X)+B_0)}{l-1}\right)^{l-1}.
\]
Since $l$ is always larger than $\Omega$ in the sum, we find 
\[
  \#\{n\leq X ~~\big \vert~~\mu^2(n)=1,~\omega(n)>\Omega\}\ll \frac{X}{\log(X)} \sum_{l>\Omega}c(A)^{-l}\ll \frac{X}{\log(X)}\cdot c(A)^{-\Omega}.
\]
Choosing $c(A)$ large enough gives the result. 
\end{proof}

If $F/\Q$ is a Galois extension, $G$ is its Galois group, and $p$ is a prime number
that is unramified in $F/\Q$, then we denote by $(F/\mathbb{Q},p)$ the
conjugacy class in $G$ of the Frobenius elements at $p$ in $F/\Q$.
If $\chi$ is the character of a complex representation of $G$, we write 
\[
\chi(p)=\begin{cases}\chi\big((F/\mathbb{Q},p)\big),~~&~~p\textup{ unramified in }F/\Q; \\0,~~&~~\textup{otherwise.}\end{cases}\]

\begin{lemma} \label{eq:character_sum_chebotarev_lemma}
Let $d$ be a squarefree integer, and let $\chi$ be  the character of a non-trivial
irreducible complex representation of $\Gal(L(\sqrt{d})/\Q)$. 

Then for all real numbers $A>1$ and $X\geq 2$  we have 
\begin{equation} \label{eq:cheb_char_sum}
\bigg| \sum_{p\leq X}\chi(p)\bigg|\ll_A \sqrt{|d|}X\log(X)^{-A},
\end{equation}
where the sum is taken over all primes $p\leq X$, and the implied constant is independent of $d$ and $\chi$ (but depends on $L$).
\end{lemma}

\begin{proof}
Fix $A>1$. Henceforth we let all implied constants depend on $L$ and $A$,
but not on $d$ or $\chi$. We  assume that $|d|\ll \log(X)^{2A-2}$, since
otherwise we can bound the left hand side of \eqref{eq:cheb_char_sum} by $\sum_{p\leq X}1$ and use the Prime Number Theorem. 

Write $F=L(\sqrt{d})$, write $G=\Gal(F/\Q)$, and let $C$ be a conjugacy class in $G$. Further, write
\[\pi_C(X)=\#\big\{p\leq X~~\colon ~~ (F/\mathbb{Q},p)  \in C\big\}.\]

\textbf{Claim:} We have 
\[\bigg| \pi_C(X)-\frac{|C|}{|G|}\Li(X) \bigg|\ll X\log(X)^{-A}.\]
\textbf{Proof of claim:} Denoting by $\Delta_F$ the discriminant of $F/\Q$, by
\cite[Lemma 7]{MR342472} we have $\log(|\Delta_F|)\ll \log(2|d|)$. Given our assumption
on $d$, we may thus apply the effective Chebotarev density theorem \cite{MR0447191}
(we use the precise form stated as \cite[Th\'{e}or\`{e}me 2]{MR644559}) to give 
\[\bigg| \pi_C(X)-\frac{|C|}{|G|}\Li(X) \bigg|\ll \Li(X^\beta)+X\log(X)^{-A},\]
where $\beta$ is the possible exceptional zero of the Dedekind zeta function
$\zeta_F(s)$ of $F$. By \cite[Lemma 8]{MR342472}, if such a $\beta$ exists there
is a quadratic subextension $K/\Q$ of $F/\Q$ such that $\zeta_K(\beta)$ is $0$ also.
Since the discriminant of any quadratic subextension is  $\ll |d|$, we can apply
Siegel's theorem \cite[Theorem 5.28 (2)]{MR2061214} to show that, for any fixed
$\epsilon>0$, we have $1-\beta\gg |d|^{-\epsilon}$. Choosing $\epsilon$ sufficiently
small in terms of $A$, our assumption on the size of $|d|$ gives
$\Li(X^\beta)\ll X\log(X)^{-A}$, proving the claim. 

Returning to the proof of the lemma, we compute 
\begin{eqnarray*}
\bigg| \sum_{p\leq X}\chi(p) \bigg| =\bigg|
\sum_{\substack{C\textup{ conj. class}\\\textup{of }G}}\quad \sum_{p\leq X,~ (F/\mathbb{Q},p)\in C}\chi(C)\bigg|= \bigg| \sum_{\substack{C\textup{ conj. class}\\\textup{of }G}}\chi(C)\pi_C(X) \bigg|,
\end{eqnarray*}
where $\chi(C)$ is the common value of $\chi$ on $C$. Both the number of conjugacy
classes of $G$, and the absolute value of any $\chi(C)$, is bounded independently
of $d$. Since $\chi$ is non-trivial, we have $\sum_{g\in G}\chi(g)=0$,
so the result follows from the claim.
\end{proof}

\begin{proposition} \label{prop:main_chebotarev_consequence}
Let $m$ be a
fixed positive integer and let $\psi$ be a Dirichlet character modulo $m$, viewed
as a $1$-dimensional character of $\Gal(\Q(\zeta_m)/\Q)$ (we allow $\psi$ to be trivial).  
Let $q$ be a positive integer, and let $\chi$ be a real Dirichlet character modulo $q$.
Suppose that for each irreducible character $\chi_i$ of $\Gal(L/\Q)$, the product
$\chi \chi_i \psi$, viewed as an irreducible character of $\Gal(L(\zeta_{mq})/\Q)$, is non-trivial.

Then for each positive integer $r$ and real numbers $\kappa\geq 1$, $A>1$, and $2\leq Y\leq X$, we have 
\begin{equation} \label{disgusting_sum_lemma_equation}
\bigg|\sum_{\substack{Y\leq n\leq X\\ n\in \cF_L}}\mu^2(nr)\kappa^{-\omega(n)}\chi(n)\psi(n)\bigg|\ll_{A,\kappa}\frac{(\omega(r)+\sqrt{q})X}{\log(X)^{A}}, \end{equation}
where the implied constant is independent of $Y$ and $\chi$, but depends on $L$ and $\psi$.  
\end{proposition}

\begin{proof}
The general case follows from the case $Y=2$, which we now restrict to. By
\Cref{number of prime factors} it suffices to prove the statement with the
sum running only over those $n$ that have at most
$\Omega$ prime factors, where $\Omega=\Omega(A)$ is defined as in the statement
of that Proposition. We now partition the sum \eqref{disgusting_sum_lemma_equation}
with that additional restriction according to the number $1\leq l\leq \Omega$ of
prime factors of $n$. Write $n=n_0p$ where $p$ is the largest prime factor of $n$,
denote by $\Pmax{n_0}$ the largest prime factor of $n_0$ if $n_0\neq 1$, and 
let $\Pmax{1}=1$.
Note that the condition $n_0p\leq X$ forces $n_0\leq X^{1-1/\Omega}$. By the
triangle inequality, it suffices to bound the sum 
\begin{equation}
 \sum_{l=0}^{\Omega-1}~\sum_{\substack{1\leq n_0\leq X^{1-1/\Omega}\\ n_0\in \cF_L\\ \omega(n_0)=l}}~\bigg| \sum_{\substack{\Pmax{n_0}< p\leq X/n_0\\ p\in \cP_L,~(p,r)=1} } \psi(p) \chi(p)\bigg|.
\end{equation}
We can drop the condition $(p,r)=1$ in the innermost sum at the expense of adding
at most $\omega(r)$ to its value, which we henceforth do. Replacing $r$ by $Nr$
we can assume that the condition $p\in \cP_L$ is simply the condition that $p$
splits completely in $L/\Q$. To treat this condition, we express the indicator
function of the identity element in $\Gal(L/\Q)$ as a sum $\sum_{1\leq i\leq t}a_i\chi_i$
for some complex numbers $a_i$ and irreducible characters $\chi_i$ of $\Gal(L/\Q)$.
Applying the triangle inequality once more, we see that it suffices to bound the sum 
\begin{equation} \label{disgusting_sum_equation_bound_1}
  \sum_{l=0}^{\Omega-1}~\sum_{\substack{1\leq n_0\leq X^{1-1/\Omega}\\ n_0\in \cF_L\\ \omega(n_0)=l}}~ \bigg(\max_{1\leq i \leq t}\bigg| \sum_{ \Pmax{n_0}\leq p\leq X/n_0 } \chi_i(p)\psi(p) \chi(p)\bigg|+\omega(r) \bigg).
\end{equation}
Since $\chi$ is a real Dirichlet character modulo $q$, it corresponds to an
irreducible character of $\Q(\sqrt{d})/\Q$ for some $d\ll q$. By assumption,
each $\chi_i \psi \chi$ is a non-trivial irreducible character of $L(\zeta_m)(\sqrt{d})$.
By \Cref{eq:character_sum_chebotarev_lemma}, for any $B>1$,  the quantity in \eqref{disgusting_sum_equation_bound_1}  is bounded above by some constant (which may depend on $B$) times the quantity
\begin{eqnarray*}  
 \sum_{l=0}^{\Omega-1}~\sum_{\substack{1\leq n_0\leq X^{1-1/\Omega}\\ n_0\in \cF_L\\ \omega(n_0)=l}}~ \left( \frac{\sqrt{q}X}{n_0}\log\left(\frac{X}{n_0}\right)^{-B}+\omega(r)\right) \ll_B  \omega(r) X^{1-1/\Omega}+\sqrt{q}X\left(\frac{\log(X)}{\Omega}\right)^{-B}\sum_{1\leq n_0\leq X}\frac{1}{n_0}.
\end{eqnarray*}
Choosing $B$ to be sufficiently large compared to $A$ gives the result. 
\end{proof}

We will also want the following bound on divisor sums over Chebotarev sets.

\begin{lemma} \label{lem:mertens:chebotarev_sum}
For every real $\kappa > 0$ we have 
\[\sum_{\substack{1\leq n \leq X\\ n\in \cF_L}}\mu^2(n)\kappa^{\omega(n)}\ll X\log(X)^{\frac{\kappa}{n_L}-1}.\]
\end{lemma}

\begin{proof}
Let $f(n)$ be the multiplicative function defined by 
\[f(p^k)=\begin{cases}\kappa,~~&~~  p\in \cP_L,~k=1;\\ 0,~~&~~\textup{otherwise.}\end{cases}\]
Then  \cite[Corollary 2.15]{MR2378655} with this choice of $f$ gives 
\begin{eqnarray*}
\sum_{\substack{1\leq n \leq X\\ n\in \cF_L}}\mu^2(n)\kappa^{\omega(n)}&\ll &
\frac{X}{\log(X)}\prod_{\substack{p\leq X \\
p\in \cP_L}}\left(1+\frac{\kappa}{p}\right)\ll
\frac{X}{\log(X)}\prod_{\substack{p\leq X \\ p\in \cP_L}}\left(1+\frac{1}{p}\right)^{\kappa}.
\end{eqnarray*} 
Now, we have
\[\prod_{\substack{p\leq X\\ p\in \cP_L}}\left(1+\frac{1}{p}\right)=
  \prod_{\substack{p\leq X\\ p\in \cP_L}}\left(1-\frac{1}{p^2}\right)
  \cdot \prod_{\substack{p\leq X\\ p\in \cP_L}}\left(1-\frac{1}{p}\right)^{-1}
\ll \prod_{\substack{p\leq X\\ p\in \cP_L}}\left(1-\frac{1}{p}\right)^{-1}\ll \log(X)^{1/n_L},
\]
where the final inequality is an analogue of Merten's theorem for Chebotarev sets;
see \cite{MR4439576} for a precise asymptotic for the product $ \prod_{\substack{p\leq X\\ p\in \cP_L}}(1-p^{-1})^{-1}$.
\end{proof}


\section{Moments of the Selmer group}\label{sec:fouvryklueners}
Let $A/\Q$ be a principally polarised abelian variety, let $g$ be its dimension,
and assume that $A$ has full rational $2$-torsion.  In this section we derive a
formula for the moments of the distribution of $\dim\Sel_2(A_d/\Q)$, where $d$
runs through a family of squarefree integers defined by certain frobenian conditions.
In the next section we will evaluate this formula under additional assumptions
on the Galois module $A[4]$, and will deduce \Cref{thm:intro_distr}.

\subsection{Setup} \label{sec:moments_setup}
Let $L/\Q$ be a finite Galois number field. Let $\Sigma$ be a finite set of
places containing $2$, $\infty$,  all places of bad reduction for $A$, and all
places that ramify in $L$. \\[0.3em]
\textit{The quadratic twist family.}
Let $N$ be the product of the prime numbers corresponding
to the finite places in $\Sigma$, let $d_0\in\Z$ be a divisor of $N$ (we allow $d_0<0$),
and let $c\in (\Z/4N\Z)^\times$.
Let $\cP_L$, $\cF_L$, $\cF_{c,L}$, and $\cF_{c,L}(X)$ be as defined in Section \ref{sec:family_densities}.
To ease notation in what follows, we omit $L$ from the subscripts. We write $d_0\cF_c=\{d_0D:D\in \cF_c\}$.
For all $r\in \Z_{\geq 1}$, we will compute the $r$-th moment of $\#\Sel_2(A_{d}/\Q)$
as $d$ runs over $d_0\cF_{c}$.

Let $S$ be the image of $\cF$ in $(\Z/4N\Z)^\times$, so that the family $\cF_{c}$
is non-empty if and only if we have $c\in S$. Equivalently,
$S$ is the image of the Galois group of $\Q(\zeta_{4N})/L\cap \Q(\zeta_{4N})$
under the natural isomorphism $\Gal(\Q(\zeta_{4N})/\Q)\to (\Z/4N\Z)^\times$.
We assume henceforth that we have $c\in S$. We remark that the pair $(c,d_0)$
determines a unique class $b=(b_v)_v \in \prod_{v\in \Sigma}\Q_v^{\times}/\Q_v^{\times 2}$
with the property that, for every $d\in d_0\cF_{c}$, we have
$d \equiv b \mod{\prod_{v\in \Sigma}\Q_v^{\times 2}}$. The constructions
in this section will typically depend on $(d_0,c)$ only through the class $b$.

Let $\ImageModSquares$ denote the image of $S$ in $(\Z/4N\Z)^\times/(\Z/4N\Z)^{\times 2}$. 
Write $K_{\Sigma}$ for the maximal multiquadratic extension of $\Q$  unramified at all primes outside $\Sigma$.
For every positive integer $m$ that is  
coprime to $N$, let $\sigma_m\in \Gal(K_\Sigma/\Q)$ denote the image of $m$
under the Artin map. The map $m\mapsto \sigma_m$ factors through 
$(\Z/4N\Z)^\times/(\Z/4N\Z)^{\times 2}$, and the image of $\ImageModSquares$
under this map is $\Gal(K_{\Sigma}/K_{\Sigma}\cap L)$. Recall that, for a squarefree integer
$m$,  we denote by $\psi_m$ the corresponding $\F_2$-valued quadratic character. Later on, we will often view $\ImageModSquares$ as a subgroup of $\oplus_{v\in \Sigma}H^1(\Q_v,\F_2)$, as explained in Notation \ref{notat:image_mod_sq_interpretation}. This identification is made in such a way that if $m\in \cF$ is a lift of $s\in \ImageModSquares$, then $s$ corresponds to the tuple $(\psi_m)_{v\in \Sigma}$. \\[0.3em]
\textit{The abelian variety $B$.}
Fix $r\in \Z_{\geq 1}$. We write $B=A^r$, and equip $B$ with the product polarisation.
For a squarefree integer $d$, we have $\Sel_2(B_d/\Q)=\Sel_2(A_d/\Q)^r$.   
 
We identify $B[2]$ with $B_d[2]$ as in  Section \ref{sec:ellcurves}. 
For each place $v$ of $\Q$ we write 
\[
  \delta_{d,v}^{(r)}\colon B_d(\Q_v)\longrightarrow H^1(\Q_v,B[2])
\]
for the coboundary map associated with the multiplication-by-$2$ Kummer sequence for $B_d$,
and denote by $\sS_{d,v}^{(r)}$ the image of $\delta_{d,v}^{(r)}$. When $d=1$, we drop it from the notation. 
 As in Section \ref{sec:ellcurves}, we denote by
$
  e\colon B[2]\times B[2]\rightarrow \F_2
$
the Weil pairing, taken to have values in $\F_2$.
It is the sum of the corresponding pairings on $A[2]$.
For each place $v$ of $\Q$, we denote by
\begin{equation} \label{eq:local_tate_pairing}
  \left \langle ~,~\right \rangle_v\colon H^1(\Q_v,B[2])\times H^1(\Q_v,B[2])\longrightarrow \F_2
\end{equation}
the local Tate pairing associated to $e$, and denote by  $\left \langle ~,~\right \rangle_\Sigma$ the sum, over $v\in \Sigma$, of these pairings.  

The map $\gamma\colon G_{\Q}\to \End A[2]$ of \Cref{not:gamma_def} factors through $\Gal(\Q(A[4])/\Q)$, hence through $\Gal(K_{\Sigma}/\Q)$.
Let $\Gamma\subset\End A[2]$ be the image of $\Gal\big(\Q(A[4])/\Q(A[4])\cap L\big)\cong \Gal(L(A[4])/L)$ under this map. 
Let $\gamma^{(r)}\colon G_{\Q}\to \End B[2]$ be the composition of $\gamma$ with the
 diagonal embedding $\End A[2]\to \End B[2]$. We view $\Gamma$ as operating on $B[2]$ via this map.

For $i\in \{1,2\}$, we  denote by $\pi_i\colon B[2]^2\to B[2]$ the $i$-th projection. Define $\Phi \colon B[2]^2\times B[2]^2\to \F_2$
to be  the map sending $(\bfu,\bfv)$   to 
$\Phi(\bfu ,\bfv )=e(\pi_1(\bfu) +\pi_1(\bfv), \pi_2(\bfv)).$
\\[0.3em]
\textit{Parametrising cohomology groups.}
For a finite set $\Sigma'$ of places, 
write $H^1_{\Sigma'}(\Q,B[2])$ for the subgroup of $H^1(\Q,B[2])=\Hom_{\cts}(G_{\Q},B[2])$
consisting of homomorphisms unramified outside $\Sigma'$.
Define \[V_0=H^1_\Sigma(\Q,B[2])=\Hom(\Gal(K_\Sigma/\Q),B[2])\quad \textup{ and }\quad Z_0=\oplus_{v\in \Sigma}\sS_{b_v,v}^{(r)}.\]

Given $d\in d_0\cF_c$ we write $\cP(d)$ for the set of primes outside $\Sigma$ which divide $d$
(thus if $d=d_0D$ with $D\in \cF_c$, then $\cP(d)$ is the set of prime factors of $D$).
Define  finite-dimensional $\F_2$-vector spaces
\[
  V'(d)=H^1_{\Sigma \cup \cP(d)}(\Q,B[2]) \quad \textup{ and }\quad  Z'(d)=\bigoplus_{v\in \Sigma\cup \cP(d)}\sS_{d,v}^{(r)} .
\]
We also set  $Z(d)=\oplus_{p\in \cP(d)} \sS_{d,p}^{(r)}$.
We write $\cA(d)$ for the subgroup of $H^1(\Q,\F_2)$ generated by the characters $\{\psi_p \colon p\in \cP(d)\}$,
and define $V(d)=\cA(d)\otimes B[2]$, viewed as a subgroup of $H^1_{\Sigma \cup \cP(d)}(\Q,B[2])$.
Note that we have 
\[
  V'(d)=V_0\oplus V(d)\quad \textup{ and }\quad Z'(d)=Z_0\oplus Z(d).
\]

 Summing the
local Tate pairings \eqref{eq:local_tate_pairing} over $v\in \Sigma \cup \cP(d)$ gives a pairing
\begin{equation} \label{eq:Kanes_pairing}
\left \langle~,~\right \rangle_d\colon V'(d)\times Z'(d)\longrightarrow \F_2,
\end{equation}
defined by the formula 
\[
\left \langle x,y\right \rangle =\sum_{v\in \Sigma \cup \cP(d)}\left \langle \res_v(x),y_v\right \rangle_v,
\]
where we write $y=(y_v)_{v\in \Sigma\cup \cP(d)}\in Z'(d)$.
When the choice of $d$ is clear from context, we often omit it from the notation.
Thus, for example, we will typically write  $V$ and $Z$ in place of $V(d)$ and $Z(d)$, respectively, and write $\left \langle ~,~\right \rangle$ in place of  $\left \langle ~,~\right \rangle_d$.

\subsection{Counting $2$-Selmer elements} \label{ssec:parameterising_selmer_elts}
From now until  Section \ref{sec:first_sum_simplification},
we \emph{fix} a $d\in d_0\cF_c$. Write $d=d_0D$, with $D\in \cF_c$.

\begin{lemma} \label{pairing_kernel_selmer}
The left kernel of the pairing \eqref{eq:Kanes_pairing} is  $\Sel_2(B_d/\Q)$.
\end{lemma}

\begin{proof}
For a prime  $p\notin \Sigma\cup \cP(d)$, Lemma \ref{lem:kummer_image_basic} gives $\sS_{d,p}^{(r)}= H^1_{\ur}(\Q_p,B[2])$.
Consequently, $\Sel_2(B_d/\Q)$ is the subspace of  $V'$ consisting of elements
whose restriction to $\oplus_{v\in \Sigma}H^1(\Q_v,B[2])$ lies in $Z'$.
The result now follows from the fact that each summand $\sS_{d,v}^{(r)}$ of $Z'$
is maximal isotropic with respect to the local Tate pairing.
\end{proof}

\begin{lemma} \label{lem:kane_formula}
For all $a\in V'$, we have 
\begin{equation} \label{eq:selmer_formula_basic}
\frac{1}{\#Z'}\sum_{b\in Z'}(-1)^{\left \langle a,b\right \rangle}= \begin{cases} 1, & a\in \Sel_2(B_d/\Q);\\ 0, & \textup{else}.\end{cases}
\end{equation}
\end{lemma}

\begin{proof}
For any $a\in V'$, write $\alpha_a\colon Z'\rightarrow \F_2$ for the homomorphism
$b\mapsto \left \langle a,b\right \rangle$. By Lemma \ref{pairing_kernel_selmer},
one has $a\in \Sel_2(B_d/\Q)$ if and only if $\alpha_a$ is identically $0$. One then has 
\begin{eqnarray*}
\frac{1}{\#Z'}\sum_{b\in Z'}(-1)^{\left \langle a,b\right \rangle}&=&\frac{1}{\#Z'}(\#\ker(\alpha_a)-\#\alpha^{-1}_a(1))
=  \begin{cases} 1, & \alpha_a=0;\\ 0, & \textup{else},\end{cases}
\end{eqnarray*}
as claimed.
\end{proof}

Combining Lemma \ref{lem:kane_formula} with the decompositions
$Z'=Z_0\oplus Z$ and $V'=V_0\oplus V$, we obtain  
\begin{equation} \label{eq:kane_sum_split}
 \#\Sel_2(A_d/\Q)^r= \frac{1}{\#Z_0}\sum_{(v_0,z_0)\in V_0\times Z_0}(-1)^{\left \langle v_0,z_0\right \rangle_\Sigma}\frac{1}{\#Z}\sum_{(v,z)\in V\times Z}(-1)^{\left \langle v_0,z \right \rangle+\left \langle v,z_0\right \rangle+\left \langle v,z\right \rangle}. 
\end{equation}

The following lemma gives a parametrisation for the spaces $V$ and $Z$.

\begin{lemma} \label{lem:paramaterisations}
The associations $\Map(\cP(d),B[2])\rightarrow V$ and $\Map(\cP(d),B[2])\rightarrow Z$ given by 
\[\epsilon \mapsto \sum_{p\in \cP(d)}\psi_p\epsilon(p)\quad \textup{ and }\quad \eta \mapsto \sum_{p\in \cP(d)}\delta_{d,p}^{(r)}(\eta(p))\]
are bijections.
\end{lemma}

\begin{proof}
For the parameterization of $V$, 
having defined $V$ as $\cA\otimes B[2]$, we may uniquely write any element of $V$
in the form $\sum_{p\in \cP(d)}\psi_p  x_p$ for some $x_p\in B[2]$.  

For the parameterization of $Z$, note that by Lemma \ref{lem:kummer_image_basic} (applied to $B$), any element of $Z$ can be
expressed uniquely in the form $\sum_{p\in \cP(d)}\delta_{d,p}^{(r)}(a_p)$ for some elements $a_p$ of $B[2]$.  
\end{proof}

Recall that $d=d_0D$, with $D\in \cF_c$.  Suppose we have a pair $(v,z)\in V\times Z$. By Lemma \ref{lem:paramaterisations},
this corresponds to a pair of functions $(\epsilon,\eta)$ from $\cP(d)$ to $B[2]$. 
For fixed $\bfu \in B[2]^2$ we define the squarefree integer 
\begin{equation}\label{eq:Dab}
D_{\bfu }=\prod_{\substack{p\in \cP(d) \\ (\epsilon(p),\eta(p))=\bfu}}p.
\end{equation}
In this way the pair $(v,z)$ determines a factorisation 
\begin{equation} \label{eq:heath_brown_factorisation}
D=\prod_{\bfu \in B[2]^2}D_{\bfu }
\end{equation}
of $D$ into positive coprime integers.
Conversely, a factorisation of the form \eqref{eq:heath_brown_factorisation} uniquely
determines the functions $\epsilon$ and $\eta$: for each $p\mid D$, define
$(\epsilon(p),\eta(p))$ to be the unique $\bfu \in B[2]^2$ such that $p\mid D_{\bfu }$.
We have thus shown the following. In the statement, recall that for $i\in \{1,2\}$,
we denote by $\pi_i\colon B[2]^2\to B[2]$ the $i$-th projection. 

\begin{lemma} \label{lem:facotrisations_of_d}
  The rule \eqref{eq:Dab} defines a bijection between $V\times Z$ and the collection of factorisations
of $D=d/d_0$ as a product of pairwise coprime positive squarefree integers indexed by
elements $\bfu  \in B[2]^2$, as in \eqref{eq:heath_brown_factorisation}. The element $(v,z)\in V\times Z$ associated to such a factorisation is given by
\begin{equation} \label{eq:explicit_factorisation_to_element}
v=\sum_{\bfu \in B[2]^2}\psi_{D_{\bfu }}\pi_1(\bfu) \quad \textup{ and }\quad z=\sum_{\bfu \in B[2]^2}\sum_{p\mid D_{\bfu
}}\delta_{d,p}^{(r)}(\pi_2(\bfu) ).
\end{equation}
\end{lemma}

\subsection{Aside: the support of an element of $V'\times Z'$}

The following definition will be useful later. 

\begin{definition} \label{def:support}
Let $(v,z)\in V\times Z$, and let $D=\prod_{\bfu\in B[2]^2}D_\bfu$ be the corresponding
factorisation of $D$. Then we define the \textit{support} $T\subseteq B[2]^2$ of $(v,z)$ to be the subset
\[T=\{\bfu \in B[2]^2~~\colon~~ D_\bfu>1\}.\]
Given an element $(v',z')\in V'\times Z'$, written as $(v_0+v,z_0+z)$ according
to the decompositions $V'=V_0\oplus V$ and $Z'=Z_0\oplus Z$ above, we define the
support of $(v',z')$ to be the support of $(v,z)$. 
\end{definition}

Suppose we have $x,y\in B[2]$. Then $\delta_d^{(r)}(x)\in H^1(\Q,B[2])$ lies in
$V'$ and is an element of $\Sel_2(B_d/\Q)$. In particular, it is in the left
kernel of the pairing  \eqref{eq:Kanes_pairing}. Similarly, as explained in the proof of Lemma \ref{support_lemma} below, the element
$\sum_{v\in \Sigma \cup \cP(d)}\delta_{d,v}^{(r)}(y)$ of $Z'$ is in the right
kernel of the pairing. Thus given $v'\in V'$ and $z'\in Z'$, we
can adjust $v'$ by $\delta_d^{(r)}(x)$ and $z'$ by
$\sum_{v\in \Sigma \cup \cP(d)}\delta_{d,v}^{(r)}(y)$ without changing the value
of $\left \langle v',z'\right \rangle$. For this flexibility to be useful later,
it will be necessary to understand how the corresponding factorisation of $D$ changes upon adjusting $v'$ and $z'$ in this way.  

\begin{lemma} \label{support_lemma}
Let $\bfw\in B[2]^2$, and $T\subseteq B[2]^2$. Define the map $\xi\colon V'\times Z'\to V'\times Z'$ by 
\begin{equation}\label{action_by_torsion}
(v',z')\longmapsto \Big(v'+\delta_d^{(r)}(\pi_1(\bfw)), z'+\sum_{v\in \Sigma \cup \cP(d)}\delta_{d,v}^{(r)}(\pi_2(\bfw))\Big).
\end{equation}
\begin{itemize}

\item[(i)] Let $(v',z')=(v_0+v,z_0+z)\in V'\times Z'$ have support $T$, and let $D=\prod_{\bfu}D_{\bfu}$ be the corresponding factorisation of $(v,z)$. Let $(\widetilde{v},\widetilde{z})\in V\times Z$ correspond to the factorisation $D=\prod_{\bfu}D_{\bfu +\bfw}$, and let
\[(\widetilde{v_0},\widetilde{z_0})=\Big(v_0+\delta_{d_0}^{(r)}(\pi_1(\bfw)),z_0+\sum_{v\in \Sigma }\delta_{d,v}^{(r)}(\pi_2(\bfw))\Big)\in V_0\times Z_0.\]
 Then  we have $\xi((v',z'))= (\widetilde{v_0}+\widetilde{v},\widetilde{z_0}+\widetilde{z}).$
 \item[(ii)] With notation as in (i), write $\widetilde{v}'=\widetilde{v_0}+\widetilde{v}$ and $\widetilde{z}'=\widetilde{z_0}+\widetilde{z}$, so that $\xi((v',z'))=(\widetilde{v}',\widetilde{z}')$. Then we have 
$\langle v',z'\rangle = \langle \widetilde{v}',\widetilde{z}'\rangle.$

\item[(iii)] The map $\xi$ induces a bijection from the set of elements of $V'\times Z'$ with support $T$ to the set of elements with support $T+\bfw$. 
\end{itemize}
\end{lemma}

\begin{proof}
(i). From \eqref{eq:explicit_factorisation_to_element} and Lemma \ref{lem:twisted_boundary_map} we have 
\begin{eqnarray*}
v_0+v+\delta_d^{(r)}(\pi_1(\bfw))&=& v_0+\delta_{d_0}^{(r)}(\pi_1(\bfw))+\psi_{D}\pi_1(\bfw)+\sum_{\bfu \in B[2]^2}\psi_{D_{\bfu }}\pi_1(\bfu)\\
&=& v_0+\delta_{d_0}^{(r)}(\pi_1(\bfw))+\sum_{\bfu \in B[2]^2}\psi_{D_{\bfu }}(\pi_1(\bfu)+\pi_1(\bfw))\\
&=& v_0+\delta_{d_0}^{(r)}(\pi_1(\bfw))+\sum_{\bfu \in B[2]^2}\psi_{D_{\bfu +\bfw}}\pi_1(\bfu).
\end{eqnarray*}
Similarly, one computes
\[z_0+z+\sum_{v\in \Sigma \cup \cP(d)}\delta_{d,v}^{(r)}(\pi_2(\bfw))=z_0+\sum_{v\in \Sigma}\delta_{d,v}^{(r)}(\pi_2(\bfw))+\sum_{\bfu \in B[2]^2}\sum_{p\mid D_{\bfu+\bfw}}\delta_{d,p}^{(r)}(\pi_2(\bfu )).\]

(ii).  The element $\delta_{d}^{(r)}(\pi_1(\bfw))\in V'$ lies in $\Sel_2(B_d/\Q)$, hence is  in the left kernel of the pairing $\langle~,~\rangle$. On the other hand, the element $\sum_{v\in \Sigma \cup \cP(d)}\delta_{d,v}^{(r)}(\pi_2(\bfw))\in Z'$ is the image of a global cocycle under the product of restriction maps. By reciprocity, it pairs trivially with all elements of $V'$.

(iii). Follows from (i).  
\end{proof}

Lemma \ref{support_lemma} allows us to relate certain terms  in \eqref{eq:kane_sum_split}. Specifically, it follows from Lemma \ref{support_lemma}  that, for any $T\subseteq B[2]^2$ and $\bfw\in B[2]^2$, we  have 
\begin{eqnarray}  
 (-1)^{\left \langle v_0,z_0\right \rangle_\Sigma}\frac{1}{\#Z}\sum_{\substack{(v,z)\in V\times Z\\(v,z)\textup{ has support }T}}(-1)^{\left \langle v_0,z \right \rangle+\left \langle v,z_0\right \rangle+\left \langle v,z\right \rangle}=& \label{eq:kane_sum_split_support} \\    (-1)^{\left \langle \widetilde{v_0},\widetilde{z_0}\right \rangle_\Sigma}\frac{1}{\#Z}\sum_{\substack{(v,z)\in V\times Z\\(v,z)\textup{ has support }T+\bfw}}(-1)^{\left \langle \widetilde{v_0},z \right \rangle+\left \langle v,\widetilde{z_0}\right \rangle+\left \langle v,z\right \rangle}, \nonumber
\end{eqnarray}
where $(\widetilde{v_0},\widetilde{z_0})$ are as in Lemma \ref{support_lemma} (i). 

One of the main roles of the notation of support comes from the following lemma. Recall that $B[2]=A[2]^r=A[2]\otimes \F_2^r$, so that $V'=H^1_{\Sigma \cup \cP(d)}(\Q,B[2])=H^1_{\Sigma\cup \cP(d)}(\Q,A[2])^r$.   

\begin{lemma} \label{lem:is_injection_criteria}
Let $(v',z')\in V'\times Z'$, let $T$ be its support. Then the $r$-tuple of
elements of $H^1_{\Sigma\cup \cP(d)}(\Q,A[2])$ corresponding to $v'$ consists
of $r$ vectors that are $\F_2$-linearly dependent in  
\[\frac{H^1(\Q,A[2])}{ H^1_\Sigma(\Q,A[2])+\delta_{d}(A[2]) } \]
if and only if there exists $\bfw\in B[2]^2$, and a codimension $1$ subspace
$M\subseteq \F_2^r$, such that $\pi_1(T+\bfw)\subseteq A[2]\otimes M$.
\end{lemma}

\begin{proof}
Let $p_i\colon A[2]^r\rightarrow A[2]$ denote the $i$-th projection. From \eqref{eq:explicit_factorisation_to_element} and the definition of the support of $(v',z')$, we see that the $r$-vectors $\{p_i(v')\}_{i=1}^r$  are linearly dependent in the quotient space appearing in the statement if and only if there exist $\lambda_1,...,\lambda_r\in \F_2$, not all $0$, such that 
\[\sum_{\bfu \in T}\psi_{D_\bfu}\sum_{i=1}^r\lambda_ip_i(\pi_1(\bfu))  \in H^1_\Sigma(\Q,A[2])+\delta_{d}(A[2]).\]
By an analogous argument to that used in the proof of Lemma \ref{support_lemma}, we see that this happens, for given $\lambda_1,...,\lambda_r$, if and only if there is $x \in A[2]$ such that 
\begin{equation} \label{eq:slightly_simplified_injection_condition}
\sum_{\bfu \in T}\psi_{D_\bfu}\sum_{i=1}^r\lambda_i p_i(\pi_1(\bfu)) = \sum_{\bfu \in T}\psi_{D_\bfu}x.
\end{equation}
Since the positive integers $\{D_\bfu\}_{\bfu\in T}$ are all greater than $1$ and coprime,
the characters $\{\psi_{D_\bfu}\}_{\bfu \in T}$ are linearly independent. We thus see
that \eqref{eq:slightly_simplified_injection_condition} is equivalent to the equality
$\sum_{i=1}^r \lambda_i p_i(\pi_1(\bfu))=x$ holding for all $\bfu \in T$. With
$\lambda_1,...,\lambda_r$ and $x$ fixed as in \eqref{eq:slightly_simplified_injection_condition},
define $\phi\colon \F_2^r\rightarrow \F_2$ to be the linear map sending the $i$-th
basis vector to $\lambda_i$. Note that this is surjective since some $\lambda_i$
is assumed non-zero. The induced surjective map
$1\otimes \phi\colon A[2]^r=A[2]\otimes \F_2^r\rightarrow A[2]$ sends $\pi_1(\bfu)$
to $\sum_{i=1}^r\lambda_i p_i(\pi_1(\bfu))$, so we can write this final equality
as $(1\otimes \phi)(\pi_1(\bfu+\bfw))=0$, where $\bfw \in B[2]^2$ is any element
mapping to $x$ under $(1\otimes \phi)\circ \pi_1$. Now
$\ker(1\otimes \phi)=A[2]\otimes \ker(\phi)$ (one inclusion is obvious and the other follows from counting dimensions), so we conclude that the condition \eqref{eq:slightly_simplified_injection_condition} is equivalent to the condition that $\pi_1(T+\bfw)\subseteq A[2]\otimes \ker(\phi)$. From this, one readily obtains the result.
\end{proof}

\subsection{Counting $2$-Selmer elements (cont.)}

We now use the explicit parameterization of $V\times Z$ given in Lemma  \ref{lem:facotrisations_of_d}   to convert \eqref{eq:kane_sum_split} into a sum of Jacobi symbols. Recall that we have fixed $d=d_0D\in d_0\cF_c$.  
Recall also that we denote by $\Phi$ the map  
$B[2]^2\times B[2]^2\rightarrow \F_2$
 sending a pair $(\bfu,\bfv)$ in  $B[2]^2\times  B[2]^2$ to 
$\Phi(\bfu ,\bfv )=e(\pi_1(\bfu) +\pi_1(\bfv), \pi_2(\bfv)).$

\begin{lemma} \label{lem:most_of_the_moment_formula_work}
Let $(v,z)\in V\times Z$, corresponding via Lemma \ref{lem:facotrisations_of_d} to a factorisation of $D$ of the form \eqref{eq:heath_brown_factorisation}. Then we have 
\begin{equation*}
  (-1)^{\left \langle v,z \right \rangle}=\prod_{\bfu }\left(\frac{-d_0}{D_{\bfu }}\right)^{e(\pi_1(\bfu) ,\pi_2(\bfu) )}(-1)^{e\left(\pi_1(\bfu) ,\gamma^{(r)}(\sigma_{D_{\bfu }})\pi_2(\bfu) \right)}~\cdot~\prod_{\bfu  \neq \bfv }\left(\frac{D_\bfu }{D_\bfv }\right)^{\Phi(\bfu,\bfv)},
\end{equation*}
where the first product is taken over all  $\bfu\in B[2]^2$, and the
second product is taken over all pairs of distinct elements $\bfu$ and $\bfv $
of $B[2]^2$.
\end{lemma}

\begin{proof}
From \eqref{eq:explicit_factorisation_to_element} we find that  $\left \langle v,z \right \rangle$ is equal to
\begin{equation*}  
\sum_{\substack{\bfu \in B[2]^2\\ \bfv \in B[2]^2}}\sum_{p\mid D_{\bfv }}\left \langle \psi_{D_{\bfu }}\pi_1(\bfu) ~,~\delta_{d,p}^{(r)}(\pi_2(\bfv) )\right \rangle_p=\sum_{\substack{\bfu \in B[2]^2\\ \bfv \in B[2]^2}}\sum_{p\mid D_{\bfv }}\left \langle \psi_{D_{\bfu }}\pi_1(\bfu) ~,~\delta_p^{(r)}(\pi_2(\bfv)) +\psi_d \cdot \pi_2(\bfv)\right \rangle_p,
\end{equation*}  
where the second equality follows from Lemma \ref{lem:twisted_boundary_map}.   Now $\pi_1(\bfu)\cup \delta_p^{(r)}(\pi_2(\bfv))$ is a quadratic character that is unramified outside $\Sigma$. Consequently, recalling the pairing $H_p$ from Section \ref{sec:local_tate}, we have 
\[\left \langle \psi_{D_{\bfu }} \pi_1(\bfu) ~,~  \delta_p^{(r)}(\pi_2(\bfv)) \right \rangle_p=H_p\big(\psi_{D_\bfu},\pi_1(\bfu)\cup \delta_p^{(r)}(\pi_2(\bfv))\big)=\begin{cases}(\pi_1(\bfu)\cup \delta_p^{(r)}(\pi_2(\bfv)))(\sigma_p),~~&~~p\mid D_{\bfu };\\ 0,~~&~~\textup{otherwise}. \end{cases}\]
Now $(\pi_1(\bfu)\cup \delta_p^{(r)}(\pi_2(\bfv)))(\sigma_p)$ is equal to $e(\pi_1(\bfu) , \gamma^{(r)}(\sigma_p)\pi_2(\bfv))$.
Since the $D_{\bfu }$ are coprime as $\bfu $ varies over elements of $B[2]^2$,
we conclude that we have
\begin{equation*} 
  \sum_{\substack{\bfu \in B[2]^2\\ \bfv \in B[2]^2}}\sum_{p\mid D_{\bfv }} \left \langle \psi_{D_{\bfu }} \pi_1(\bfu) ~,~\delta_p^{(r)}(\pi_2(\bfv))  \right \rangle_p= \sum_{\bfu \in B[2]^2}e\big(\pi_1(\bfu) ~,~\gamma^{(r)}(\sigma_{D_{\bfu }})\pi_2(\bfu) \big).
\end{equation*}
Since $\left \langle ~,~\right \rangle_p=H_p\otimes e$, we have 
\begin{equation*}
(-1)^{\left \langle \psi_{D_{\bfu }} \pi_1(\bfu) ~,~\psi_d   \pi_2(\bfv)  \right \rangle_p}=(-1)^{H_p(\psi_{D_\bfu},\psi_d)e(\pi_1(\bfu) ,\pi_2(\bfv) )}=(D_{\bfu }~,~d)_p^{e(\pi_1(\bfu) ,\pi_2(\bfv) )},
\end{equation*}
where $(~,~)_p$ is the quadratic Hilbert symbol at $p$.
Evaluating the Hilbert symbol in terms of Legendre symbols and recalling that $d=d_0D$, we find
\[(-1)^{\left \langle \psi_{D_{\bfu }} \pi_1(\bfu) ~,~\psi_d   \pi_2(\bfv)  \right \rangle_p} = \begin{cases} \left(\frac{D_{\bfu }}{p}\right)^{e(\pi_1(\bfu) ,\pi_2(\bfv) )},~~&~~p\nmid D_{\bfu };\\ 
\left(\frac{-d_0D /D_{\bfu }}{p}\right)^{e(\pi_1(\bfu) ,\pi_2(\bfv) )},~~&~~p\mid D_{\bfu }.\end{cases}\]
 Consequently, temporarily writing $t={\sum_{\substack{\bfu 
\in B[2]^2 \\\bfv \in B[2]^2}}\sum_{p\mid D_{\bfv }}\left \langle \psi_{D_{\bfu }} \pi_1(\bfu) ~,~\psi_d   \pi_2(\bfv)  \right \rangle_p}$, we have 
 \begin{eqnarray*}
 (-1)^t\quad =
 \prod_{\bfu \neq \bfv }\left(\frac{D_\bfu }{D_\bfv }\right)^{e(\pi_1(\bfu) ,\pi_2(\bfv) )} ~\cdot~\prod_{\bfu }\left(\frac{-d_0}{D_{\bfu }}\right)^{e(\pi_1(\bfu) ,\pi_2(\bfu) )} ~\cdot~\prod_{\bfu }\left(\frac{D/D_{\bfu }}{D_{\bfu }}\right)^{e(\pi_1(\bfu) ,\pi_2(\bfu) )}.
 \end{eqnarray*}
Now note that the rightmost product is equal to $\prod_{\bfu \neq \bfv }  \left(\frac{D_{\bfv }}{D_{\bfu }}\right)^{e(\pi_1(\bfu) ,\pi_2(\bfu) )}$.
\end{proof}

\begin{lemma} \label{eq:minor_fouvry_kluners_lemma}
Let $z\in Z$, written in the form \eqref{eq:explicit_factorisation_to_element}, and let $v_0\in V_0$. Then we have
\[
  (-1)^{\left \langle v_0,z\right \rangle}=\prod_{\bfu \in B[2]^2}(-1)^{\left \langle v_0 ,\psi_{D_\bfu}\pi_2(\bfu)\right \rangle_\Sigma}.
\]
\end{lemma}

\begin{proof}
From \eqref{eq:explicit_factorisation_to_element} and Lemma \ref{lem:twisted_boundary_map} we have 
\begin{equation*} 
\left \langle v_0,z\right \rangle =\sum_{\bfu \in B[2]^2}\sum_{p\mid D_{\bfu }}\left \langle v_0,\delta_{d,p}^{(r)}(\pi_2(\bfu) )\right \rangle_p=\sum_{\bfu \in B[2]^2}\sum_{p\mid D_{\bfu }}\left \langle v_0,\delta_p^{(r)}(\pi_2(\bfu)) +\psi_d \cdot \pi_2(\bfu)\right \rangle_p.
\end{equation*}
  Since both $\delta_p^{(r)}(\pi_2(\bfu))$ and $v_0$ are  
unramified at all $p\mid D_{\bfu }$, it follows that $ \big\langle v_0,\delta_p^{(r)}(\pi_2(\bfu))    \big \rangle_p$ is
zero at all $p\mid D_{\bfu }$.  
Next, since $\psi_d$ and $\psi_{D_\bfu}$ differ by a quadratic
character unramified at all $p\mid D_\bfu$, and since $v_0$ is unramified at all $p\mid D_\bfu$, it follows that 
\[\sum_{p\mid D_\bfu}\left \langle v_0,\psi_{d}\pi_2(\bfu)\right \rangle_p=\sum_{p\mid D_\bfu}\left \langle v_0,\psi_{D_\bfu}\pi_2(\bfu)\right \rangle_p.\]
Since both $v_0$ and $\psi_{D_\bfu}$ are unramified outside $\Sigma \cup \{p\mid D_\bfu\}$, it follows from reciprocity that we have 
\[\sum_{p\mid D_\bfu}\left \langle v_0,\psi_{D_\bfu}\pi_2(\bfu)\right \rangle_p=\left \langle v_0,\psi_{D_\bfu}\pi_2(\bfu)\right \rangle_\Sigma,\]
from which the result follows.
\end{proof}

The remaining term in \eqref{eq:kane_sum_split} is $(-1)^{\left \langle v,z_0\right \rangle}$. With $v$ written in the form \eqref{eq:explicit_factorisation_to_element} above, we have
\begin{equation} \label{eq:bad_local_condtions}
(-1)^{\left \langle v,z_0\right \rangle}=\prod_{\bfu \in B[2]^2}(-1)^{\left \langle \psi_{D_\bfu}\pi_1(\bfu),z_0\right \rangle_\Sigma}.
\end{equation}
To explicate \eqref{eq:bad_local_condtions} further we would have to examine more
closely the Selmer conditions at the `bad' primes $p\mid N$. However, it will 
suffice for now to note that the dependence of the value of \eqref{eq:bad_local_condtions}
on the factors $D_{\bfu }$ is only mild: it only depends on the values of $D_{\bfu }$ modulo $4N$, since
those values determine the classes of $D_{\bfu }$ in $\prod_{p\mid N}\Q_p^{\times}/\Q_{p}^{\times 2}$.

\begin{notation} \label{notation_including_phi}
For each $\bfu  \in B[2]^2$, $v_0 \in V_0$, and $z_0\in Z_0$,   there is a unique quadratic Dirichlet character $\Upsilon_{\bfu}^{(v_0,z_0)}$
modulo $4N$ that, for all squarefree integers $m$ coprime to $N$, satisfies
\[
  \Upsilon_{\bfu}^{ (v_0,z_0)}(m)=(-1)^{\left \langle \psi_{m}\pi_1(\bfu),z_{0}\right \rangle_\Sigma+
   \left \langle v_0 ,\psi_{m}\pi_2(\bfu)\right \rangle_\Sigma}\cdot
  \left(\frac{-d_0}{m}\right)^{e(\pi_1(\bfu) ,\pi_2(\bfu))}\cdot  (-1)^{e\left(\pi_1(\bfu) ,\gamma^{(r)}(\sigma_m)\pi_2(\bfu) \right)}.
\]
\end{notation}

\begin{proposition} \label{prop:explicit_selmer_formula_non-varying}
With the notation above, we have 

\begin{eqnarray*} 
 \#\Sel_2(A_d/\Q)^r= 
       \frac{1}{\#Z_0}\sum_{(v_0,z_0)\in V_0\times Z_0}(-1)^{\left \langle v_0,z_0\right \rangle_\Sigma}\sum_{D=\prod_{\bfu} D_{\bfu
}}\prod_{\bfu} 2^{-2rg\omega(D_{\bfu})}\Upsilon_{\bfu}^{ (v_0,z_0)}(D_{\bfu })\cdot \prod_{\bfu \neq \bfv }\left(\frac{D_\bfu }{D_\bfv }\right)^{\Phi(\bfu ,\bfv )},
 \end{eqnarray*}
\noindent where the inner sum is taken over all factorisations of $D=d/d_0$ as a product
of pairwise coprime positive squarefree integers indexed by $\bfu  \in B[2]^2$,
as  in \eqref{eq:heath_brown_factorisation}, the first product is taken
over all $\bfu \in B[2]^2$, and the second product is taken over all pairs $(\bfu,\bfv)$
of distinct elements of  $B[2]^2$. 
\end{proposition}

\begin{proof}
Note first that $\#Z=\#B[2]^{\omega(D)}=2^{2rg\omega(D)}$. The result now follows by combining
\eqref{eq:kane_sum_split} with \eqref{eq:bad_local_condtions} and  Lemmas
\ref{lem:most_of_the_moment_formula_work} and \ref{eq:minor_fouvry_kluners_lemma}.
\end{proof}

\subsection{A formula for the $r$-th moment of the Selmer group}
We now consider the effect of varying $d\in d_0\cF_c$ in the formula for
$\#\Sel_2(A_d/\Q)^r$ given in Proposition \ref{prop:explicit_selmer_formula_non-varying}.  
 We begin by noting that the subspace
$Z_0=\oplus_{v\in \Sigma} \sS_{d,v}^{(r)}$ of $\oplus_{v\in \Sigma}H^1(\Q_v,B[2])$
is independent of $d$, since the fixed image $b\in \prod_{v\in \Sigma}\Q_v^{\times}/\Q_v^{\times 2}$
of $d$ determines the $\Q_v$-isomorphism class of $B_d$ for each
$v\in \Sigma$ (recall that $b$ is a function of $d_0$ and $c$). 
The Dirichlet characters $\Upsilon_{\bfu}^{ (v_0,z_0)}$, and  the subspace $V_0$ of $H^1(\Q,B[2])$, are  independent of $d$  by definition. It will also be useful later to note that the characters $\Upsilon_{\bfu}^{ (v_0,z_0)}$ are independent of $c$, and that $V_0$ is independent of both $c$ and $d_0$.

\begin{definition} \label{def:the_sums_to_study}
We define $\cD(X)$ to be the set of
all $B[2]^2$-indexed collections $(D_\bfu)_{\bfu\in B[2]^2}$ of positive, pairwise coprime,
squarefree integers $D_\bfu$ satisfying
\begin{equation}
  \prod_{\bfu \in B[2]^2}D_{\bfu } \in \cF(X),
\end{equation}
and define $\cD_c(X)=\{(D_{\bfu})_{\bfu}: \prod_{\bfu}D_{\bfu}\in \cF_c(X)\}$.
For a subset $T\subseteq B[2]^2$, we define $\cD^T(X)$ to be the subset
of $\cD(X)$ consisting of tuples $(D_\bfu)_{\bfu\in B[2]^2}$ for which
$D_\bfu>1$ precisely when $\bfu \in T$ (cf. \Cref{def:support}), and define $\cD_c^T(X) = \cD_c(X)\cap \cD^T(X)$.
Note that $\cD_c(X)$ is the disjoint union of the sets $\cD_c^T(X)$ as $T$ ranges
over all subsets of $B[2]^2$. Finally, for $z_0\in Z_0$,  $v_0\in V_0$, and $T\subseteq B[2]^2$, define
\begin{equation} \label{eq:main_sum_simplified_notation}
\cS^T_{v_0,z_0}(X) =\sum_{(D_\bfu)_\bfu\in \cD_c^T(X)}\prod_{\bfu}2^{-2rg\omega(D_\bfu)}\Upsilon_{\bfu}^{ (v_0,z_0)}(D_\bfu)\cdot \prod_{\bfu\neq \bfv}\left(\frac{D_\bfu}{D_{\bfv}}\right)^{\Phi(\bfu,\bfv)},
\end{equation}
where the products are taken over $B[2]^2$.
\end{definition}

\begin{proposition} \label{prop:selmer_moment_formulae}
We have 
\begin{eqnarray*}
 \sum_{d\in d_0\cF_c(X)}
  \#\Sel_2(A_{d}/\Q)^r=   \frac{1}{\#Z_0}~\sum_{(v_0,z_0)\in V_0\times Z_0} (-1)^{\left \langle v_0,z_0\right \rangle_\Sigma}\sum_{T\subseteq B[2]^2}\cS^T_{v_0,z_0}(X).
\end{eqnarray*}
\end{proposition}

\begin{proof}
Sum the formula in the statement of Proposition \ref{prop:explicit_selmer_formula_non-varying} over $d\in d_0\cF_c(X)$.
\end{proof}

The following lemma records a relationship satisfied by the sums $S_{v_0,z_0}^{T}(X)$. 

\begin{lemma}\label{can_shift_support_lem}
Let $\bfw \in B[2]^2$. Then we have 
\[(-1)^{\langle v_0,z_0\rangle_\Sigma}\cS^T_{v_0,z_0}(X)=(-1)^{\langle \widetilde{v_0},\widetilde{z_0}\rangle}\cS^{T+\bfw}_{\widetilde{v_0},\widetilde{z_0}}(X),\]
where 
\[\widetilde{v_0}=v_0+\delta_{d_0}^{(r)}(\pi_1(\bfw))\quad \textup{ and }\quad \widetilde{z_0}=z_0+\sum_{v\in \Sigma }\delta_{d,v}^{(r)}(\pi_2(\bfw)).\]
\end{lemma}

\begin{proof}
Combine \eqref{eq:kane_sum_split_support} with  \eqref{eq:bad_local_condtions}, Lemmas
\ref{lem:most_of_the_moment_formula_work} and \ref{eq:minor_fouvry_kluners_lemma}, then sum over all $d\in d_0\cF_c(X)$. 
\end{proof}


\subsection{Asymptotics for the sums $\cS^T_{v_0,z_0}$}\label{sec:first_sum_simplification}

A key role in understanding the sum \eqref{eq:main_sum_simplified_notation} is played by the notion of \emph{linked indices}. 

\begin{definition} \label{quad_form_q_defi}
Let $\bfu$ and $\bfv$ be distinct elements of $B[2]^2$.  We say that $\bfu$ and $\bfv$ are \emph{linked} if precisely one of the symbols 
\[\left(\frac{D_\bfu }{D_{\bfv }}\right)\quad \textup{ or }\quad \left(\frac{D_\bfv }{D_{\bfu }}\right)\]
appears with exponent $1$ in the sum \eqref{eq:main_sum_simplified_notation}, equivalently
if 
$\Phi(\bfu ,\bfv )+\Phi(\bfv ,\bfu )=1.$ 
Denote by $q\colon B[2]^2\rightarrow \F_2$ the non-degenerate quadratic form on $B[2]^2$ sending $\bfu $ to 
$q(\bfu )=e(\pi_1(\bfu) ,\pi_2(\bfu) ).$
We have
\begin{eqnarray*}
\Phi(\bfu ,\bfv )+\Phi(\bfv ,\bfu )&=& e(\pi_1(\bfu) +\pi_1(\bfv) , \pi_2(\bfv) )+e(\pi_1(\bfv) +\pi_1(\bfu) ,\pi_2(\bfu) )\\
&=& q(\bfu +\bfv ).
\end{eqnarray*}
In particular, $\bfu $ and $\bfv $ are linked if and only if $q(\bfu +\bfv )=1$.
\end{definition}

We call a subset $T\subseteq B[2]^2$ a \emph{maximal unlinked subset} if no two
elements $\bfu \neq \bfv $ of $T$ are linked, and if $T$ is maximal with
respect to this property. The following lemma is essentially \cite[Lemma 7]{MR1292115} (see also \cite[Lemma 18]{MR2276261}).

\begin{lemma} \label{lem:geom_of_unlinked_indices}
The maximal unlinked subsets of $B[2]^2$  are precisely those of the form
$\bfc+W$ where $\bfc\in B[2]^2$, and $W$ is a maximal isotropic
subspace  with respect to the quadratic form $q$. In particular, every maximal unlinked subset has size $\#B[2]=4^{rg}$. 
\end{lemma}

\begin{proof}
As above, $\bfu, \bfv\in B[2]^2$ are unlinked if and only if $q(\bfu +\bfv )=0$.
From this  it follows that if $T$ is unlinked, then so is $\bfc+T$ for any
$\bfc\in B[2]^2$. Now fix a maximal unlinked subset $W$ of $B[2]^2$ with
$0\in W$, so that $q(\bfu)=0$ for all $\bfu \in W$. Given $\bfu, \bfv\in W$,
we claim that $\bfu+\bfv$ is in $W$. Indeed, given any $\bfw\in W$, since $q$ is a quadratic form we have
\[q(\bfu+\bfv+\bfw)=q(\bfu+\bfv)+q(\bfu+\bfw)+q(\bfv +\bfw)+q(\bfu)+q(\bfv)+q(\bfw)=0.\]
So $\bfu +\bfv $ and $\bfw$ are unlinked for each $\bfw\in W$. Maximality of $W$
then gives the claim. Consequently, $W$ is an isotropic subspace of $B[2]^2$ with
respect to $q$, and it is immediate that such subspaces are necessarily unlinked.
This proves the characterisation of maximal unlinked subsets.

Finally, since the first factor of $B[2]^2$ is visibly isotropic for $q$, we see that
$q$ contains at least one isotropic subspace of dimension $\dim B[2]=\frac{1}{2}\dim B[2]^2$.
By \cite[Corollary 8.12]{quadformsbook} it follows that all maximal isotropic
subspaces of $B[2]^2$ have this dimension, completing the proof. 
\end{proof}

\subsubsection{The error term}

The aim of this subsection is to prove the following bound. 
In what follows we will write $n_L$ for the degree of $L/\Q$.

\begin{proposition} \label{prop:error_term_F_K}  
Suppose that $T\subseteq B[2]^2$ is not maximal unlinked. Then we have
\[|\cS^{T}_{v_0,z_0}(X)| \ll X\log(X)^{\frac{1}{n_L}-1- \frac{1}{n_L \cdot 4^{rg+1}}}.\]
\end{proposition}

To begin the proof, we first divide the range of summation into intervals.
Set $\Delta=1+\log(X)^{-4^{rg}}$ and divide the ranges of the variables $D_{\bfu}$
into intervals of the form $[\Delta^n, \Delta^{n+1})$ for $n=0,1,2,...$. Note that
the only integer in the $n=0$ interval is $1$. For $\bfu \in B[2]^2$ we let
$A_\bfu $ denote a real number\footnote{We have chosen the notation $A_\bfu$
to align with \cite{MR2276261}. We hope that the similarity with the notation
for twists of the abelian variety $A$ will not be confusing.} of the form
$\Delta^n$ with $1\leq \Delta^n\leq X$, let $\bfA=(A_{\bfu})_{\bfu \in B[2]^2}$,
and define 
\begin{equation} \label{eq:sums_with_as}
\cS^{T}_{v_0,z_0}(X,\bfA) =\sum_{\substack{(D_\bfu)_\bfu \in \cD_c^{T}(X)\\ A_{\bfu}\leq D_{\bfu}< \Delta A_{\bfu}}}\prod_{\bfu}2^{-2rg\omega(D_\bfu)}\Upsilon_{\bfu}^{(v_0,z_0)}(D_\bfu)\cdot \prod_{\bfu \neq \bfv}\left(\frac{D_\bfu}{D_{\bfv}}\right)^{\Phi(\bfu,\bfv)}.
\end{equation}
Since for small positive $\alpha$ we have $\log(1+\alpha)\approx \alpha$, it
follows that for large $X$ we have $\log(X)/\log(\Delta)\approx \log(X)^{1+4^{rg}}$,
so the number of choices of $\bfA$ for which the sum in \eqref{eq:sums_with_as}
is non-empty is $O(\log(X)^{(1+4^{rg})\cdot 16^{rg}})$.
Following \cite[Section 5.4]{MR2276261} we split the collection of all such
choices of $\bfA$ into families. 

\textbf{First family: $\prod_\bfu A_\bfu$ large.} The first family is defined by the condition 
\begin{equation} \label{eq:first_family_of_A}
\prod_{\bfu \in B[2]^2}A_\bfu \geq \Delta^{-16^{rg}}X.
\end{equation}
We treat this family as in \cite[Equations (33)-(34)]{MR2276261}. We have 
\begin{eqnarray*}
\bigg|\sum_{\bfA\textup{ satisfies }\eqref{eq:first_family_of_A}}\cS^T_{v_0,z_0}(X,\bfA) 
  \bigg|&\leq &\sum_{\bfA\textup{ satisfies }\eqref{eq:first_family_of_A}}
  \sum_{\substack{(D_\bfu )_\bfu \in \cD_c^{T}(X)\\
  A_{\bfu}\leq D_{\bfu}< \Delta A_{\bfu}}}\prod_{\bfu}4^{-rg\omega(D_\bfu)} \\ 
      &\leq & \sum_{\Delta^{-16^{rg}}X\leq n\leq X}\mu^2(n)\cdot 4^{rg\omega(n)}\\
&\ll & (1-\Delta^{-16^{rg}})X\log(X)^{4^{rg}-1}\\
&\ll& X\log(X)^{-1}.
\end{eqnarray*}
Here for the second inequality we take $n=\prod_\bfu D_\bfu$. The term $4^{rg\omega(n)}$
arises as the product $4^{-rg\omega(n)}\cdot 16^{rg\omega(n)}$ where $16^{rg\omega(n)}$
is the number of ways of writing a squarefree  $n$ as an ordered product
of $16^{rg}$
positive integers. The third inequality is \cite[Lemma 12]{MR2276261}. For the final inequality, we note that 
\begin{eqnarray*}
1-\Delta^{-16^{rg}}=1-(1+\log(X)^{-4^{rg}})^{-16^{rg}}=1-(1-16^{rg}\log(X)^{-4^{rg}}+O(\log(X)^{-2\cdot 4^{rg}})).
\end{eqnarray*}
In summary, we have 
\begin{equation} \label{eq:end_of_first_family}
\bigg|\sum_{\bfA\textup{ satisfies }\eqref{eq:first_family_of_A}}\cS^{T}_{v_0,z_0}(X,\bfA) \bigg| \ll X\log(X)^{-1}.
\end{equation}

Note that if $\bfA$ does not satisfy \eqref{eq:first_family_of_A} then the
condition $\prod_{\bfu}D_{\bfu}\leq X$ is automatically satisfied for any
collection $(D_\bfu)_{\bfu}$ with $A_\bfu \leq D_\bfu < \Delta A_\bfu$.
We thus ignore this condition henceforth.

 \textbf{Second family: $2$ large linked variables.} Set $X^\dagger=\log(X)^{3(1+16^{rg}(1+4^{rg}))}$.
 We now consider $\bfA$ satisfying 
\begin{equation} \label{eq:second_family_defining_conditions}
 \prod_{\bfw\in B[2]^2}A_{\bfw}\leq \Delta^{-16^{rg}}X,\quad \textup{there exist linked }\bfu, \bfv \textup{ such that }A_\bfu , A_\bfv \geq X^\dagger.
\end{equation}
  The argument follows \cite[Equations (40)-(42)]{MR2276261} drawing on results of
  Heath-Brown \cite{MR1347489} exploiting double oscillation of Jacobi symbols.
  Specifically, given $\bfA$ satisfying \eqref{eq:second_family_defining_conditions}
  and corresponding linked variables $\bfu,\bfv$, after swapping $\bfu$ and $\bfv$ if necessary, we have  
\begin{eqnarray*}
|\cS^{T}_{v_0,z_0}(X,\bfA)|\leq  \sum_{\substack{A_\bfw\leq D_\bfw\leq \Delta A_\bfw \\
\bfw \neq \bfu,\bfv}}  \bigg| \sum_{\substack{1\leq D_\bfu \leq \Delta A_\bfu\\
1\leq D_\bfv \leq \Delta A_\bfv \\
D_\bfu D_\bfv \equiv h \pmod{4N}}}f_1(D_\bfu)f_2(D_\bfv)\left(\frac{D_\bfu}{D_{\bfv}}\right)\bigg|,
\end{eqnarray*}
where in the outer sum the $D_{\bfw}$ are coprime to $N$ and to each other, and
in the inner sum the $D_\bfu$ and $D_\bfv$ are odd and coprime integers with no
further constraints than those shown. The terms $h$, $f_1$ and $f_2$  depend on
$(D_{\bfw})_{\bfw \neq \bfu,\bfv}$. Specifically, they are defined as
$h=c(\prod_{\bfw \neq \bfu ,\bfv }D_\bfw )^{-1}\pmod{4N}$,
\begin{eqnarray*}
f_1(D_\bfu)=
  \charfunc_{\substack{D_\bfu\in\cF\\ D_\bfu \geq A_\bfu}}~4^{-rg\omega(D_\bfu)}
   \mu^2\bigg(D_\bfu \!\!\!\prod_{\bfw \neq \bfu,\bfv}\!\!\!\!D_\bfw \bigg)
   \Upsilon_{\bfu}(D_\bfu)\!\!\!\prod_{\bfw\neq\bfu,\bfv}\left(\frac{D_\bfu}{D_\bfw}\right)^{\Phi(\bfu,\bfw)}
   \!\left(\frac{D_\bfw}{D_\bfu}\right)^{\Phi(\bfw,\bfu)},
 \end{eqnarray*}
and $f_2(D_\bfv)$ is defined similarly by swapping $\bfu$ and $\bfv$ in the definition of $f_1$.
The constraint that the collections $(D_\bfw)_{\bfw \in B[2]^2}$ arising lie in
$\cD_c^{T}(X)$ rather than just $\cD_c(X)$ is automatically satisfied thanks to
\eqref{eq:second_family_defining_conditions}. Since the functions $f_1$ and $f_2$
are bounded above by $1$ in absolute value, we can use \cite[Lemma 15]{MR2276261}
to bound the inner sum, provided we can remove the congruence condition on
$D_{\bfu }D_\bfv $. To do this, we split the inner sum into sub-sums by
fixing values $h_1$ and $h_2$ modulo $4N$ whose product is $h$, and insisting
that $D_\bfu \equiv h_1$ and $D_\bfv \equiv h_2$. These conditions can then be
moved into the functions $f_1$ and $f_2$. Now applying  \cite[Lemma 15]{MR2276261}
with $\epsilon = 1/6$ to the inner sum, and summing the result over the remaining variables gives 
\begin{eqnarray*}
  |\cS^{T}(X,\bfA)| &\ll &\Delta^2A_\bfu A_\bfv(X^\dagger)^{-1/3}\cdot \prod_{\bfw \neq \bfu,\bfv}\Delta A_\bfw \leq X(X^\dagger)^{-1/3}.
\end{eqnarray*}
 Summing over the $O(\log(X)^{(1+4^{rg})\cdot 16^{rg}})$-many possibilities for $\bfA$   gives
\begin{equation} \label{end_of_second_family}
  \sum_{\bfA\textup{ satisfies }\eqref{eq:second_family_defining_conditions}}\cS^{T}(X,\bfA) \ll X\log(X)^{-1}.
\end{equation}

\textbf{Third family: $1$ large and $1$ small linked variables.}
We now take $\epsilon>0$ to be chosen later, and define $X^\ddagger=\Delta^l$
to be the least integral power of $\Delta$ that is larger than $\exp(\log(X)^\epsilon)$.
Note that for sufficiently large $X$ we have $X^\ddagger > X^\dagger$. We now consider all $\bfA$ satisfying 
\begin{equation} \label{eq:third_family_definition}
 \textup{Neither \eqref{eq:first_family_of_A} nor \eqref{eq:second_family_defining_conditions} holds,
 and }\exists~ \bfu \neq \bfv  \textup{ linked with } 1< A_\bfv < X^\dagger, ~A_\bfu\geq X^\ddagger.
\end{equation}
Note once again that, given $\bfA$ satisfying   \eqref{eq:third_family_definition},
if we have $(D_\bfw )_{\bfw }\in \cD_c(X)$ with $A_\bfw \leq D_\bfw \leq \Delta A_\bfw$
for each $\bfw$, then $(D_\bfw)_{\bfw }$ is automatically in $\cD_c^{T}(X)$.
 
Taking $\bfA$ satisfying \eqref{eq:third_family_definition}, with corresponding
linked indices $\bfv$ and $\bfu$, in the expression \eqref{eq:sums_with_as} for 
$\cS^{T}_{v_0,z_0}(X,\bfA)$ we group all terms involving $D_\bfu$. By quadratic
reciprocity for Jacobi symbols, if we fix $(D_\bfw)_{\bfw\neq \bfu}$, there is a
real Dirichlet character $\chi_{\bfu, (D_\bfw)_{\bfw \neq \bfu}}$ modulo $4N$ such that 
 \[
 \Upsilon^{(v_0,z_0)}_\bfu (D_\bfu)\prod_{\bfw \neq \bfu}\left(\frac{D_\bfu}{D_{\bfw }}\right)^{\Phi(\bfu ,\bfw )}
 \left(\frac{D_\bfw }{D_{\bfu }}\right)^{\Phi(\bfw ,\bfu )}=\chi_{\bfu, (D_\bfw)_{\bfw\neq \bfu}}(D_\bfu)
\prod_{\bfw \neq \bfu} \left(\frac{D_\bfu }{D_{\bfw }}\right)^{\Phi(\bfu ,\bfw )+\Phi(\bfw ,\bfu )}.
 \]
We then have 
\begin{equation*}
 |\cS^{T}_{v_0,z_0}(X,\bfA)| \leq \sum_{\substack{A_\bfw \leq D_\bfw \leq \Delta A_\bfw \\ \bfw \neq \bfu}}
 \Bigg|\sum_{\substack{A_\bfu \leq D_\bfu \leq \Delta A_\bfu \\ D_\bfu \equiv cd_1^{-1}\hspace{-0.9em}\pmod{4N} \\
  D_\bfu \in \cF}}\mu^2(d_1D_\bfu)4^{-{rg}\omega(D_\bfu)}\chi_{\bfu, (D_\bfw)_{\bfw \neq \bfu}}(D_\bfu) \left(\frac{D_\bfu}{d_2}\right)\Bigg|,
\end{equation*}
where we set $d_1=\prod_{\bfw \neq \bfu}D_\bfw$ and $d_2=\prod_{\bfw \textup{ linked to }\bfu}D_\bfw$,
suppressing the dependency on the collection $(D_\bfw )_{\bfw \neq \bfu }$ to ease notation.
In the first sum we leave implicit the remaining conditions on the $D_\bfw$ (it will suffice
to take them only to be odd and coprime to $N$; relaxing the conditions in this way
only enlarges the sum). Now by the assumption \eqref{eq:third_family_definition} on $\bfA$,
$d_2$ is divisible by at least one odd prime $p$ which is coprime to $N$ and unramified in $L/\Q$.
It follows that for every irreducible character $\chi$ of $L(\zeta_{4N})$, the product
$\chi(-)\cdot  \left(\frac{-}{d_2}\right)$ corresponds to a non-trivial irreducible
character of $L(\zeta_{4N})(\sqrt{d})$ for some $d\ll (\Delta X^\dagger)^{16^{rg}}$.
Expressing the indicator function of the condition $D_\bfu \equiv cd_1^{-1} \pmod{4N}$
as a finite sum of Dirichlet characters modulo $4N$ (with complex coefficients bounded
in absolute value), it follows from \Cref{prop:main_chebotarev_consequence} that, for
every real $Y>0$, we have 
\[
  |\cS^{T}_{v_0,z_0}(X,\bfA)| \ll_Y \frac{X}{\Delta A_\bfu}\cdot \frac{\left(\log(X)+(X^\dagger\right)^{16^{rg}/2})\Delta A_\bfu }{\log(X^\ddagger)^Y}\ll X(X^\dagger)^{16^{rg}/2}\log(X)^{-\epsilon Y}.
\]
Since $X^\dagger$ is a fixed power of $\log(X)$,
summing over the $O(\log(X)^{(1+4^{rg})\cdot 16^{rg}})$ possibilities for $\bfA$ and taking $Y$ large enough compared to $\epsilon$, we find
 \begin{equation} \label{eq:end_of_third_family}
 \sum_{\bfA\textup{ satisfies }\eqref{eq:third_family_definition}}\cS^{T}_{v_0,z_0}(X,\bfA)  \ll X\log(X)^{-1}.
 \end{equation} 
 
 \textbf{Fourth family: fewer than $4^{rg}$ large variables.}
We now consider those $\bfA$ that satisfy:
\begin{equation} \label{eq:4th_family_def}
\textup{we have }A_\bfu  \geq X^\ddagger\textup{ for at most }4^{rg}-1 \textup{ indices }\bfu .
\end{equation}
Here the argument proceeds as in \cite[Lemma 38]{MR2726105}. We trivially have 
\begin{equation} \label{eq:few_large_vars_1}
\bigg| \sum_{\bfA\textup{ satisfies }\eqref{eq:4th_family_def}}\cS^{T}_{v_0,z_0}(X,\bfA)\bigg| \leq \sum_{ (D_\bfu )_\bfu  } \prod_{\bfu }4^{-rg\omega(D_\bfu )},
\end{equation}
where the right hand sum is taken over all collections $(D_\bfu )_{\bfu \in B[2]^2}$
of pairwise coprime integers satisfying $\prod_{\bfu }D_\bfu \leq X$, with each
$D_\bfu $ in $\cF$, and  at most $4^{rg}-1$ of the $D_\bfu$ larger than 
$X^\ddagger$ in absolute value. We split this sum up according to the number $t$
of those $D_\bfu$ that are $> X^\ddagger$, let $n$ be the product of these
`large' variables, and let $m$ be the product of the remaining `small variables'. 
This gives 
\begin{equation*}
\textup{RHS of }\eqref{eq:few_large_vars_1} \leq \sum_{t=0}^{4^{rg}-1}\sum_{m\leq (X^\ddagger)^{16^{rg}-t}}\mu^2(m)(4^{rg}-t4^{-rg})^{\omega(m)}\sum_{\substack{n\leq X/m \\ n\in \cF}}(t\cdot 4^{-rg})^{\omega(n)}.
\end{equation*}
Here, similarly to the treatment of the first family,
the term $(4^{rg}-t4^{-rg})^{\omega(m)}$ arises as the product $(16^{rg}-t)^{\omega(m)}\cdot 4^{-rg\omega(m)}$,
the first factor being the number of ways to write a positive squarefree integer $m$ as an ordered product of $16^{rg}-t$
positive integers. The term $(t\cdot 4^{-rg})^{\omega(n)}$ arises similarly.
Applying \Cref{lem:mertens:chebotarev_sum} to the innermost sum gives
 \begin{eqnarray*}
\textup{RHS of }\eqref{eq:few_large_vars_1} &\ll &\sum_{t=0}^{4^{rg}-1}\sum_{m\leq (X^\ddagger)^{16^{rg}-t}}\mu^2(m)(4^{rg}-t4^{-rg})^{\omega(m)}\cdot(X/m)\log(X)^{t\cdot 4^{-rg}\cdot n_L^{-1}-1}\\
&\ll & X\left(\sum_{t=0}^{4^{rg}-1}\log(X)^{t\cdot 4^{-rg}n_L^{-1}-1}\right)\cdot \left(\sum_{m\leq (X^\ddagger)^{16^{rg}}}\mu^2(m)\frac{4^{rg\omega(m)}}{m}\right)\\
&\ll & X\log(X)^{\frac{1}{n_L}-1-\frac{1}{n_L\cdot 4^{rg}}}\cdot \log(X^\ddagger)^{4^{rg}}\ll X\log(X)^{\frac{1}{n_L}-1-\frac{1}{n_L \cdot 4^{rg}}+4^{rg}\epsilon}
\end{eqnarray*} 
 where for the second last inequality we use Merten's formula (see e.g. \cite[Theorem 4.2]{KMS2021}). Taking $\epsilon>0$ sufficiently small, we finally arrive at 
 \begin{equation} \label{eq:few_large_vars_2}
\bigg| \sum_{\bfA\textup{ satisfies }\eqref{eq:4th_family_def}}\cS^{T}_{v_0,z_0}(X,\bfA)\bigg| \ll X\log(X)^{\frac{1}{n_L}-1- \frac{1}{n_L \cdot 4^{rg+1}}}.
\end{equation}

\textbf{Completion of the proof of \Cref{prop:error_term_F_K}.} Combining \eqref{eq:end_of_first_family},
\eqref{end_of_second_family}  \eqref{eq:end_of_third_family}, and \eqref{eq:few_large_vars_2},
it only remains to treat the $\bfA$ that fall outside each of the $4$ families above.
Take $X$ sufficiently large so that $X^\ddagger>X^\dagger$, and consider such an
exceptional $\bfA$. Since $\bfA$ lies outside the $4$-th family, there are at least
$4^{rg}$ indices $\bfu$ such that $A_\bfu\geq X^\ddagger$. Since $\bfA$ lies outside
the $1$st and $2$nd families, the collection of all these $\bfu $ form an
unlinked subset. Thus by \Cref{lem:geom_of_unlinked_indices} the collection of
such indices $\bfu$ form a maximal unlinked subset of $B[2]^2$. Call this
subset $\cU$. Finally, since $\bfA$ lies outside the $3$rd family, we see that
$A_\bfv=1$ for each $\bfv \notin \cU$. But then from the definition of
$\cD_c^{T}(X)$ there can be no element $(D_\bfw)_{\bfw \in B[2]^2}$ in $\cD_c^{T}(X)$
with $A_\bfw \leq D_\bfw \leq \Delta A_\bfw$ for each $\bfw $. This completes the proof of \Cref{prop:error_term_F_K}.

\subsubsection{The main term}
We now turn to the sums
$\cS^{T}_{v_0,z_0}(X)$ where $T\subseteq B[2]^2$ is maximal unlinked. Recall that we have 
\begin{equation} \label{recalling_ST}
\cS^{T}_{v_0,z_0}(X)=\sum_{(D_\bfu)_\bfu \in \cD_c^{T}(X)}\prod_{\bfu}2^{-2rg\omega(D_\bfu)}\Upsilon_{\bfu}^{(v_0,z_0)}(D_\bfu)\cdot
  \prod_{\bfu \neq \bfv }\left(\frac{D_\bfu }{D_{\bfv }}\right)^{\Phi(\bfu ,\bfv )},
\end{equation}
where  $\cD_c^{T}(X)$ is the subset of elements $(D_\bfu)_{\bfu\in B[2]^2}$ of $\cD_c(X)$ for which $D_\bfu>1$ precisely when $\bfu \in T$. Recall  from \Cref{lem:geom_of_unlinked_indices} that any maximal unlinked subset $T$ of $B[2]^2$ has the form $\bfc+T_0$ where $\bfc \in B[2]^2$ and $T_0$ is a maximal isotropic subspace with respect to the quadratic form $q$ of \Cref{quad_form_q_defi}. By Lemma \ref{can_shift_support_lem}, it suffices to understand the sums $\cS^{T}_{v_0,z_0}(X)$ where $T$ is a maximal isotropic subspace.   Recall also from the beginning of the section that the condition that all prime factors of elements of $\cF$ split completely in $L/\Q$ forces each element of $\cF$ to reduce
modulo $4N$ into a specific subgroup $S\leq (\Z/4N\Z)^{\times}$, and that we have chosen $c\in S$.  

\begin{notation} \label{notat:the_combinatorial_sum_notat}
Given $h\in S$, we define $\alpha(h)\in \{0,1\}$ so that $h\equiv (-1)^{\alpha(h)}\pmod 4$. 
Note that both $\alpha$ and all of the characters $\Upsilon_\bfu^{(v_0,z_0)}$, $\bfu \in B[2]^2$, factor through $\tfrac{ (\Z/4N\Z)^{\times}}{(\Z/4N\Z)^{\times 2}}$. Consequently, we can unambiguously evaluate them on elements of $\ImageModSquares$.

For $W\subseteq B[2]^2$ maximal isotropic, define 
\begin{equation} \label{defi_Y_sum}
 \sY_{v_0,z_0}(W)=\frac{1}{\#\ImageModSquares^{4^{rg}}}\sum_{\substack{(h_\bfu )_{\bfu }\in \ImageModSquares^W }}
 \prod_{\bfu \in W}\Upsilon^{(v_0,z_0)}_\bfu(h_\bfu)\cdot
 \prod_{\{\bfu,\bfv\}} (-1)^{\alpha(h_\bfu)\alpha(h_\bfv)\Phi(\bfu,\bfv)},
\end{equation}
and denote by $\cD_c^{\subseteq W}(X)$ the subset of elements
$(D_\bfu)_{\bfu \in B[2]^2}$ of $\cD_c(X)$ for which $D_\bfu=1$ for all $\bfu \notin W$,
i.e. the union $\bigcup_{T\subseteq W}\cD_c^T(X)$.
\end{notation}

\begin{proposition} \label{prop:main_term_sums_analysis}
Let $W\subseteq B[2]^2$ be a maximal isotropic subspace. Then we have
\begin{equation*}
\cS^{W}_{v_0,z_0}(X)=  \frac{\sY_{v_0,z_0}(W)}{\#S}\cdot  \#\cF(X)+
O\left(X\log(X)^{\frac{1}{n_L}-1- \frac{1}{n_L \cdot 4^{rg+1}}}\right) .
\end{equation*}
\end{proposition}

\begin{proof}  
Write $\cS^{\subseteq W}_{v_0,z_0}(X)=\sum_{T\subseteq W}\cS_{v_0,z_0}^{T}(X)$, and note that we have
\[\cS^{\subseteq W}_{v_0,z_0}(X) =\sum_{(D_\bfu)_\bfu \in \cD_c^{\subseteq W}(X)}\prod_{\bfu}2^{-2rg\omega(D_\bfu)}\Upsilon^{(v_0,z_0)}_{\bfu}(D_\bfu)\cdot
  \prod_{\bfu \neq \bfv }\left(\frac{D_\bfu }{D_{\bfv }}\right)^{\Phi(\bfu ,\bfv )}.\]
It suffices to prove the result for $\cS^{\subseteq W}_{v_0,z_0}(X)$ in place
of $\cS^{W}_{v_0,z_0}(X)$, since by \Cref{prop:error_term_F_K} these two quantities agree up to an acceptable error. 

Projecting onto $(D_\bfu)_{\bfu\in W}$,
we identify $\cD_c^{\subseteq W}(X)$ with the set of tuples $(D_\bfu )_{\bfu \in W}$
of pairwise coprime elements of $\cF$ such that $\prod_{\bfu \in W}D_\bfu \leq X$
and $\prod_{\bfu \in W}D_\bfu \equiv c \pmod{4N}$. Since $W$ is unlinked, given
$\bfu \neq \bfv$ in $W$ we have $\Phi(\bfu, \bfv)=\Phi(\bfv,\bfu)$. In particular,
by quadratic reciprocity for Jacobi symbols we have 
\[
\left(\frac{D_\bfu}{D_\bfv}\right)^{\Phi(\bfu,\bfv)}\left(\frac{D_\bfv}{D_\bfu }\right)^{\Phi(\bfv,\bfu)}=
  (-1)^{ \Phi(\bfu ,\bfv )\cdot \frac{D_\bfu-1}{2} \cdot \frac{D_\bfv -1}{2}}.
\]

 Conditioning on the value of each $D_\bfu$
modulo $4N$, and recalling that each $\Upsilon_\bfu^{(v_0,z_0)}$ is a Dirichlet character modulo $4N$, we find 
\begin{equation} \label{eq:splitting_out_main_term}
\cS^{\subseteq W}_{v_0,z_0}(X) 
 = \!\! \sum_{\substack{(h_\bfu )_{\bfu }\in S^W\\\prod_{\bfu}h_\bfu=c}}
  \prod_{\bfu \in W}\Upsilon^{(v_0,z_0)}_\bfu (h_\bfu)\cdot
  \prod_{\{\bfu,\bfv\}} (-1)^{\alpha(h_\bfu)\alpha(h_\bfv)\Phi(\bfu,\bfv)}\!\!\!\!\!
  \sum_{\substack{(D_\bfu)_{\bfu}\in \cD_c^{\subseteq W}(X)\\ D_\bfu \equiv h_\bfu\hspace{-0.8em} \pmod{4N}}}
  \prod_{\bfu \in W}4^{-rg\omega(D_\bfu)},
\end{equation}
where the second product is taken over all unordered pairs $\{\bfu ,\bfv \}$ of (distinct) elements of $W$. 

\Cref{eq:equidistribution_input} below, whose
proof is similar to the treatment of the $3$rd family in the proof of \Cref{prop:error_term_F_K},
gives asymptotics for the inner sum in \eqref{eq:splitting_out_main_term}.  Specifically it shows that, up to an acceptable error, its value is  $\#\cF(X) \cdot |S|^{-4^{rg}}$. Thus, again up to an acceptable error, we have 
\begin{equation}\label{just_c_to_remove}
\cS^{\subseteq W}_{v_0,z_0}(X) =\frac{\#\cF(X)}{\#S^{4^{rg}}}\cdot \!\! \sum_{\substack{(h_\bfu )_{\bfu }\in S^W\\\prod_{\bfu}h_\bfu=c}}
  \prod_{\bfu \in W}\Upsilon^{(v_0,z_0)}_\bfu (h_\bfu)\cdot
  \prod_{\{\bfu,\bfv\}} (-1)^{\alpha(h_\bfu)\alpha(h_\bfv)\Phi(\bfu,\bfv) }.
  \end{equation}
Now, consulting Notation \ref{notation_including_phi} we see that, when $\bfu=0$, the character $\Upsilon_\bfu^{(v_0,z_0)}$ is trivial. Further, if either $\bfu$ or $\bfv$ is zero, then  $\Phi(\bfu,\bfv)=0$. In particular, given any tuple $(h_\bfu)_\bfu \in S^{W}$, we can change the value  at $\bfu=0$ by an arbitrary element of $S$, without changing the value of 
 \[\prod_{\bfu \in W}\Upsilon^{(v_0,z_0)}_\bfu (h_\bfu)\cdot
  \prod_{\{\bfu,\bfv\}} (-1)^{\alpha(h_\bfu)\alpha(h_\bfv)\Phi(\bfu,\bfv)}.\]
  Consequently, the sum on the right hand side of \eqref{just_c_to_remove} is independent of the choice of $c\in S$. In particular, we may remove the constraint $\prod_{\bfu}h_\bfu=c$ at the cost of dividing by $|S|$.  Finally, since $\alpha$, and the functions $\Upsilon^{(v_0,z_0)}_\bfu$, $\bfu \in U$, depend only on $h\in S$ via the image of $h$ in $ \ImageModSquares$, we can replace $S$ by $\ImageModSquares$ in the expression 
  \[\frac{1}{\#S^{4^{rg}}}\sum_{ (h_\bfu )_{\bfu }\in S^W  }
  \prod_{\bfu \in W}\Upsilon^{(v_0,z_0)}_\bfu (h_\bfu)\cdot
  \prod_{\{\bfu,\bfv\}} (-1)^{\alpha(h_\bfu)\alpha(h_\bfv)\Phi(\bfu,\bfv) }\]
  without changing its value, hence obtaining the claimed expression.
\end{proof}

 \begin{lemma} \label{eq:equidistribution_input}
Let $W\subseteq B[2]^2$ be a maximal isotropic subspace. Then
for all elements $(h_\bfu )_{\bfu \in W}\in S^{W}$ satisfying $\prod_{\bfu \in W}h_{\bfu} = c$, we have 
 \[
 \sum_{\substack{(D_\bfu )_{\bfu }\in \cD_c^{\subseteq W}(X)\\ D_\bfu \equiv h_\bfu \hspace{-0.8em}\pmod{4N}}}
 \prod_{\bfu \in W}4^{-rg\omega(D_\bfu)}=\frac{\#\cF(X)}{\#S^{4^{rg}}}+O\left(X\log(X)^{\frac{1}{n_L}-1- \frac{1}{n_L \cdot 4^{rg+1}}}\right).\]
 \end{lemma}
 
\begin{proof}
First, note that since we have $\prod_{\bfu \in W}h_{\bfu} = c$, we may replace $\cD_c^{\subseteq W}(X)$ by
$\cD^{\subseteq W}(X)$ without changing the sum.
Recall that $S$ is the image of $\Gal(\Q(\zeta_{4N})/F)$ in $(\Z/4N\Z)^\times$, where $F=L\cap \Q(\zeta_{4N})$.
Let $\chi_1,...,\chi_n$ be the irreducible characters of $(\Z/4N\Z)^\times$,
chosen so that the first $k=\frac{\varphi(4N)}{|S|}$ characters are precisely
those that lift from the quotient $(\Z/4N\Z)^\times/S$. Here $\varphi$ is
Euler's phi function. For any $h_\bfu \in S$ we can write the indicator
function of $h_\bfu $ as the sum
 \begin{equation*}
\charfunc_{h_\bfu }=\frac{1}{\varphi(4N)} \sum_{i=1}^k\chi_i+\frac{1}{\varphi(4N)} \sum_{i=k+1}^n\chi_i(h_\bfu ^{-1})\chi_i.
 \end{equation*}
Since each $p\in \cP_L$ splits completely in $L$ (so in particular in $F$),
we have, for every $D_{\bfu}\in \cF$ and for every $i\in \{1,\ldots,k\}$, that
$\chi_i(D_\bfu )=1$. Thus we have 
 \begin{equation} \label{eq:indicator_identity}
 \charfunc_{h_\bfu }(D_\bfu )=\frac{1}{\#S}+\frac{1}{\varphi(4N)}\sum_{i=k+1}^n\chi_i(h_\bfu ^{-1})\chi_i(D_\bfu ).
 \end{equation}
In the sum in the statement of the lemma, we drop the conditions $D_{\bfu }\equiv h_{\bfu }\pmod{4N}$
at the expense of multiplying the summand by $\prod_{\bfu \in W}\charfunc_{h_\bfu }(D_\bfu)$.
Expanding this product using \eqref{eq:indicator_identity} leaves the sum
\[
\frac{1}{\#S^{4^{rg}}}\sum_{\substack{(D_\bfu )_{\bfu }\in \cD^{\subseteq W}(X)}}
  \prod_{\bfu\in W}4^{-rg\omega(D_\bfu)}=\frac{1}{\#S^{4^{rg}}}\cdot \#\cF(X),
\]
along with $O(1)$-many sums of the shape
\begin{equation} \label{eq:the_error_term_sums_non_triv_chi}
 \textup{constant}\cdot\sum_{\substack{(D_\bfu )_{\bfu }\in \cD^{\subseteq W}(X)}}
 \bigg(\prod_{u\in W}4^{-rg\omega(D_\bfu)}\bigg)\cdot\bigg(\prod_{\bfu \in U}\chi_{i_\bfu}(D_\bfu)\bigg)
\end{equation}
for some non-empty subsets $U\subseteq W$ and indices $k+1\leq i_\bfu \leq n$.
The argument used to treat the $4$th family in the proof of \Cref{prop:error_term_F_K}
shows that, at the expense of an acceptable error term, in
\eqref{eq:the_error_term_sums_non_triv_chi} we may restrict the range of summation
to elements $(D_\bfu)_{\bfu}$ such that for each $\bfu $ we have
$D_{\bfu}\geq X^\ddagger=\exp(\log(X)^\epsilon)$ for some sufficiently small $\epsilon>0$.
Let $\bfu $ be an index lying in $U$. Then we have 
\begin{equation}\label{eq:3rd_family_analogue_sum}
\left| \eqref{eq:the_error_term_sums_non_triv_chi} \right| \ll
  \sum_{\substack{(D_\bfv)_{\bfv \in W\setminus \{\bfu\}}\\
  \prod_{\bfv }D_\bfv \leq X}}\mu^2(m) 4^{-rg\omega(m)}
  \bigg| \sum_{\substack{X^\ddagger \leq D_\bfu \leq X/m\\ D_\bfu \in \cF}}
  \mu^2(mD_\bfu)4^{-rg\omega(D_\bfu)}\chi_{i_\bfu}(D_\bfu)\bigg|,
\end{equation}
where we have set $m=\prod_{\bfv\neq\bfu}D_\bfv$. Now from the
definition of $S$ as the subgroup of $\Gal(\Q(\zeta_{4N})/\Q)$
consisting of elements that act trivially on $L\cap \Q(\zeta_{4N})$, we see that
for each $k+1\leq i\leq n$  and every irreducible character $\chi$ of $\Gal(L/\Q)$,
the product $\chi_i\chi$ is a non-trivial irreducible character of
$\Gal(L(\zeta_{4N})/\Q).$ Thus applying \Cref{disgusting_sum_lemma_equation},
for all $B>0$ we have 
\begin{eqnarray*}
 \textup{RHS of \eqref{eq:3rd_family_analogue_sum} }& \ll_B & 
 \frac{X\log(X)}{\log(X^\ddagger)^B}\sum_{m\leq X}\frac{1}{m}\ll \frac{X\log(X)^2}{\log(X)^{\epsilon B} }.
\end{eqnarray*}  
Choosing $B$ sufficiently large in terms of $\epsilon$ we see that each of the
sums \eqref{eq:the_error_term_sums_non_triv_chi} go into the error term, giving the result.
\end{proof}

\subsection{A combinatorial formula for the moments}

We begin with a few simplifications of the summands in  \eqref{defi_Y_sum}. 

\begin{notation} \label{notat:image_mod_sq_interpretation}
The natural map $\cF\to \oplus_{v\in \Sigma}H^1(\Q_v,\F_2)$
sending $m$ to $(\res_v\psi_m)_v$ factors through an injection $\ImageModSquares\hookrightarrow \oplus_{v\in \Sigma}H^1(\Q_v,\F_2)$.
More intrinsically, this is the restriction to $\ImageModSquares$ of the natural map
$\tfrac{(\Z/4N\Z)^{\times}}{(\Z/4N\Z)^{\times 2}} \cong \prod_{p\mid N}\Z_p^{\times}/\Z_p^{\times 2} \hookrightarrow \oplus_{v\in \Sigma}H^1(\Q_v,\F_2)$,
where the second map takes a tuple to the corresponding tuple of local quadratic characters, along with the trivial character at the real place.
Via this map, we will view $\ImageModSquares$ as a subgroup of  $\bigoplus_{v\in \Sigma}H^1(\Q_v,\F_2)$,
and $\ImageModSquares \otimes B[2]$ as a subgroup of $\oplus_{v\in \Sigma}H^1(\Q_v,B[2])$.
Given $u\in B[2]$, and $m\in \cF$ lifting $h\in \ImageModSquares$, this latter identification
takes the class of $h\otimes u$ in $\ImageModSquares\otimes B[2]$ to the tuple $(\res_v \psi_mu)_{v\in \Sigma}$.

Let $W\subseteq B[2]^2$ be a maximal isotropic subspace for $q$. Define 
\[\eta\colon\ImageModSquares^W\longrightarrow \ImageModSquares\otimes W\]
sending $(h_w)_w$ to $\sum_{w\in W}h_w\otimes w$.
Note that for $i=1,2$, the projection $\pi_i\colon B[2]^2\to B[2]$ induces a map
$\ImageModSquares\otimes W\rightarrow \ImageModSquares\otimes B[2]$, which we denote $\pi_i$ also. Via the identification above, we view these maps as being valued in $\oplus_{v\in \Sigma}H^1(\Q_v,B[2])$.

The natural map restriction map $V_0=H^1_{\Sigma}(\Q_v,B[2])\to \oplus_{v\in \Sigma}H^1(\Q_v,B[2])$ is injective. In what follows, as with $\ImageModSquares \otimes B[2]$, we will typically identify $V_0$ with its image in  $\oplus_{v\in \Sigma}H^1(\Q_v,B[2])$.
\end{notation}

 \begin{lemma} \label{preliminary_evaluation_of_sum}
We have 
\begin{equation} \label{SWsum}
\sY_{v_0,z_0}(W)=\frac{1}{|\ImageModSquares|^{4^{rg}}}\sum_{t\in \ImageModSquares\otimes W}(-1)^{\left \langle v_0,\pi_2(t)\right \rangle_\Sigma+
\left \langle \pi_1(t),z_0\right \rangle_\Sigma}\sum_{\substack{(h_w)_w\in \ImageModSquares^W\\ \eta((h_w)_w)=t}}(-1)^{\Xi_W((h_w)_w)},
\end{equation}
where
\[\Xi_W((h_w)_w)=\sum_{\{w,w'\}\subseteq W}\alpha(h_w)\alpha(h_{w'})e(\pi_1(w),\pi_2(w'))+\sum_{w\in W}e(\pi_1(w),\gamma(\sigma_{h_w})\pi_2(w)).\]
 \end{lemma}
 
 \begin{proof}
Since $W$ is isotropic for $q$, we have $e_2(\pi_1(\bfu),\pi_2(\bfu))=0$
for all $\bfu\in W$. 
 It now follows  straight from the definition of $\sY_{v_0,z_0}(W)$ that 
\begin{align*}
  \sY_{v_0,z_0}(W)=&\frac{1}{|\ImageModSquares|^{4^{rg}}}\sum_{(h_w)\in \ImageModSquares^W }(-1)^{\left \langle v_0,\sum_{w\in W}h_w\pi_2(w)\right \rangle_\Sigma+\left \langle \sum_{w\in W}h_w\pi_1(w),z_0\right \rangle_\Sigma+\Xi_W((h_w)_w)}.
\end{align*}
The result follows from noting that the first two terms in the exponent depend only on the image of the tuple $(h_w)_w$ under $\eta$. 
\end{proof}

\begin{definition}  \label{contrib_main_W_sum}
Define
\begin{equation*}
\sY(W)=\frac{1}{|Z_0|}\sum_{(v_0,z_0)\in V_0\times Z_0}(-1)^{\left \langle v_0,z_0\right \rangle_\Sigma}\sY_{v_0,z_0}(W). 
\end{equation*} 
\end{definition}

 \begin{theorem} \label{thm:main_selmer_asymptotic_moment_dist}
 Let $L/\Q$ be a finite Galois extension, let $c\in S\subseteq \left(\Z/4N\Z\right)^{\times}$ and let $d_0$ be a (positive or negative)
 divisor of $N$.  Then we have 
\begin{eqnarray*}
\lim_{X\rightarrow \infty} \frac{1}{\# \cF_{c}(X)}\sum_{ d\in d_0\cF_{c}(X) }
 \frac{ \#\Sel_2(A_{d}/\Q)^r}{\#A[2]^r}=   \sum_{W\subseteq B[2]^2\textup{ max. iso.}}\sY(W).
\end{eqnarray*}
 \end{theorem}
 
 \begin{proof}
 Combining  Propositions \ref{prop:selmer_moment_formulae} and \ref{prop:error_term_F_K} with Proposition \ref{prop:asymptotic_cF} (the latter giving asymptotics for $\# \cF_{c}(X)$), gives 
 \begin{eqnarray*}
 \lefteqn{\frac{1}{\#  \cF_{c}(X)}\sum_{ d\in  d_0\cF_{c}(X)}
  \#\Sel_2(A_{d}/\Q)^r =} \\
& & \frac{1}{\#  \cF_{c}(X)}\cdot \frac{1}{\left|Z_0\right|}~\sum_{(v_0,z_0)\in V_0\times Z_0} (-1)^{\left \langle v_0,z_0\right \rangle_\Sigma}\sum_{\substack{T\subseteq B[2]^2\\T\textup{ max. unlinked}}} \cS^T_{v_0,z_0}(X) \quad +o(1).
\end{eqnarray*}
 By Lemma \ref{lem:geom_of_unlinked_indices}, the maximal unlinked subsets of $B[2]^2$ are precisely those of the form $W+\bfc$, where $\bfc \in B[2]^2$ and $W$ is a maximal isotropic subspace for the quadratic form $q$. Note that each maximal isotropic subspace for $q$, having size $\#B[2]$, has $\#B[2]$ distinct cosets in $B[2]^2$. By Lemma \ref{can_shift_support_lem} we may thus replace the condition that $T$ be maximal unlinked in the above displayed formula with the condition that $T$ be maximal isotropic for $q$, at the expense of multiplying the right hand side by $\#B[2]=\#A[2]^r$. Now Proposition \ref{prop:main_term_sums_analysis} and a further application of Proposition \ref{prop:asymptotic_cF} gives the result.
 \end{proof}
 
 \begin{remark} \label{moments_bounded_remark}
Theorem \ref{thm:main_selmer_asymptotic_moment_dist} shows that 
 the limit 
\[
 \lim_{X\rightarrow \infty} \frac{1}{\#  \cF_{c}(X)}\sum_{ d\in  d_0\cF_{c}(X)}
\#\Sel_2(A_{d}/\Q)^r
\]
exists and is finite.
 \end{remark}
 
 We end this section by giving a useful variant of Theorem \ref{thm:main_selmer_asymptotic_moment_dist}. 
 Note that we have an isomorphism 
$\Sel_2(A_d/\Q)^r\cong\Hom(\F_2^r,\Sel_2(A_d/\Q)),$
the map from right to left sending a homomorphism $\phi$ to $(\phi(f_i))_{i=1}^r$, where $f_i$ denotes the $i$-th standard basis vector of $\F_2^r$. In particular, we can replace the quantity  $\#\Sel_2(A_{d_0D}/\Q)^r$ in Proposition \ref{prop:selmer_moment_formulae} with $\#\Hom(\F_2^r,\Sel_2(A_d/\Q))$. Using Lemma \ref{lem:is_injection_criteria}, we can prove a version of Theorem \ref{thm:main_selmer_asymptotic_moment_dist} where we instead count  (a slight variant of) injective homomorphisms into the Selmer group. The resulting sums arising this way will turn out to be easier to work with. Before giving this result however, we need to define a certain `systematic subspace' of the Selmer group. 

\subsection{The systematic subspace}\label{sec:syst_subspace} 

\begin{definition}\label{def:syst_subspace}
Let $L/\Q$, $c$ and $d_0$ be as above. Recall that the pair $(d_0,c)$ determines a unique class $b=(b_v)_v\in \idelesmodsquaresset{\Q}{\Sigma}$. We define the \emph{systematic subspace}
$\cS_{b}$ to be the subgroup of $H^1(L/\Q,A[2])$ consisting of elements
that  satisfy the Selmer conditions associated to $b_v$ at all places $v\in \Sigma$. That is, we set  
\[\cS_{b}= \ker\Big(H^1\big(L/\Q,A[2]\big)\longrightarrow  \bigoplus_{v\in \Sigma}  H^1(\Q_v,A[2])/\sS_{b_v,v} \Big).\]
\end{definition}

\begin{remark}
We have $\cS_{b}\subseteq \Sel_2(A_{d}/\Q)$ for every $d\in d_0\cF_c$. Indeed, the Selmer conditions at $v\in \Sigma$ are satisfied by assumption, while the condition that all elements of $\cS_{b}$ have trivial restriction to $L$ ensures that, for all primes $p\notin \Sigma$ with $p\mid d$, we have  $\res_p(\cS_{b})=0$. 

Note that any element of $H^1(L/\Q,A[2])$ is trivial on restriction to the maximal multi-quadratic subextension of $L/\Q$. In particular, since $\Sigma$ contains all primes that ramify in $L$, we see that $\cS_b$ is naturally a subspace of $H^1(K_\Sigma\cap L/\Q,A[2])$. Thus, $\cS_b^r$ is naturally a subspace of $V_0=H^1(K_\Sigma/\Q,B[2])$.
\end{remark}

 Recall that, per Notation \ref{notat:image_mod_sq_interpretation}, we view both $\ImageModSquares\otimes B[2]$ and $V_0$ as  subgroups of $\oplus_{v\in \Sigma}H^1(\Q_v,B[2])$.  

\begin{lemma} \label{lem:alternative_syst_sub_desc}
We have 
\[\cS_{b}^r=V_0\cap (\ImageModSquares\otimes B[2])^\perp \cap Z_0,\]
where the orthogonal complement is taken with respect to the local Tate pairing $\left \langle~,~\right \rangle_\Sigma$.
\end{lemma}
 
 \begin{proof}
 Note that $\cS_{b}^r$ is the subgroup of $V_0\cap Z_0$ consisting of cocycles which factor through  $\Gal(K_\Sigma \cap L/\Q)$.  Recall that  $\left \langle~,~\right \rangle_\Sigma$ is equal to $H_\Sigma\otimes e$, where
\[H_\Sigma\colon\oplus_{v\in \Sigma}H^1(\Q_v,\F_2)\times \oplus_{v\in \Sigma} H^1(\Q_v,\F_2)\longrightarrow \F_2\]
is   the sum of the local  `Hilbert symbol'  pairings. We thus have 
\[V_0\cap (\ImageModSquares\otimes B[2])^\perp = \big(H^1(K_\Sigma/\Q,\F_2)\cap \ImageModSquares^\perp\big)\otimes B[2],\]
where the orthogonal complement on the left, respectively right is taken with respect to
$\left \langle~,~\right \rangle_\Sigma$, respectively $H_\Sigma$. By definition,
$\cF$ is generated by primes $p\nmid N$ that split completely in the Galois extension $L/\Q$.
Given such a prime $p$, and given $\chi\in H^1(K_\Sigma/\Q,\F_2)$, it follows from reciprocity that $H_\Sigma(\chi,\psi_p)=0$ if and only if $\chi\big(\sigma_p)=0$.
By the Chebotarev density theorem, this last condition holds for all $p\in \cF$ if and only if $\chi$ factors through $\Gal(K_\Sigma \cap L/\Q)$.
\end{proof}
The following result shows that $\Sel_2(A_d/\Q)\cap H^1(K_\Sigma/\Q,A[2])$  is equal to the systematic subspace for $100\%$ of $d\in d_0\cF_c$.

\begin{proposition} \label{lem:systematic_subspace_main_lemma}
We have 
\[\lim_{X\rightarrow \infty} \frac{\#\{d\in  d_0\cF_{c}(X)~~\colon~~\Sel_2(A_d/\Q)\cap H^1(K_\Sigma/\Q,A[2])=\cS_{b}\}}{ \# \cF_{c}(X) }=1.\]
\end{proposition}

\begin{proof}
Examining the explicit parameterisation of Selmer elements given in \Cref{ssec:parameterising_selmer_elts}, we see that 
\begin{eqnarray*}
 \lefteqn{\frac{1}{\#  \cF_{c}(X)}\sum_{ d\in  d_0\cF_{c}(X)}
  \#\Big(\Sel_2(A_d/\Q)\cap H^1(K_\Sigma/\Q,A[2])\Big)^r =} \\
& & \frac{1}{\#  \cF_{c}(X)}\cdot \frac{1}{\#Z_0}~\sum_{(v_0,z_0)\in V_0\times Z_0} (-1)^{\left \langle v_0,z_0\right \rangle_\Sigma}\sum_{ T\subseteq 0\times B[2] } \cS^T_{v_0,z_0}(X).
\end{eqnarray*}
The single maximal unlinked subset of $0\times B[2]$ is $0\times B[2]$ itself.
Arguing as in the proof of Theorem \ref{thm:main_selmer_asymptotic_moment_dist}, we thus see that we have
\begin{eqnarray*}
 \lim_{X\rightarrow \infty} \frac{1}{\# \cF_{c}(X)}\sum_{ d\in  d_0\cF_{c}(X)}
  \#\Big(\Sel_2(A_d/\Q)\cap H^1(K_\Sigma/\Q,A[2])\Big)^r =   \sY(0\times B[2]).
\end{eqnarray*}
Taking $W=0\times B[2]$ in Lemma \ref{preliminary_evaluation_of_sum} (noting that the map $\eta$ is surjective) we see that
\[ \sY_{v_0,z_0}(0\times B[2])=\frac{1}{\#(\ImageModSquares\otimes B[2])}\sum_{t\in \ImageModSquares\otimes B[2]}(-1)^{\left \langle v_0,t\right \rangle_\Sigma}=\mathbbm{1}_{v_0\in (\ImageModSquares\otimes B[2])^\perp }.\]
Since $Z_0$ is its own orthogonal complement with
respect to the pairing $\left \langle~,~\right \rangle_\Sigma$, we have
\begin{eqnarray*}
\lim_{X\rightarrow \infty} \frac{1}{\# \cF_{c}(X)}\sum_{ d\in  d_0\cF_{c}(X)}
\#\Big(\Sel_2(A_d/\Q)\cap H^1(K_\Sigma/\Q,A[2])\Big)^r     
& =& \#V_0\cap (\ImageModSquares\otimes B[2])^\perp \cap Z_0=\#\cS_{b}^r,
\end{eqnarray*}
where the final equality follows from Lemma \ref{lem:alternative_syst_sub_desc}.
The result follows  from the case $r=1$.
\end{proof}

\subsection{Counting injections}
In the statement of the following result, for $\F_2$-vector spaces $U$, $W$, we denote by $\Inj(U,W)$  the set of injective homomorphisms $U\to W$. 

\begin{theorem} \label{thm:inj_moments}
Let $L/\Q$ be a finite Galois extension, let $c\in S\subseteq \left(\Z/4N\Z\right)^{\times}$ and let $d_0$ be a (possibly negative)
divisor of $N$.  

Then we have 
\begin{eqnarray*}
 \lim_{X\rightarrow \infty} \frac{1}{\#  \cF_{c}(X)}\sum_{d\in  d_0\cF_{c}(X)}
  \#\Inj\bigg(\F_2^r,\frac{\Sel_2(A_{d}/\Q)}{\delta_{d}(A[2])+\cS_{b}}\bigg)   
=  \frac{1}{ \#\cS_{b}^r} \sum_{\substack{W\subseteq B[2]^2\textup{ max iso.}\\W\textup{ satisfies }(\star)}}\sY(W),
\end{eqnarray*}
where the condition in the sum on the right hand side is:  
\begin{equation*}
  \textup{there is no codimension }1\textup{ subspace }N\subseteq \F_2^r \textup{ such that }\pi_1(W)\subseteq A[2]\otimes N.\tag{$\star$}
\end{equation*}
\end{theorem}

\begin{proof}
For $T\subseteq B[2]^2$, we say that $T$ satisfies condition $(\star)'$ if
there does not exist an element $\bfw \in B[2]^2$ and a codimension $1$ subspace
$N\subseteq \F_2^r$ such that one has $\pi_1(T+\bfw)\subseteq A[2]\otimes N$.
When $T$ is a maximal isotropic subspace, this condition specialises to condition $(\star)$ in the statement.
Arguing as in Proposition \ref{prop:selmer_moment_formulae} and using Lemma \ref{lem:is_injection_criteria} we see that
\begin{equation}
\frac{1}{|Z_0|}~\sum_{(v_0,z_0)\in V_0\times Z_0}
  (-1)^{\left \langle v_0,z_0\right \rangle}\sum_{\substack{T\subseteq B[2]^2\\T\textup{ satisfies }(\star)'}}\cS^T_{v_0,z_0}(X)
\end{equation}
is equal to the sum, over $d\in d_0\cF_{c}(X)$, of the number of homomorphisms $\F_2^r \rightarrow \Sel_2(A_d/\Q)$
that are injective after composing with the quotient map 
\[ \Sel_2(A_d/\Q)	 \longrightarrow \frac{\Sel_2(A_d/\Q)}{\delta_d(A[2])+\Sel_2(A_d/\Q)\cap H^1(K_\Sigma/\Q,A[2])}.\]
 Denote by $\lambda_d$ the number of such homomorphisms.
Arguing as in the proof of Theorem \ref{thm:main_selmer_asymptotic_moment_dist}, we see that the limit, as $X$ tends to infinity, of the quantity $ \frac{1}{\#  \cF_{c}(X)}\sum_{\substack{d\in  d_0\cF_{c}(X)}}
  \lambda_d$ is equal to  
\[   \#A[2]^r  \sum_{\substack{W\subseteq B[2]^2 \textup{ max iso.}\\W\textup{ satisfies }(\star)}}\sY(W). \]

Denote by $\lambda_d'$ the number of homomorphisms  $\F_2^r \rightarrow \Sel_2(A_d/\Q)$ that are injective after composing with the quotient map 
\[ \Sel_2(A_d/\Q)	 \longrightarrow \frac{\Sel_2(A_d/\Q)}{\delta_d(A[2])+ \cS_{b}}.\]
We claim that we have
\begin{eqnarray}\label{eq:sum_lambdas}\lim_{X\rightarrow \infty} \frac{1}{\#  \cF_{c}(X)}\sum_{ d\in  d_0\cF_{c}(X)  }
  \lambda_d=\lim_{X\rightarrow \infty} \frac{1}{\#  \cF_{c}(X)}\sum_{ d\in  d_0\cF_{c}(X) }
\lambda_d'.
\end{eqnarray}
Indeed, by the Cauchy--Schwarz inequality, we have 
\begin{eqnarray*}
\frac{1}{\#  \cF_{c}(X)} \sum_{ d\in  d_0\cF_{c}(X)}
 \mid\lambda_d-\lambda_d'\mid&=& \frac{1}{\# \cF_{c}(X)}\sum_{ d\in  d_0\cF_{c}(X) }\charfunc_{\lambda_d\neq \lambda_d'}\cdot \mid \lambda_d-\lambda_d'\mid
 \end{eqnarray*}
 \[\phantom{lots of space}\leq \bigg(\frac{\#\{d\in d_0\cF_{c}(X)~~\colon~~\lambda_d\neq \lambda_d'\}}{ \# \cF_{c}(X) }\bigg)^{\tfrac{1}{2}} \cdot \bigg( \frac{1}{\# \cF_{c}(X)}\sum_{ d\in  d_0\cF_{c}(X)}
 \mid\lambda_d-\lambda_d'\mid^2\bigg)^{\tfrac{1}{2}}.\]
 Now $ \mid\lambda_d-\lambda_d'\mid^2$ is bounded above by $4\#\Sel_2(A_d/\Q)^{2r}$. It thus follows from Remark \ref{moments_bounded_remark} that the limit, as $X$ tends to infinity, of the quantity  
\[
  \frac{1}{\# \cF_{c}(X)}\sum_{ d\in  d_0\cF_{c}(X) } \mid\lambda_d-\lambda_d'\mid^2
\]
exists and is finite.
We have $\lambda_d = \lambda_d'$ whenever $\Sel_2(A_d/\Q)\cap H^1(K_\Sigma/\Q,A[2])=\cS_{b}$, so that the claim follows from \Cref{lem:systematic_subspace_main_lemma}.

Each injective homomorphism   $\F_2^r\to \Sel_2(A_d/\Q)/\big(\delta_d(A[2])+ \cS_{b}\big)$ lifts to 
$\#\big(\delta_d(A[2])+ \cS_{b}\big)^r$ 
homomorphisms $\F_2^r\rightarrow \Sel_2(A_d/\Q)$. By Lemma \ref{lem:twisted_boundary_map} we have $\# \big(\delta_d(A[2])+ \cS_{b}\big)^r=\#A[2]^r\cdot \#\cS_{b}^r$ whenever $d\neq d_0$, hence 
\[\lambda_d'=\#A[2]^r\cdot \#\cS_{b}^r\cdot 
  \#\Inj\bigg(\F_2^r,\frac{\Sel_2(A_d/\Q)}{\delta_d(A[2])+\cS_{b}}\bigg)   \]
for all but finitely many $d$, so that the result follows from
\eqref{eq:sum_lambdas}.
\end{proof}

\begin{remark}\label{rmrk:direct_sum}
For all but finitely many $d\in d_0\cF_{c}$, the sum $\delta_d(A[2])+\cS_b$ is direct. Indeed, if there exists a prime $p\mid d$ with $p\notin \Sigma$, then it follows from Lemma \ref{lem:twisted_boundary_map} (restricted to the inertia group at $p$) that $\delta_d(A[2])\cap \cS_b=0.$ 
\end{remark}

\section{Combinatorics of the leading term}\label{sec:combinatorics}
Take all the notation of the previous section. The main purpose of this section
is to prove \Cref{thm:intro_distr}. This will be achieved by determining the ``moments''
of the distribution obeyed by the $2$-Selmer groups in suitable families. In light
of Theorem \ref{thm:inj_moments}, the purely combinatorial statement of
\Cref{thm:main_term_inj_contribution} below, the proof of which will occupy most
of the section, will be the final missing ingredient.

Recall from Definition \ref{def:definition_of_gamma} that we have a natural
injection $\gamma\colon \Gal(\Q(A[4])/\Q)\hookrightarrow \End(A[2])$,
which is a homomorphism for the additive structure on $\End(A[2])$.
Recall also that we denote by $\Gamma$ the image of
$\Gal\big(\Q(A[4])/L\cap \Q(A[4])\big)\cong \Gal(L(A[4])/L)$ under this map.
In the notation of the previous section, $\Gamma$ is the additive subgroup of
$\End A[2]$ generated by $\{\gamma(\sigma_m)\colon m\in \ImageModSquares\}$.
Denote by $\End_{\Gamma}(A[2])$ the additive subgroup of endomorphisms of $A[2]$
commuting with all elements of $\Gamma$. We will sometimes assume the following condition.
 
\begin{condition}\label{assumption:simple_gamma_mod}
There are no $\Gamma$-invariant subspaces of $A[2]$ other than $0$ and $A[2]$.
Moreover, we have $\End_{\Gamma}(A[2])=\F_2$. 
\end{condition}
\begin{remark}\label{rmrk:favored}
We briefly compare this condition with the hypotheses of \cite{Smith2}*{Theorem 2.14, Case 2.13} when $L=\Q$.
Smith's condition of the Galois module being ``potentially favored'' is implied by the assumption
of full rational $2$-torsion: just take $w=1$ in \cite{Smith2}*{Proposition 4.1}.
Next, Smith's condition of the module being ``uncofavored''
is equivalent to our assumption that $A[2]$ have no proper non-trivial $\Gamma$-invariant subgroups.
The alternating structure on $N$ is induced by the principal polarisation on $A$.
In light of the uncofavored condition, the condition in the fourth bullet point
of \cite{Smith2}*{Case 2.13} can be checked to be equivalent to the assumption that
$\End_{\Gamma}A[2] = \F_2$; see also \cite{Smith2}*{Remark 4.6}.
Finally, one can check that the assumption $L=\Q$ implies that as $d$ runs over
$d_0\cF_c$, there are no ``universal Selmer elements''.
\end{remark}

\begin{theorem} \label{thm:main_term_inj_contribution}
Take the notation and assumptions of \Cref{thm:inj_moments}. Suppose further that
Condition \ref{assumption:simple_gamma_mod} is satisfied. Then we have 
\begin{equation} \label{main_term_inj_contribution}
    \sum_{\substack{W\subseteq B[2]^2\textup{ max iso.}\\W\textup{ satisfies }(\star)}}\sY(W)=
      \#\cS_{b}^r\cdot 2^{\frac{r(r+1)}{2}}.
\end{equation}
\end{theorem}
The theorem will be proved in Section \ref{sec:main_term_inj_proof}.

\subsection{Description of maximal isotropic subspaces} \label{ssec:max_iso_descr}
  
We begin by explicitly describing the maximal isotropic subspaces of $B[2]^2$. 
Recall from Definition \ref{quad_form_q_defi} the quadratic form $q\colon B[2]^2\rightarrow \F_2$
sending $\bfu$ to $e(\pi_1(\bfu),\pi_2(\bfu))$. 

\begin{definition}\label{def:U_phi_max_iso}
Let $(U,\phi)$ be a pair consisting of a subspace $U$ of $B[2]$, and a homomorphism
$\phi\colon U\rightarrow B[2]/U^\perp$ satisfying $e(u,\phi(u))=0$ for all $u\in U$.
Here $U^\perp$ denotes the orthogonal complement of $U$ with respect to the Weil pairing
on $B[2]$. Define the subspace $W_{U,\phi}$ of $B[2]^2$ as  
\[
  W_{U,\phi}=\big\{(u,\phi(u)+u')~~\colon~~u\in U, u' \in U^\perp\big\}.
\]
By construction, $W_{U,\phi}$ is isotropic for $q$. Since we also have
\[
  \dim W_{U,\phi}=\dim U+\dim U^\perp=\dim B[2],
\]
we conclude that $W_{U,\phi}$ is a maximal isotropic subspace for $q$.
Note  that the subspaces $W_{U,\phi}$ are distinct as $U$ and $\phi$ vary.
\end{definition}
   
\begin{lemma} \label{desc_of_max_iso}
Let $W\subseteq B[2]^2$ be a maximal isotropic subspace for $q$. Then $W$ has the form $W_{U,\phi}$ for some $U,\phi$
as in \Cref{def:U_phi_max_iso}.
\end{lemma}
 
\begin{proof}  
Define $U=\pi_1(W)\subseteq B[2]$, and define
\[
  K=\pi_2\left(\ker(\pi_1\colon W\rightarrow B[2])\right)\subseteq B[2].
\]
We have a short exact sequence 
\begin{equation}
0\longrightarrow K \stackrel{v\mapsto (0,v)}{\longrightarrow}W\stackrel{\pi_1}{\longrightarrow} U\longrightarrow 0. 
\end{equation}
 We claim that $K=U^\perp$. Indeed, as in Lemma \ref{lem:geom_of_unlinked_indices}, since $W$ is maximal isotropic, we have 
$\dim W=\dim B[2].$
Since  we also have $\dim B[2]=\dim U+\dim U^\perp$, it suffices to show
that $K\subseteq U^\perp$. Let $u\in U$ and let $(u,\lambda)$ be any element of $W$ mapping to $u$ under $\pi_1$. Since $W$ is isotropic  for $q$, we have  $e(u,\lambda)=0$. Now given any $k\in K$, the element $(u,\lambda+k)$ is in $W$, so we have 
\[0=e(u,\lambda+k)=e(u,\lambda)+e(u,k)=e(u,k).\]
Thus $e(u,k)=0$. Since $u\in U$ and $k\in K$ were arbitrary, we conclude that $K$ is contained in $U^\perp$, and the claim is proved.

In light of the claim, we have a well-defined homomorphism $\phi\colon U=\pi_1(W)\rightarrow B[2]/U^\perp$ given by lifting elements of $U$ to $W$ under $\pi_1$, and then applying $\pi_2$. Moreover, we have
\[W=\big\{(u,\phi(u)+U^\perp)~~\colon~~u\in U\big\}.\]
The  condition that $W$ be isotropic for $q$ forces $e(u,\phi(u))=0$ for all $u\in U$.
\end{proof}
 
 \begin{definition}
   For a given $U\subseteq B[2]$, define 
\[\sA_U= \big\{\phi\in \Hom(U,B[2]/U^\perp) ~~\colon~~e(u,\phi(u))=0 \textup{ for all }u\in U \big\}.\]
\end{definition}

\begin{remark} \label{alt_pairings_correspondence_rmk}
Since the pairing 
$ U\otimes B[2]/U^\perp \rightarrow \F_2$
induced by $e$ is non-degenerate, the natural map $B[2]/U^\perp \rightarrow U^*$
sending $x$ to $e(-,x)$ is an isomorphism. From this we see that the map sending
$\phi \in \sA_U$ to the pairing $P_\phi\colon U\otimes U\rightarrow \F_2$, defined by
$ P_\phi(u,w)=e(u,\phi(w))$,
identifies  $\sA_U$ with the vector space of alternating pairings on $U$.
In particular, we have $\dim \sA_U=\tfrac{1}{2}n(n-1)$, where $n=\dim U$.
\end{remark}

 \subsection{Reduction to $\Gamma$-invariant subspaces}
 We now turn to understanding the left hand side of \eqref{main_term_inj_contribution}, i.e. the sum
   $$ \sum_{\substack{W\subseteq B[2]^2\textup{ max iso.}\\W\textup{ satisfies }(\star)}}\sY(W)= \sum_{\substack{W\subseteq B[2]^2\textup{ max iso.}\\W\textup{ satisfies }(\star)}}\frac{1}{|Z_0|}\sum_{(v_0,z_0)\in V_0\times Z_0}(-1)^{\left \langle v_0,z_0\right \rangle_\Sigma}\sY_{v_0,z_0}(W),$$
   where condition $(\star)$ is as in the statement of  \Cref{thm:inj_moments}.
 The group $\Gamma\subseteq \End(A[2])$ operates diagonally on $B[2]=A[2]^r$.
The key observation making the analysis of the sums $\sY_{v_0,z_0}(W)$ manageable is that
for $U$, $\phi$ as in \Cref{def:U_phi_max_iso} we have $\sY_{v_0,z_0}(W_{U,\phi})=0$
unless $U$ is invariant under $\Gamma$. 
We will prove a refined version of this that additionally gives information
about the sums $\sY_{v_0,z_0}(W_{U,\phi})$ when $\Gamma U\subseteq U$. Our starting point will be the expression for the sums $\sY_{v_0,z_0}(W)$ given in the statement of Lemma \ref{preliminary_evaluation_of_sum}, see also Notation \ref{notat:image_mod_sq_interpretation}.  We refer to that lemma for the definition of the function $\Xi_W$ appearing in the statement of Lemma \ref{innermost_sum_lemma} below.

Finally, we introduce the following additional notation.

\begin{notation} \label{random_notat_for_this_prop} 
For a subspace $U$ of $B[2]$,
write $T_1=\ImageModSquares\otimes U$ and $T_2=\ImageModSquares\otimes U^\perp$.
Define $\eta_1\colon \ImageModSquares^U\rightarrow T_1$ and $\eta_2\colon\ImageModSquares^{U^\perp}\rightarrow T_2$ by 
\[
  \eta_1(\bfx)=\sum_{u\in U}x_u\otimes u\quad \textup{ and }\quad \eta_2(\bfy)=\sum_{\lambda \in U^\perp}y_\lambda \otimes \lambda.
\] 
\end{notation}
 
For $\phi \in \sA_U$, we have an isomorphism 
\begin{equation} \label{W_parametrisation}
  U\oplus U^\perp \stackrel{\sim}{\longrightarrow}W_{U,\phi}
\end{equation}
sending $(u,\lambda)$ to $(u, \phi(u)+\lambda)$. Via this isomorphism we may write
an element $(h_w)_{w}\in \ImageModSquares^W$ as a tuple $\bfh=(h_{u,\lambda})_{u\in U,\lambda\in U^\perp}$. 
The isomorphism \eqref{W_parametrisation} induces an isomorphism $\ImageModSquares\otimes W_{U,\phi}\cong  T_1\times T_2$.

\begin{lemma} \label{innermost_sum_lemma}
Let $U\subseteq B[2]$, $\phi\in \sA_U$, and $W=W_{U,\phi}$.
Let $t\in \ImageModSquares\otimes W$ have image  $(t_1,t_2)\in T_1\times T_2$. 
Denote by
\begin{equation} \label{eq:IT_definition}
I_t=\sum_{\substack{(h_w)_w\in \ImageModSquares^W\\ \eta((h_w)_w)=t}}(-1)^{\Xi_W((h_w)_w)}
\end{equation}
the innermost sum in \eqref{SWsum}.
Then we have $I_t=0$ unless $\Gamma U\subseteq U$. When $\Gamma U\subseteq U$, we have 
\[
I_t= \frac{\#\ImageModSquares^{2^{2rg}}}{\#\ImageModSquares^{\#U}\cdot \#T_2}
  \sum_{\substack{\bfx\in \ImageModSquares^U\\ \eta_1(\bfx)=t_1}}(-1)^{\Upsilon_\phi(\bfx)},
\]
where 
\begin{equation} \label{defi_of_upsilon}
\Upsilon_\phi(\bfx)=\sum_{\{u,u'\}\subseteq U}\alpha(x_u)\alpha(x_{u'})e\big(u,\phi(u')\big)+\sum_{u\in U }e\big(\gamma^{(r)}(\sigma_{x_u})u,\phi(u)\big).
\end{equation}
\end{lemma}

\begin{proof}
Temporarily define projections $p_1\colon \ImageModSquares^{U\oplus U^\perp}\rightarrow \ImageModSquares^U$ and
$p_2\colon \ImageModSquares^{U\oplus U^\perp}\rightarrow \ImageModSquares^{U^\perp}$ by setting 
\[
  p_1(\bfh)=(\textstyle\sum_{\lambda}h_{u,\lambda})_u\quad \textup{ and }
  \quad  p_2(\bfh)=(\textstyle\sum_{u}h_{u,\lambda})_\lambda.
\]
Further, write $\widetilde{\eta}$ for the map $\ImageModSquares^{U\oplus U^\perp}\rightarrow T_1\times T_2$ defined by 
\[
  \widetilde{\eta}(\bfh)=\big(\eta_1\circ p_1(\bfh),\eta_2\circ p_2(\bfh)\big).
\] 
Under the identifications $\ImageModSquares\otimes W\cong T_1\times T_2$ and
$\ImageModSquares^W\cong \ImageModSquares^{U\oplus U^\perp}$ above,
the map $\widetilde{\eta}$ identifies with the map
$\eta\colon \ImageModSquares^W\rightarrow \ImageModSquares\otimes W$ defined previously  

Writing an element $(h_w)_{w}\in \ImageModSquares^W$ as a tuple $\bfh=(h_{u,\lambda})_{u\in U,\lambda\in U^\perp}$
using the isomorphism \eqref{W_parametrisation}, we have
\begin{equation*}
\Xi_W(\bfh)=\sum_{\{(u,\lambda),(u',\lambda')\}\subseteq U\oplus U^\perp}\alpha(h_{u,\lambda})\alpha(h_{u',\lambda'})e(u,\phi(u')+\lambda')+\sum_{u\in U, \lambda \in U^\perp}e(u,\gamma^{(r)}(\sigma_{h_{u,\lambda}})(\phi(u)+\lambda)).
\end{equation*}
In the first sum, we have $e(u,\lambda')=0$ since $u\in U$ and $\lambda'\in U^\perp$.
Further, this sum is over unordered distinct pairs $(u,\lambda)$ and $(u',\lambda')$ of $U\oplus U^\perp$.
If $u=u'$ then $e(u,\phi(u'))=0$, so we can assume that $u\neq u'$. Thus in the first
sum above, we can sum instead over unordered distinct pairs $u,u'$ of $U$, and then
sum freely over pairs $\lambda, \lambda'$ of $U^\perp$. 
For the second sum, using \Cref{gamma_and_weil_lemma} we see that, for $u\in U$ and $\lambda \in U^\perp$, we have
\[e(u,\gamma^{(r)}(\sigma_{h_{u,\lambda}})(\phi(u)+\lambda))=e(\gamma^{(r)}(\sigma_{h_{u,\lambda}})u,\phi(u))+e(\gamma^{(r)}(\sigma_{h_{u,\lambda}})u,\lambda).\]
For $u\in U$, set  $x_u(\bfh)=\sum_{\lambda \in U^\perp}h_{u,\lambda}$. From the above discussion, we conclude that
\begin{equation} \label{eq:slowly_simplifying_Xi}
 \Xi_W(\bfh) = \Upsilon_\phi((x_u(\bfh))_u)+\sum_{u\in U,\lambda \in U^\perp}e(\gamma^{(r)}(\sigma_{h_{u,\lambda}})u,\lambda).
\end{equation}

Substituting the expression \eqref{eq:slowly_simplifying_Xi} for  $\Xi_W(\bfh)$ into \eqref{eq:IT_definition} we have
\begin{equation} \label{IT_simplification_formula}
I_t= \sum_{\substack{\bfx\in \ImageModSquares^U\\ \eta_1(\bfx)=t_1}}(-1)^{\Upsilon_\phi(\bfx)}\cdot
\sum_{\substack{\bfh\in \ImageModSquares^{U\oplus U^\perp}\\ \eta_2\circ p_2(\bfh)=t_2\\ p_1(\bfh)=\bfx}}
(-1)^{\sum_{(u,\lambda)\in U\oplus U^\perp}e(\gamma^{(r)}(\sigma_{h_{u,\lambda}})u,\lambda)}.
\end{equation}
The function $\Theta\colon \ImageModSquares^{U\oplus U^\perp} \rightarrow \{\pm 1\}$ defined by
\[
  \Theta\colon \bfh\longmapsto (-1)^{\sum_{(u,\lambda)\in U\oplus U^\perp}e(\gamma^{(r)}(\sigma_{h_{u,\lambda}})u,\lambda)}
\]
is a homomorphism.  
Let us temporarily denote by $B$ the subgroup of $\ImageModSquares^{U\oplus U^\perp}$
consisting of elements $\bfh$ for which $\eta_2(\bfh)=0$ and $p_1(\bfh)=0$.
Then the innermost sum 
\begin{equation} \label{eq:inner_inner_sum}
\sum_{\substack{\bfh\in \ImageModSquares^{U\oplus U^\perp}\\ \eta_2\circ p_2(\bfh)=t_2\\ p_1(\bfh)=\bfx}}(-1)^{\sum_{(u,\lambda)\in U\oplus U^\perp}e(\gamma^{(r)}(\sigma_{h_{u,\lambda}})u,\lambda)}
\end{equation}
in the above expression for $I_t$ is the sum of the values of $\Theta$ over a coset of $B$.
Consequently, this sum is $0$ unless $\Theta$ is identically $1$ on $B$, that is, unless the quantity
\begin{equation} \label{has_to_be_zero_on_A}
\sum_{(u,\lambda)\in U\oplus U^\perp}e(\gamma^{(r)}(\sigma_{h_{u,\lambda}})u,\lambda)
\end{equation}
vanishes identically on $B$. We will show that this happens precisely when $U$ is invariant under $\Gamma$.
This will show that $I_t$ is zero unless $\Gamma U\subseteq U$.

Given $\bfh\in \ImageModSquares^{U\oplus U^\perp}$, we have 
\[
  \eta_2\circ p_2(\bfh)=\sum_{\lambda \in U^\perp}\big(\textstyle{\sum_{u\in U}}h_{u,\lambda}\big)\otimes \lambda\quad
  \textup{ and }\quad p_1(\bfh)=\big(\sum_{\lambda \in U^\perp}h_{u,\lambda}\big)_{u\in U}.
\]
We claim that the condition that \eqref{has_to_be_zero_on_A} vanish on all elements of $B$
is equivalent to the condition that it vanish on all elements of $\ImageModSquares^{U\oplus U^\perp}$.
Indeed, given an arbitrary element $\bfh\in\ImageModSquares^{U\oplus U^\perp}$,
we can ensure that $\eta_2\circ p_2(\bfh)=0=p_1(\bfh)$ by modifying only the values of
$h_{0,\lambda}$ and $h_{u,0}$ for varying $u\in U$ and $\lambda \in U^\perp$;
such a modification does not change the value of \eqref{has_to_be_zero_on_A}. Now pick $u\in U$,
$\lambda \in U^\perp$, and $h\in \ImageModSquares$. Take $\bfh\in \ImageModSquares^{U\oplus U^\perp}$ to be equal to $h$ in the $(u,\lambda)$-slot, and equal to $0$ everywhere else. The condition that \eqref{has_to_be_zero_on_A} vanish on $\bfh$ is the condition that 
$e(\gamma^{(r)}(\sigma_{h})u,\lambda)=0.$
Since $h$ and $\lambda$ were chosen arbitrarily, we conclude that $\sigma u \in (U^\perp)^\perp=U$ for all $\sigma\in \Gamma$. Since $u\in U$ was chosen arbitrarily, we conclude that $\Gamma U\subseteq U$. Conversely, if $\Gamma U\subseteq U$ then the right hand side of \eqref{has_to_be_zero_on_A} visibly vanishes for all $\bfh\in \ImageModSquares^{U\oplus U^\perp}$. 

Henceforth, suppose that $\Gamma U\subseteq U$. Then  \eqref{eq:inner_inner_sum} is equal to
\[N_{t_2,\bfx}:=\#\{\bfh\in \ImageModSquares^{U\oplus U^\perp}~~\colon~~\eta_2\circ p_2(\bfh)=t_2,~~p_1(\bfh)=\bfx\}.\]
The map $\ImageModSquares^{U\oplus U^\perp}\rightarrow T_2\times \ImageModSquares^{U}$ sending $\bfh$ to $(\eta_2\circ p_2(\bfh),p_1(\bfh))$ is readily seen to be surjective (even when restricted to the subset of elements $\bfh\in \ImageModSquares^{U\oplus U^\perp}$ for which $\bfh_{u,\lambda}=0$ whenever both $u$ and $\lambda$ are non-zero). From this it follows that, for all $\bfx\in \ImageModSquares^U$ and $t_2\in T_2$, we have 
 \[N_{t_2,\bfx}=\frac{\#\ImageModSquares^{2^{2rg}}}{\#\ImageModSquares^{\#U}\cdot \#T_2}.\]
Substituting this into \eqref{IT_simplification_formula} gives the result. 
\end{proof}
\begin{proposition} \label{main_invariant_gamma_computation_prop} 
Let $U\subseteq B[2]$ and $\phi\in \sA_U$. Then
we have $\sY_{v_0,z_0}(W_{U,\phi})=0$ unless $\Gamma U\subseteq U$. Suppose that $\Gamma U\subseteq U$. Then we have 
\[\sY_{v_0,z_0}(W_{U,\phi})=\mathbbm{1}_{v_0\in (\ImageModSquares\otimes U^\perp)^\perp}\cdot \frac{1}{\#\ImageModSquares^{\#U}}\sum_{t_1\in T_1}(-1)^{\left \langle v_0,\phi(t_1)\right \rangle_\Sigma+\left \langle t_1,z_0\right \rangle_\Sigma}\sum_{\substack{\bfx\in \ImageModSquares^U\\ \eta_1(\bfx)=t_1}}(-1)^{\Upsilon_\phi(\bfx)},\]
where $\Upsilon_\phi(\bfx)$ is defined as in \Cref{innermost_sum_lemma}.
\end{proposition}

\begin{proof}
Write $W=W_{U,\phi}$.
Fix a linear map $s\colon B[2]/U^\perp \longrightarrow B[2]$
giving a section to the quotient map $B[2]\rightarrow B[2]/U^\perp$. By an abuse of notation,
we denote the composition $s\circ \phi$ by $\phi$ also, and thus view $\phi$ as a linear map $U\rightarrow B[2]$.
We then define the function $\Upsilon_\phi\colon \ImageModSquares^U\rightarrow \F_2$
by the expression in \eqref{defi_of_upsilon}, that is, with the help of the section $s$,
we extend the definition in the statement to the case when $U$ is not necessarily $\Gamma$-invariant.

An immediate consequence of Lemma \ref{innermost_sum_lemma} and  \eqref{SWsum} is that one has
$\sY_{v_0,z_0}(W)=0$ unless $\Gamma U\subseteq U$. When $\Gamma U\subseteq U$, substituting the
expression for $I_t$ afforded by Lemma \ref{innermost_sum_lemma} into  \eqref{SWsum} we see that
 \begin{eqnarray*}
\sY_{v_0,z_0}(W)&=&\frac{1}{\#\ImageModSquares^{\#U}\cdot \#T_2}\sum_{(t_1,t_2)\in T_1\times T_2}(-1)^{\left \langle v_0,\phi(t_1)+t_2\right \rangle_\Sigma+\left \langle t_1,z_0\right \rangle_\Sigma}\sum_{\substack{\bfx\in \ImageModSquares^U\\ \eta_1(\bfx)=t_1}}(-1)^{\Upsilon_\phi(\bfx)}\\
  &=& \frac{1}{\#\ImageModSquares^{\#U}}\sum_{t_1\in T_1}(-1)^{\left \langle v_0,\phi(t_1)\right \rangle_\Sigma+\left \langle t_1,z_0\right \rangle_\Sigma}\sum_{\substack{\bfx\in \ImageModSquares^U\\ \eta_1(\bfx)=t_1}}(-1)^{\Upsilon_\phi(\bfx)}\frac{1}{\#T_2}\sum_{t_2\in T_2}(-1)^{\left \langle v_0,t_2\right \rangle_\Sigma}.
 \end{eqnarray*}
The innermost sum is equal to $\#T_2$ if $v_0$ is orthogonal to $T_2=\ImageModSquares\otimes U^\perp$, and is equal to $0$ otherwise. This completes the proof.
\end{proof}

For convenience, we state the following immediate consequence of Proposition \ref{main_invariant_gamma_computation_prop}. 

\begin{cor}
Let $W\subseteq B[2]^2$ be a maximal isotropic subspace. Then we have $\sY(W)=0$ unless $\pi_1(W)$ is stable under $\Gamma$ (i.e. unless $\Gamma(\pi_1(W))\subseteq \pi_1(W)$).
\end{cor}

\begin{proof}
Set $U=\pi_1(W)$. By (the proof of) Lemma \ref{desc_of_max_iso}, we have $W=W_{U,\phi}$ for some $\phi \in \sA_U$. Suppose that $U$ is not stable under $\Gamma$. By Proposition \ref{main_invariant_gamma_computation_prop}, for all $(v_0,z_0)\in V_0\times Z_0$, we have $\sY_{v_0,z_0}(W)=0$. Thus $\sY(W)=0$ also.
\end{proof}

 \subsection{Description of $\Gamma$-invariant subspaces in the case of interest}
 
 \begin{lemma} \label{simple_subspace_lemma}
Suppose that Condition \ref{assumption:simple_gamma_mod} is satisfied.
Then every $\Gamma$-invariant
subspace of $B[2]=A[2]\otimes\F_2^r$ has the form $A[2]\otimes T$ for some subspace $T\subseteq \F_2^r$. 
\end{lemma}
 
\begin{proof}
Let $R$ be the subring of $\End(A[2])$ generated by $\Gamma$, and let $M$ be a $\Gamma$-invariant
subspace of $A[2]^r$. Thus $M$ is an $R$-submodule of $A[2]^r$. Consider the map 
\[\Theta\colon \Hom(\F_2^r,\F_2)\rightarrow \Hom_R(M,A[2])\]
sending $\phi$ to the restriction of $1\otimes \phi$ to $M$. Let $N$ be the kernel of $\Theta$, so that we have an induced injection 
\begin{equation} \label{inj_to_homs_xi_map}
  \Hom(\F_2^r,\F_2)/N\hookrightarrow \Hom_{R}(M,A[2]).
\end{equation}
Let $s=\dim N$. The collection of subspaces of $A[2]^r$ of the form $A[2]\otimes T$
is stable under the action of $\GL_r(\F_2)$ on the second factor of $A[2]\otimes \F_2^r$.
Consequently, acting by an element of $\GL_r(\F_2)$ if necessary, we can assume that $N$
is the span of the first $s$ coordinate projections, hence we have
\[M\subseteq 0\times A[2]^{r-s}=A[2]\otimes (0\times \F_2^{r-s}).\]
We will show that this containment is in fact an equality. By assumption,
$A[2]$ is a simple $R$-module. In particular, the $R$-submodule $M$ of $A[2]^r$ is semisimple,
and every simple $R$-submodule of $M$ is isomorphic to $A[2]$. Thus we have $M\cong A[2]^t$ for some
integer $t\geq 0$. Counting dimensions, it suffices to show that $t\geq r-s$. Since
we assumed that $\End_R(A[2])=\F_2$, we have
\[\dim \Hom_R(M,A[2])=\dim \Hom_R(A[2]^t,A[2])=\dim \End_R(A[2])^t=t.\] 
Taking dimensions in \eqref{inj_to_homs_xi_map} gives the desired inequality.  
\end{proof}

\begin{remark}
In the notation of the proof of \Cref{simple_subspace_lemma}, before acting by $\GL_r(\F_2)$,
let $N^\circ\subseteq \F_2^r$ be the annihilator of $N\subseteq \Hom(\F_2^r,\F_2)$.
Then the proof of \Cref{simple_subspace_lemma} shows that we have $M=A[2]\otimes N^\circ$.
\end{remark}

\subsection{Combinatorics of the main term}\label{sec:maintermcomb}

When Condition \ref{assumption:simple_gamma_mod} is satisfied, it follows from \Cref{simple_subspace_lemma} and \Cref{main_invariant_gamma_computation_prop} that
the only maximal isotropic subspaces $W=W_{U,\phi}$ that give a non-zero contribution
to \eqref{main_term_inj_contribution} are those for which we have $U=B[2]$.  Motivated by this, we restrict attention to these subspaces (but do not presently assume Condition \ref{assumption:simple_gamma_mod}).

 By \Cref{desc_of_max_iso}, the set of such subspaces is in bijection with the set  
\[\sA:=\sA_{B[2]}=\{\phi \in \End(B[2])~~\colon~~e_2(u,\phi(u))=0\textup{ for all }u\in B[2]\}.\]
For $\phi\in \sA$ and $(v_0,z_0)\in V_0\times Z_0$, write $\sY_{v_0,z_0}(\phi)=\sY_{v_0,z_0}(W)$, where $W=W_{B[2],\phi}$. By \Cref{main_invariant_gamma_computation_prop} we have 
\begin{equation} \label{expression_for_S_phi_again}
\sY_{v_0,z_0}(\phi)=\frac{1}{|\ImageModSquares|^{2^{2rg}}}\sum_{t\in \ImageModSquares\otimes B[2]}(-1)^{\left \langle v_0,\phi(t)\right \rangle_\Sigma+\left \langle t,z_0\right \rangle_\Sigma}\sum_{\substack{\bfx\in \ImageModSquares^U\\ \eta_1(\bfx)=t}}(-1)^{\Upsilon_\phi(\bfx)}.
\end{equation}

The next step is to sum this expression over all $\phi \in \sA$. The following lemma establishes the key property of $\Upsilon_\phi$ we will need in order to do this. It is the analogue of \cite[Lemma 13]{MR1292115}. For the statement, we denote by $L_\phi\colon T_1\otimes T_1\rightarrow \F_2$ the bilinear pairing defined by
\[L_\phi(t,t')=e\big((\alpha \otimes 1)(t), (\alpha \otimes \phi)(t')\big).\]

\begin{lemma} \label{quad_form_phi_upsilon_lemma}
Let $\phi \in \sA$. Then, for any $\bfx, \bfy \in \ImageModSquares^{B[2]}$, we have
\[  \Upsilon_\phi(\bfx+\bfy)+\Upsilon_\phi(\bfx)+\Upsilon_\phi(\bfy)=L(\eta_1(\bfx),\eta_1(\bfy)).\]
In particular, the function $\Upsilon_\phi\colon \ImageModSquares^{B[2]}\rightarrow \F_2$ is a quadratic form whose associated bilinear pairing factors through $\eta_1$. 
\end{lemma}

\begin{proof}
 By definition, for any $\bfx\in \ImageModSquares^{B[2]}$, we have 
\[\Upsilon_\phi(\bfx)=\sum_{\{u,u'\}\subseteq B[2]}e\big(\alpha(x_u)u,\alpha(x_{u'})\phi(u')\big)+\sum_{u\in B[2] }e\big(u,\gamma(\sigma_{x_u})\phi(u)\big).\]
The second sum is linear in $\bfx$, so we may ignore it for the purposes of the lemma. The first sum is bilinear in $\bfx$, from which we deduce that 
\[ \Upsilon_\phi(\bfx+\bfy)+\Upsilon_\phi(\bfx)+\Upsilon_\phi(\bfy)=\sum_{\{u,u'\}\subseteq B[2]}\Big(e\big(\alpha(x_u)u,\alpha(y_{u'})\phi(u')\big)+e\big(\alpha(y_u)u,\alpha(x_{u'})\phi(u')\big)\Big).\]
It will be convenient to temporarily fix a linear ordering $<$ on $B[2]$. Having done this, the above expression is equal to 
\begin{equation} \label{HB_linear_ordering}
\sum_{u }\sum_{u'>u}e\big(\alpha(x_u)u,\alpha(y_{u'})\phi(u')\big)+\sum_u\sum_{u'>u}e\big(\alpha(y_u)u,\alpha(x_{u'})\phi(u')\big).\end{equation}
Since $e(u,\phi(u))=0$ for all $u\in B[2]$, the pairing $(u,u')\mapsto e(u,\phi(u'))$ is alternating, and in particular symmetric. Thus the second of the two sums above is equal to 
\[\sum_{u}\sum_{u'>u}e\big(\alpha(x_{u'})u',\alpha(y_u)\phi(u)\big)=\sum_{u}\sum_{u'<u}e\big(\alpha(x_u)u,\alpha(y_{u'})\phi(u')\big)\]
where the equality follows from relabelling $u\leftrightarrow u'$ and swapping the order of summation. Inserting this into \eqref{HB_linear_ordering} and again using that $e(u,\phi(u))=0$ for all $u$, we have 
\begin{eqnarray*}
\Upsilon_\phi(\bfx+\bfy)+\Upsilon_\phi(\bfx)+\Upsilon_\phi(\bfy)&=& \sum_{u}\sum_{u'}e\big(\alpha(x_u)u,\alpha(y_{u'})\phi(u')\big),
\end{eqnarray*}
which is equal to $L(\eta_1(\bfx),\eta_1(\bfy))$ as desired.
\end{proof}
 
\begin{definition}
Define 
\[
  \Psi\colon \ker(\eta_1)\rightarrow \Hom(\sA,\F_2)
\]
sending $\bfx$ to the map $\phi \mapsto \Upsilon_\phi(\bfx)$. Note that $\Psi$ is
a homomorphism, since by \Cref{quad_form_phi_upsilon_lemma}, for every
$\phi \in \sA$, the restriction of $\Upsilon_\phi$ to $\ker(\eta_1)$ is a homomorphism.
Let $\im(\Psi)^\circ \subseteq \sA$ denote the annihilator of $\im(\Psi)$.
\end{definition}

\begin{lemma}\label{lem:sum_over_phi}
For each $t_1\in \ImageModSquares\otimes B[2]$, fix an $\bfx(t_1)\in \ImageModSquares^{B[2]}$
with $\eta_1(\bfx)=t_1$. Given $t_1\in \ImageModSquares\otimes B[2]$,
$\bfx\in \ImageModSquares^{B[2]}$, and $v_0\in V_0$, let
$\lambda_{t_1,\bfx,v_0}\colon \sA\to \F_2$ be given by $\phi\mapsto
\Upsilon_{\phi}(\bfx) + \left \langle v_0,\phi(t_1)\right \rangle_\Sigma$.
Then we have 
\begin{eqnarray} \label{eq:formula_in_statement11}
\sum_{\substack{W\subseteq B[2]^2\textup{ max iso.}\\\pi_1(W)=B[2]}}\sY(W)= \frac{ \#\im(\Psi)^\circ}{\#(\ImageModSquares\otimes B[2])}\sum_{v_0 \in V_0}
\sum_{t_1\in \ImageModSquares\otimes B[2]} \mathbbm{1}_{\lambda_{t_1,\bfx(t_1),v_0}\in \im(\Psi)}\mathbbm{1}_{v_0+t_1\in Z_0}.
\end{eqnarray} 
\end{lemma}

\begin{proof}
Summing \eqref{expression_for_S_phi_again} over $\phi \in \sA$ gives
\begin{equation} \label{after_summing_over_phi}
 \sum_{\phi \in \sA}\sY_{v_0,z_0}(\phi)=
 \frac{1}{\#\ImageModSquares^{2^{2rg}}}\sum_{t_1\in T_1}(-1)^{\left \langle t_1,z_0\right \rangle_\Sigma}
 \sum_{\substack{\bfx\in \ImageModSquares^{B[2]}\\ \eta_1(\bfx)=t_1}}
 \sum_{\phi \in \sA}(-1)^{ \Upsilon_\phi(\bfx)+\left \langle v_0,\phi(t_1)\right \rangle_\Sigma}.
\end{equation}
First we fix $t_1\in T_1$, and consider the sum
\begin{equation} \label{inner_phi_sum}
 \sum_{\substack{\bfx\in \ImageModSquares^{B[2]}\\ \eta_1(\bfx)=t_1}}\sum_{\phi \in \sA}(-1)^{ \Upsilon_\phi(\bfx)+\left \langle v_0,\phi(t_1)\right \rangle_\Sigma}.
\end{equation}
 For fixed $\bfx$, the map $\sA\rightarrow \F_2$ given by 
\[\phi \mapsto \Upsilon_\phi(\bfx)+\left \langle v_0,\phi(t_1)\right \rangle_\Sigma\]
is a homomorphism. Consequently, the expression in \eqref{inner_phi_sum} is equal to 
\begin{equation*}
\#\sA\cdot \#\{\bfx\in \ImageModSquares^{B[2]}~~\colon~~ \eta_1(\bfx)=t_1 \textup{ and }\Upsilon_\phi(\bfx)=\left \langle v_0,\phi(t_1)\right \rangle_\Sigma \textup{ for all }\phi \in \sA\}.
\end{equation*}
By \Cref{quad_form_phi_upsilon_lemma}, the above expression is equal to 
\begin{align} \label{slowly_moving_in_right_direction}
&  \#\sA\cdot \#\{\bfx\in \ker(\eta_1)~~\colon~~   \Upsilon_\phi(\bfx)=
  \Upsilon_{\phi}(\bfx(t_1)) + \left \langle v_0,\phi(t_1)\right \rangle_\Sigma\textup{ for all }\phi \in \sA\}\nonumber\\
& = \#\sA\cdot \#\{\bfx\in \ker(\eta_1)~~\colon~~   \Psi(\bfx)(\phi)=\Upsilon_{\phi}(\bfx(t_1))+
  \left \langle v_0,\phi(t_1)\right \rangle_\Sigma  \quad\textup{ for all }\phi \in \sA\}.
\end{align}
Since $\Psi$ is a homomorphism, and so is the map $\lambda_{t_1,\bfx(t_1),v_0}\colon\sA\rightarrow \F_2$,
it follows that the quantity \eqref{slowly_moving_in_right_direction} is equal to 
\begin{eqnarray*}
\#\sA\cdot \#\{\bfx\in \ker(\eta_1)~~\colon~~   \Psi(\bfx)=\lambda_{t_1,\bfx(t_1),v_0} \}&=&\#\sA\cdot \#\ker(\Psi)\cdot  \mathbbm{1}_{\lambda_{t_1,\bfx(t_1),v_0}\in \im(\Psi)}.
\end{eqnarray*}
Now, we have $\#\ker(\Psi)=\#\ker(\eta_1)/\#\im(\Psi)$ and $\#\im(\Psi)=\#\sA/\#\im(\Psi)^\circ$.
Combining the above, we have
\begin{eqnarray*}
 \sum_{\substack{\bfx\in \ImageModSquares^{B[2]}\\ \eta_1(\bfx)=t_1}}\sum_{\phi \in \sA}(-1)^{ \Upsilon_\phi(\bfx)+\left \langle v_0,\phi(t_1)\right \rangle_\Sigma}&=&\#\ker(\eta_1)\cdot \#\im(\Psi)^\circ\cdot \mathbbm{1}_{\lambda_{t_1,\bfx(t_1),v_0}\in \im(\Psi)}\\
 &=& \frac{\#\ImageModSquares^{2^{2rg}}\cdot \#\im(\Psi)^\circ}{\#T_1}\cdot \mathbbm{1}_{\lambda_{t_1,\bfx(t_1),v_0}\in \im(\Psi)}.
\end{eqnarray*}
 Substituting into \eqref{after_summing_over_phi}, we obtain
\begin{eqnarray*}
  \sum_{\phi \in \sA} \sY_{v_0,z_0}(\phi)&=& \#\im(\Psi)^\circ\cdot  \frac{1}{\#T_1}\sum_{t_1\in T_1} (-1)^{\left \langle t_1,z_0\right \rangle_\Sigma} \cdot \mathbbm{1}_{\lambda_{t_1,\bfx(t_1),v_0}\in \im(\Psi)}.
\end{eqnarray*}
Recalling from \eqref{contrib_main_W_sum} the definition of $\sY(W)$, we see that the left hand side of \eqref{eq:formula_in_statement11} is equal to 
\begin{eqnarray*}
\frac{  \#\im(\Psi)^\circ}{\#Z_0}\sum_{(v_0,z_0)\in V_0\times Z_0}(-1)^{\left \langle v_0,z_0\right \rangle_\Sigma}
  \cdot\frac{1}{\#T_1}\sum_{t_1\in T_1} (-1)^{\left \langle t_1,z_0\right \rangle_\Sigma} \cdot \mathbbm{1}_{\lambda_{t_1,\bfx(t_1),v_0}\in \im(\Psi)}    \\
= \frac{\#\im(\Psi)^\circ}{ \#(\ImageModSquares\otimes B[2])}\sum_{v_0 \in V_0} \sum_{t_1\in \ImageModSquares\otimes B[2]} \mathbbm{1}_{\lambda_{t_1,\bfx(t_1),v_0}\in \im(\Psi)}\frac{1}{\#Z_0}\sum_{z_0\in Z_0}(-1)^{\left \langle v_0+t_1,z_0\right \rangle_\Sigma} \\
= \frac{ \#\im(\Psi)^\circ}{ \#(\ImageModSquares\otimes B[2])}\sum_{v_0 \in V_0} \sum_{t_1\in \ImageModSquares\otimes B[2]} \mathbbm{1}_{\lambda_{t_1,\bfx(t_1),v_0}\in \im(\Psi)}\mathbbm{1}_{v_0+t_1\in Z_0}, \phantom{lots of spacea}
\end{eqnarray*}
where for the last equality we are using that $Z_0$ is its own orthogonal complement
with respect to the pairing $\left \langle ~,~\right \rangle_\Sigma$.
\end{proof} 
 
\begin{lemma} \label{commuting_ref}
For all $\phi \in \im(\Psi)^\circ$ and $\sigma \in \Gamma$, we have $\phi \sigma= \sigma\phi$.
\end{lemma}

\begin{proof}   
Let  $\phi \in \im(\Psi)^\circ$. By definition, we have 
\[
\im(\Psi)^\circ=\{\phi \in \sA~~\colon~~ \Upsilon_\phi(\bfx)=0~~\textup{ for all }\bfx\in \ker(\eta_1)\}.
\]
Recall that $\ker(\eta_1)$ is the set of elements $\bfx\in \ImageModSquares^{B[2]}$ for which $\sum_{u\in B[2]}x_u\otimes u=0$.
 
Fix $h\in \ImageModSquares$ and $u_1, u_2\in B[2]$. Suppose that $\{u_1,u_2,u_1+u_2\}$ has size $3$. Define $\bfx=(x_u)_u\in \ImageModSquares^{B[2]}$ by taking $x_{u_1}=x_{u_2}=x_{u_1+u_2}=h,$ and setting $x_u=0$ otherwise.
Note that $\bfx$ is in $\ker(\eta_1)$. For $\phi \in \sA$ we have
\begin{eqnarray*}
\Upsilon_\phi(\bfx)&\stackrel{\textup{def}}{=}& \sum_{\{u,u'\}\subseteq B[2]}e\big(\alpha(x_u)u,\alpha(x_{u'})\phi(u')\big)+\sum_{u\in B[2] }e\big(\gamma(\sigma_{x_u})u,\phi(u)\big)\\
&=& e\big(\alpha(h)u_2,\alpha(h)\phi(u_1)\big)+e\big(u_1,\gamma^{(r)}(\sigma_{h})\phi(u_2)\big)+e\big(u_2,\gamma^{(r)}(h)\phi(u_1)\big)\\
&\stackrel{}{=}&e\big(u_1,(\phi \gamma^{(r)}(\sigma_{h})-\gamma^{(r)}(\sigma_{h})\phi)u_2\big),
\end{eqnarray*}
where the final equality follows from \Cref{gamma_and_weil_lemma} and the fact that
the pairing $(u,u')\mapsto e(u,\phi(u'))$ is symmetric. Since we have $\phi \in \im(\Psi)^\circ$, we deduce that
\begin{equation} \label{eq:gamma_phi_commute}
e(u_1,(\phi \gamma^{(r)}(\sigma_h)-\gamma^{(r)}(\sigma_h)\phi)u_2)=0.
\end{equation}
This equality also holds, for any $h$, when $\{u_1,u_2,u_1+u_2\}$ has size smaller than $3$. Indeed, if one of $u_1, u_2$ is zero then it is trivial, while the case $u_1=u_2$ again follows by combining   \Cref{gamma_and_weil_lemma} with the fact that
the pairing $(u,u')\mapsto e(u,\phi(u'))$ is symmetric. Thus \eqref{eq:gamma_phi_commute} holds for all 
  $u_1,u_2\in B[2]$ and $h\in \ImageModSquares$, giving the result.  
\end{proof}

    \begin{notation}
 Denote by $\sA^\Gamma$ the subspace of $\sA$ consisting of elements commuting with all elements of $\Gamma$. 
 \end{notation}

Let $\overline{\phi}$ be a $\Q$-endomorphism of $B$ that is invariant under the Rosati involution. Via Section \ref{tate_quad_form_subsec}, we associate to $\overline{\phi}$ a theta group $\cH_{\overline{\phi}}$ for $B[2]$, with associated quadratic form \[q_{\overline{\phi},\Sigma}\colon\oplus_{v\in \Sigma}H^1(\Q_v,B[2])\longrightarrow \tfrac{1}{2}\Z/\Z\cong \F_2.\]
The bilinear pairing associated to $q_{\overline{\phi},\Sigma}$ sends $(f,g)$ to $\left \langle f, \phi(g) \right \rangle_\Sigma$. 
As explained in Section \ref{tate_quad_form_subsec} we have $q_{\overline{\phi},\Sigma}(Z_0)=0$ and $q_{\overline{\phi},\Sigma}(V_0)=0$.

\begin{lemma} \label{lem:identifying_as_quad_form_new}
Suppose that $\phi \in \sA^\Gamma$ lifts to a
$\Q$-endomorphism $\overline{\phi}$ of $B$ that is invariant under the Rosati involution.
Then, for any $t_1\in T_1$, $v_0\in V_0$, and $\bfx\in \ImageModSquares^{B[2]}$ lifting $t_1$,  we have  
 \begin{equation*} \label{lambda_and_q_identity}
 \lambda_{t_1,\bfx,v_0}(\phi)=q_{\overline{\phi},\Sigma}(v_0+t_1).
 \end{equation*}
 \end{lemma}

 \begin{proof}
 Recall that, by definition, we have 
$\lambda_{t_1,\bfx,v_0}(\phi)=\left \langle v_0,\phi(t_1)\right \rangle_\Sigma+\Upsilon_\phi(\bfx).$
We now compute
 \[q_{\overline{\phi},\Sigma}(v_0+t_1)=q_{\overline{\phi},\Sigma}(v_0)+q_{\overline{\phi},\Sigma}(t_1)+\left \langle v_0,\phi(t_1)\right \rangle_\Sigma=q_{\overline{\phi},\Sigma}(t_1)+\left \langle v_0,\phi(t_1)\right \rangle_\Sigma.\]
Write $\bfx=(x_u)_{u\in B[2]}$. Then 
$t_1=\eta_1(\bfx)=\sum_{u\in B[2]}x_u\otimes u.$ Thus we have
\begin{eqnarray*}
q_{\overline{\phi},\Sigma}(t_1)&=& \sum_{u\in B[2]}q_{\overline{\phi},\Sigma}(x_u\otimes u)+\sum_{\{u,u'\}\subseteq B[2]}\left \langle x_u\otimes u ,x_{u'}\otimes \phi(u')\right \rangle_\Sigma.
 \end{eqnarray*}
Recall that we have $\left \langle ~,~\right \rangle_\Sigma=H_\Sigma\otimes e$, where $H_\Sigma$ is the Hilbert symbol pairing of Section \ref{sec:local_tate}.   
 Thus 
 \[\left \langle x_u\otimes u ,x_{u'}\otimes \phi(u')\right \rangle_\Sigma=H_\Sigma(x_u,x_{u'})e_2(u,\phi(u'))=\alpha(x_u)\alpha(x_{u'})e_2(u,\phi(u')).\]
Here, for the final  equality, recall that $x_u, x_{u'} \in \ImageModSquares$ are the classes of some $m,m'\in \mathcal{F}$, and that $H_\Sigma(x_u,x_{u'})$ is equal to $\sum_{v\in \Sigma}H_v(\psi_m,\psi_{m'})$. As elements of $\mathcal{F}$, $m$ and $m'$ are positive and are units locally at all primes $p\in \Sigma$. Thus  $H_v(\psi_m,\psi_{m'})$ is zero save possibly at $v=2$; at $v=2$ it is equal to $1$ if and only if both $m$ and $m'$ are congruent to $-1$ modulo $4$, that is, if and only if $\alpha(x_u)\alpha(x_{u'})=1$.

 Comparing with the expression for $\lambda_{t_1,\bfx,v_0}(\phi)$, to prove the result it is enough to show that we have $q_{\overline{\phi},\Sigma}(x\otimes u)=e(\gamma^{(r)}(\sigma_x)u,\phi(u))$
  for all  $x\in \ImageModSquares$, $u\in B[2]$.
  Fix $m\in \mathcal{F}$ lifting $x\in \ImageModSquares$, and take $u\in B[2]$. Then we have
\[
  q_{\overline{\phi},\Sigma}(x\otimes u)=q_{\overline{\phi},\Sigma}(\psi_m\cup u)=
H_\Sigma(\delta^{(r)}(u)\cup \phi(u),\psi_m)=
\sum_{p\mid m}e(\delta^{(r)}(u)(\sigma_p),\phi(u))=e(\gamma^{(r)}(\sigma_x)u,\phi(u)),\]
where the second equality is a consequence of \Cref{boundary_map_theta}(ii),
the third equality follows from reciprocity, and the final equality follows from the definitions of $\gamma^{(r)}$ and $\sigma_x$.  
 \end{proof}
 
\begin{proposition} \label{prop:when_lifts_to_endomorphism}
Suppose that every element of $\sA^\Gamma$  lifts to a
$\Q$-endomorphism $\overline{\phi}$ of $B$ that is invariant under the Rosati involution. Then we have $\im(\Psi)^\circ=\sA^\Gamma$.  For any $t_1\in T_1$, $v_0\in V_0$, and $\bfx\in \ImageModSquares^{B[2]}$ lifting $t_1$, we have $\lambda_{t_1,\bfx,v_0}\in \im(\Psi)$ if and only if 
$q_{\overline{\phi},\Sigma}(v_0+t_1)=0$ for all $\phi \in \sA^\Gamma.$
 \end{proposition}
 
 \begin{proof}
To show that $\im(\Psi)^\circ=\sA^\Gamma,$ note that by \Cref{commuting_ref} we have $\im(\Psi)^\circ\subseteq \sA^\Gamma$.
To prove the reverse inclusion, take $\phi \in \sA^\Gamma$ and  $\bfy\in \ker(\eta_1)$.
Let $\overline{\phi}$ be a  $\Q$-endomorphism  of $B$ that is invariant under the Rosati involution and that induces $\phi$ on $B[2]$.
Then by  Lemma \ref{lem:identifying_as_quad_form_new} we have 
\[\Upsilon_\phi(\bfy)=\left \langle v_0,\phi(\eta_1(\bfy))\right \rangle_\Sigma+q_{\overline{\phi}}(v_0+\eta_1(\bfy))=q_{\overline{\phi}}(v_0)=0.\]
Since this holds for all $\bfy \in  \ker(\eta_1)$, we have $\phi \in \im(\Psi)^\circ$.

 Finally, taking annihilators in the equality $\im(\Psi)^\circ=\sA^\Gamma$ gives $\im(\Psi)=(\sA^\Gamma)^\circ$. The remaining claim in the statement is an immediate consequence of this and Lemma \ref{lem:identifying_as_quad_form_new}. 
 \end{proof}
 
  Recall that we are viewing both $V_0$ and $\ImageModSquares \otimes B[2]$ inside $\oplus_{v\in \Sigma}H^1(\Q_v,B[2])$, as explained in Notation \ref{notat:image_mod_sq_interpretation}.  In this ambient space, $V_0$ and $\ImageModSquares \otimes B[2]$ have trivial intersection. In particular, $\fA=V_0 \oplus (\ImageModSquares\otimes B[2])$ is naturally a subspace of $\oplus_{v\in \Sigma}H^1(\Q_v,B[2])$. 
 
 \begin{corollary} \label{cor:when_all_lifts}
  Suppose that every element $\phi$ of $\sA^\Gamma$  lifts to a   $\Q$-endomorphism $\overline{\phi}$ of $B$ that is invariant under the Rosati involution. Then we have
\begin{eqnarray*}
\sum_{\substack{W\subseteq B[2]^2\textup{ max iso.}\\\pi_1(W)=B[2]}}\sY(W)=\frac{ \#\sA^\Gamma\#(\fA\cap Z_0) }{\#(\ImageModSquares\otimes B[2])}.
\end{eqnarray*}
 \end{corollary}
 
 \begin{proof}
 An immediate consequence of Lemma \ref{lem:sum_over_phi} and Proposition \ref{prop:when_lifts_to_endomorphism}   is that the sum on the left hand side equals
 \begin{eqnarray*}
\frac{ \#\sA^\Gamma}{\#(\ImageModSquares\otimes B[2])}\#\{a\in \mathfrak{A}\cap Z_0~~\colon~~ q_{\overline{\phi},\Sigma}(a)=0 \textup{ for all }\phi \in \sA^\Gamma\}.
\end{eqnarray*}
As noted above, we have $q_{\overline{\phi},\Sigma}(Z_0)=0$ (see Section \ref{tate_quad_form_subsec}), giving the result.
 \end{proof}

\begin{lemma} \label{lem:what_is_this_space?}
We have 
\[\frac{\#(\fA\cap Z_0)}{ \#(\ImageModSquares\otimes B[2])}=\# (V_0\cap (\ImageModSquares\otimes B[2])^\perp \cap Z_0),\]
where the orthogonal complement is taken with respect to the local Tate pairing.
\end{lemma}
 
 \begin{proof}
We have
\begin{equation} \label{what_is_this_intersection_size}
\#(\fA\cap Z_0)=\frac{\#\fA\#Z_0}{\#(\fA+Z_0)}=\frac{\#\fA\#Z_0 \#(\fA^\perp \cap Z_0)}{\prod_{v\in \Sigma}\#H^1(\Q_v,B[2])},\end{equation}
where for the last equality we are using that $Z_0$ is its own orthogonal complement. Now $\#\fA=\#V_0 \#(\ImageModSquares\otimes B[2])$.
Moreover, we have $\prod_{v\in \Sigma}\#H^1(\Q_v,B[2])=\#V_0\#Z_0$, since both $V_0$ and $Z_0$
are maximal isotropic for $\left \langle~,~\right \rangle_\Sigma$ (for $Z_0$ see \cite[Proposition 4.10]{MR2833483}, while for $V_0$ this is a consequence of Poitou--Tate duality, see \cite[Theorem I.4.10]{MR2261462}).  
Substituting into \eqref{what_is_this_intersection_size} we deduce that
\[\#(\fA\cap Z_0)=\# (\ImageModSquares\otimes B[2])\cdot  \# (\fA^\perp \cap Z_0).\]
Since $V_0$ is its own orthogonal complement under $\left \langle~,~\right \rangle_\Sigma$,
we have $\fA^\perp=V_0\cap (\ImageModSquares\otimes B[2])^\perp$.  
 \end{proof}
 
  \begin{theorem} \label{thm:main_can_evaluate_sums_thm}
  Suppose that every element of $\sA^\Gamma$  lifts to an  endomorphism of $B$,
  defined over $\mathbb{Q}$, that is invariant under the Rosati involution. Then we have
\begin{eqnarray*}
\sum_{\substack{W\subseteq B[2]^2\textup{ max iso.}\\\pi_1(W)=B[2]}}\sY(W)=\#\sA^\Gamma \cdot \#\cS_{b}^r.
\end{eqnarray*}
 \end{theorem}
 
 \begin{proof}
 This is an immediate consequence of Corollary \ref{cor:when_all_lifts}, Lemma \ref{lem:what_is_this_space?} and  Lemma \ref{lem:alternative_syst_sub_desc}. 
 \end{proof}
 
\subsection{Proof of Theorem \ref{thm:main_term_inj_contribution}}\label{sec:main_term_inj_proof}
Recall that $B[2]=A[2]^r=A[2]\otimes \F_2^r$ and that, under this identification,
$\sigma \in \Gamma$ operates on $B[2]$ as $\sigma\otimes 1$; this is our convention for viewing $\Gamma$ as a subgroup of $\End(B[2])$. 
   
\begin{notation}
Denote by $\Sym_r(\F_2)$ the vector space of $r\times r$ symmetric matrices over $\F_2$.
Recall that, in this section, we are denoting by $e$ the Weil pairing on $B[2]$, which by convention takes values in $\F_2$.
When we wish to refer to the corresponding pairing on $A[2]$, we will denote it $e'$. We denote by $P$ the standard `dot-product' pairing on $\F_2^r$.
\end{notation}

\begin{lemma} \label{lem:desc_of_a_gamma}
Suppose that Condition \ref{assumption:simple_gamma_mod} is satisfied. Then we have 
   \begin{equation*} \label{desc_of_a_gamma}
 \sA^\Gamma=1\otimes \Sym_r(\F_2).
 \end{equation*}
  \end{lemma}
  
  \begin{proof}
  Since $\End_\Gamma(A[2])=\F_2$  it follows that $\End_{\Gamma}(B[2])=1\otimes \End(\F_2^r)$. Take $\phi_0\in \End(\F_2^r)$. Then $1\otimes \phi_0 \in \sA$ if and only if $e(u,(1\otimes \phi_0)u)=0$ for all $u\in A[2]\otimes \F_2^r$. We have $e=e'\otimes P$, from which we see that $1\otimes \phi_0$ lies in $\sA$ if and only if $\phi_0$ is self-adjoint with respect to  $P$. That is, if and only if $\phi_0$ is represented by a symmetric matrix in the standard basis for $\F_2^r$. 
  \end{proof}
  
\begin{remark}
 Note that if $\phi=1\otimes \phi_0 \in \sA^\Gamma$, then the pairing $e(-,\phi(-))$ is given by $e'\otimes P_{\phi_0}$, where $P_{\phi_0}$ is the symmetric pairing on $\F_2^r$ given by $(u,w)\mapsto P(u,\phi_0(w))$. 
 \end{remark}

\begin{proof}[Proof of Theorem \ref{thm:main_term_inj_contribution}] 
Since \Cref{assumption:simple_gamma_mod} is satisfied,
it follows from Proposition \ref{main_invariant_gamma_computation_prop} and Lemma
\ref{simple_subspace_lemma} that the only subspaces $W$ contributing to the left hand side of \eqref {main_term_inj_contribution} are those for which $\pi_1(W)=B[2]$.

 Let $\phi \in  \sA^\Gamma$. By Lemma \ref{lem:desc_of_a_gamma}, $\phi$
 has the form $1\otimes \phi_0$ for some $\phi_0\in \Sym_r(\F_2)$.
From this we see that we can lift $\phi$ to a $\Q$-endomorphism $\overline{\phi}$ of $B$, invariant under the Rosati involution.
Indeed, we can take $\overline{\phi}$ to be any endomorphism of $B=A^r$ given by a lift of $\phi_0$ to a symmetric $r\times r$ matrix with coefficients in $\Z$. 

From the above discussion we see, in particular, that the hypothesis of Theorem \ref{thm:main_can_evaluate_sums_thm} is satisfied. The result now follows from Theorem \ref{thm:main_can_evaluate_sums_thm} and Lemma \ref{lem:desc_of_a_gamma}. 
 \end{proof}
 
 \subsection{Distribution of Selmer groups}\label{sec:mainproof} 
  
\begin{theorem} \label{thm:inj_moments_main}
Let $L/\Q$ be a finite Galois extension, let 
$b\in \idelesmodsquaresset{\Q}{\Sigma}$ be such that $\FbL$ is non-empty. Suppose that \Cref{assumption:simple_gamma_mod} holds.
Then for all $r\in \Z_{\geq 0}$ we have 
\[\lim_{X\rightarrow \infty} \frac{1}{\#\FbL(X)}\sum_{d\in\FbL(X)}
  \#\Inj\bigg(\F_2^r,\frac{\Sel_2(A_{d}/\Q)}{\delta_{d}(A[2])+\cS_{b}}\bigg)  =2^{\frac{r(r+1)}{2}}.
\]
\end{theorem}
\begin{proof}
  Recall that a choice of $b\in \idelesmodsquaresset{\Q}{\Sigma}$ is
  equivalent to a choice of a pair consisting of a (possibly negative) divisor $d_0$ of $N$
  and a class $\bar{c}\in \left(\Z/4N\Z\right)^{\times}/\left(\Z/4N\Z\right)^{\times 2}$; and that moreover $\FbL$ is non-empty
  if and only if we have $\bar{c}\in \bS$. The result is therefore
  an immediate consequence of \Cref{thm:inj_moments} and
  \Cref{thm:main_term_inj_contribution}, applied to lifts $c\in S$ of $\bar{c}$.
\end{proof}

\begin{lemma}\label{lem:hom_moments}
  With the assumptions and notation of \Cref{thm:inj_moments_main},
we have, for all $r\in \Z_{\geq 0}$, 
\[
  \lim_{X\rightarrow \infty} \frac{1}{\#\FbL(X)}\sum_{d \in \FbL(X)}
  \#\Hom\bigg(\frac{\Sel_2(A_{d}/\Q)}{\delta_{d}(A[2])+\cS_{b}},\F_2^r\bigg) = \prod_{j=1}^r (1+2^j).
\]
\end{lemma}
\begin{proof}
  Let $r\in \Z_{\geq 0}$.
  For a finite abelian group $D$, let $D^\lor$ denote the Pontryagin dual $\Hom(D,\Q/\Z)$. For
  two abelian groups $D_1$, $D_2$, we have $\Inj(D_1,D_2) = \Surj(D_2^\lor,D_1^\lor)$, where
  $\Surj$ denotes the set of surjective homomorphisms. Moreover, for $i=1,2$, $D_i^\lor$ is
  (non-canonically) isomorphic to $D_i$, so that we have $\#\Inj(D_1,D_2) = \#\Surj(D_2,D_1)$.
  In particular, \Cref{thm:inj_moments_main} also gives the average number of surjections
  from the $2$-Selmer group onto $\F_2^r$.

  For $k\in \Z_{\geq 0}$, let $\fn(k,r)$
  denote the number of $k$-dimensional subspaces of $\F_2^r$. Then, for all
  finite dimensional $\F_2$-vector spaces $D$, we have
  \[
    \#\Hom(D,\F_2^r) = \sum_{k=0}^r \fn(k,r)\#\Surj(D,\F_2^k),
  \]
  and an analogous formula for the sought-for average number of homomorphisms from the $2$-Selmer group to $\F_2^r$.
  It is well-known that we have $\fn(k,r) = \prod_{i=1}^k\frac{1-2^{r-i+1}}{1-2^i}$,
  so by \Cref{thm:inj_moments_main} the sought-for average is equal to
  \[
    \sum_{k=0}^r \left(\prod_{i=1}^k\Big(\frac{1-2^{r-i+1}}{1-2^i}\Big)2^{k(k+1)/2}\right).
  \]
  The result follows from \cite{MR2445243}*{Theorem 348}, applied with $a=1$ and $x=2$.
\end{proof}

We now deduce Theorem \ref{thm:intro_distr}.
\begin{theorem}\label{thm:main_distr}
  Let $L/\Q$ be a finite Galois extension,
  and $b\in \idelesmodsquaresset{\Q}{\Sigma}$ be such that $\FbL$ is non-empty.
  Suppose that \Cref{assumption:simple_gamma_mod} is satisfied. For $r\in \Z_{\geq 0}$, let 
  \[
    \alpha(r) = \prod_{j\geq 1}(1+2^{-j})^{-1}\prod_{j=1}^{r}\frac{2}{2^{j}-1}.
  \]
  Let $n_b=\dim_{\F_2}\cS_{b}$. Then there exists $m_b\in \{0,1\}$ such that
  for all  $d\in \FbL$, we have $\dim\Sel_2(A_{d}/\Q)\equiv n_b+m_b\pmod 2$, and
  for all $r\in\Z_{\geq 0}$ we have
  \[
    \lim_{X\to \infty} \frac{\#\{d\in \FbL(X): \dim\Sel_2(A_{d}/\Q)=2g+n_b+m_b+2r\}}{\#\FbL(X)}=\alpha(m_b+2r).
  \]
\end{theorem}
\begin{proof}
  The existence of $m_b$ follows from \Cref{cor:2-Selmer rank}.
  The final assertion then follows by combining this constancy of $2$-Selmer
  rank parity, \Cref{lem:hom_moments}, Remark \ref{rmrk:direct_sum}, and \cite{MR1292115}*{Lemma 19}.
\end{proof}



%

\section{Hasse principle for genus \texorpdfstring{$1$}{1} hyperelliptic curves} \label{sec:HassePrinc}
Let $f(x)\in \Q[x]$ be a separable monic quartic polynomial.
For $d\in \Q^\times$, let $C_d$ be the hyperelliptic curve with affine model $dy^2 = f(x)$.
Let $A_d$ be the Jacobian of $C_d$, so that $C_d$ defines a $2$-covering of $A_d$,
and let $[C_d]$ denote the corresponding class is $H^1(\Q,A_d[2])$.
When $d$ is a square in $\Q^\times$, we drop it from the subscripts.

Since $f$ is monic, the curve $C$ has a rational point, so the corresponding
class $[C]\in H^1(\Q,A[2])$ lies in the image of the coboundary map
$\delta\colon A(\Q)/2A(\Q)\to H^1(\Q,A[2])$.
One checks readily that under the identification 
$H^1(\Q,A[2])\cong H^1(\Q,A_d[2])$ of Section \ref{sec:SelmerBackg},
the class $[C]\in H^1(\Q,A[2])$ gets identified with
$[C_d]\in H^1(\Q,A_d[2])$, so that $C_d$ has a rational point if and only
if $[C]$ lies in the image of $\delta_d$, and it has a point everywhere
locally if and only if $[C]$ lies in $\Sel_2(A_d/\Q)$.

Let $L$ be the splitting field of $f$.
Assume for the rest of the subsection that $A/\Q$ has full rational $2$-torsion,
and that Condition \ref{assumption:simple_gamma_mod} is satisfied for $A$.
Define 
$$
\cF'=\{d\in \Z_{\neq 0} : d\text{ is squarefree and }[C]\in \Sel_2(A_d/\Q)\}.
$$
The condition that $A$ have full rational $2$-torsion implies that the Galois
group of $L/\Q$ is contained in the Klein $4$-group.
By \Cref{prop:KummerIntersection}, the condition that $[C]\in \Sel_2(A_d/\Q)$
implies that all odd primes $p|d$ at which $A$ has good reduction split completely in $L$, so that
$\cF'$ is a union of families of the form $\FbL$, 
as defined in the Introduction. 
In particular, the following is an immediate consequence of \Cref{prop:asymptotic_cF}.
In the case where $f(x)$ is irreducible, i.e. when $\Gal(L/\Q)$ is the full Klein $4$-group,
this is also proven in \cite[Corollary 1.5]{Novak}.

\begin{proposition}
  There is a positive constant $c_{\cF'}$ such that 
  we have
  \[
    \#\{d\in \cF' : |d|<X\} \sim c_{\cF'}X(\log X)^{\frac{1}{[L:\Q]}-1}.
  \]
\end{proposition}

We now prove Corollary \ref{cor:intro_hyperell} from the Introduction.

\begin{theorem}\label{thm:globalPointsLowerBound}
  Let $\Sigma$ be a finite set of places of $\Q$ containing $2$, $\infty$,
  and all places of bad reduction for $A$, and let $b\in \idelesmodsquaresset{\Q}{\Sigma}$
  be such that we have $\emptyset \neq \FbL\subset \cF'$ and such that for all $d\in \FbL$ the
  $2^\infty$-Selmer corank $\rk_{2}(A_d/\Q)$ is odd. Suppose that the systematic
  subspace $\cS_b$, as defined in \Cref{def:syst_subspace}, is $1$-dimensional.
  Assume that for all $d\in \FbL$ for which $\rk_{2}(A_d/\Q)=1$, the $2$-part of the Shafarevich--Tate group
  of $A_d$ over $\Q$ is finite. Then we have
  $$
    \liminf_{X\to \infty}\frac{\#\{d\in \FbL:|d|<X, C_d(\Q)\neq \emptyset\}}
    {\#\{d\in \FbL:|d|<X\}} > 0.4194.
  $$
\end{theorem}
\begin{proof}
  By Theorem \ref{thm:main_distr}, the proportion among $d\in \FbL$ of those for which one has $\dim\Sel_2(A_d/\Q) = 3$ is $\alpha(0)\approx 0.4194\ldots$.
  Since the $2$-part of the Shafarevich--Tate group of all such $d$ is assumed to be finite, the non-degeneracy of the Cassels--Tate pairing forces its order to be a square, and hence trivial, so that we have $\Sel_2(A_d/\Q)=\delta_d(A_d/\Q)$.
In particular, for all such $d$, the class of $C$ is in the image of $\delta_d$, and hence $C_d$ has a rational point.
\end{proof}

\section{Hasse principle for Kummer varieties}\label{sec:kummer}
Let $A/\Q$ be a principally polarised abelian variety with full rational $2$-torsion.
Let $g=\dim A$. Fix a non-zero class $\zeta \in H^1(\Q,A[2])$ and denote by $Y_\zeta$ the corresponding $2$-covering of $A$.
Thus $Y_\zeta$ is a torsor under $A$ equipped with an involution $\iota$ compatible with multiplication by $-1$ on $A$.
For a field extension $F/\Q$, we have $Y_\zeta(F)\neq \emptyset$ if and only if the image of $\zeta$ in $H^1(F,A)$ is zero. 

For a squarefree integer $d$, denote by $Y_{\zeta,d}$ the corresponding quadratic twist of $Y_\zeta$.
For a place $v$ of $\Q$ and $b_v\in \Q_v^{\times}/\Q_v^{\times 2}$,
we denote by $Y_{\zeta,b_v}$ the quadratic twist by $b_v$ of the base-change of $Y_\zeta$ to $\Q_v$.
Noting that $H^1(\Q,A[2])=\Hom_{\cts}(G_\Q,A[2])$, we denote by $L/\Q$ the fixed field of the kernel of $\zeta$.
Let $\Sigma$ be a finite set of places of $\Q$, containing $2$, $\infty$, all places at which $A$ has bad reduction, and all places that ramify in $L/\Q$. 

Recall from Section \ref{sec:GalAction4Torsion} that we can naturally view $\Gal(L(A[4])/L)$ as an additive subgroup $\Gamma$ of $\End(A[2])$.
We remind the reader that Condition \ref{assumption:simple_gamma_mod} is the assertion
that $\End_{\Gamma}(A[2])=\F_2$ and that the only subspaces of $A[2]$ stable under $\Gamma$ are $0$ and $A[2]$. 

For  $b=(b_v)_v\in \prod_{v\in \Sigma}\Q_v^{\times}/\Q_v^{\times 2}$, recall from Definition \ref{def:syst_subspace} that the systematic subspace $\cS_b$ is defined as the subspace of $H^1(L/\Q,A[2])$ consisting of elements which, for all $v\in \Sigma$, satisfy the Selmer conditions associated to $A_{b_v}$. 

Finally, recall from Section \ref{sec:ParityRanks} the definition of the local functions $\kappa_v\colon \Q_v^{\times}/\Q_v^{\times 2}\to \Z/2\Z$. Given $b=(b_v)_v\in \prod_{v\in \Sigma}\Q_v^{\times}/\Q_v^{\times 2}$, we define $\kappa(b) \in \Z/2\Z$ by  
\[\kappa(b)=\rk_2(A/\Q)+\sum_{v\in \Sigma}\kappa_v(b_v) \pmod 2.\]
By Theorem \ref{thm:2infty-Selmer rank}, given any squarefree integer $d$ whose class in $\prod_{v\in \Sigma}\Q_v^{\times}/\Q_v^{\times 2}$ is equal to $b$, we have 
$\rk_2(A_d/\Q)\equiv \dim\Sel_2(A_d/\Q)\equiv \kappa(b) \pmod 2.$ 

We will sometimes impose the following condition.
\begin{condition} \label{cond_E_new} 
There exists $b=(b_v)_{v\in \Sigma}\in \prod_{v\in \Sigma}\Q_v^{\times}/\Q_v^{\times 2}$ such that (1): we have $\cS_b=\{0,\zeta\}$;   and (2): the parity $\kappa(b)$ is odd.
\end{condition}

\begin{remark} \label{rmk:condition_E}
The condition that $\cS_b$ be generated by $\zeta$ can be rephrased as follows. The condition that $\zeta\in \cS_b$ is equivalent
to the condition that for all $v\in \Sigma$, we have $Y_{\zeta,b_v}(\Q_v)\neq \emptyset$; while the condition that 
no element $\zeta' \in H^1(L/\Q,A[2])\setminus \{0,\zeta\}$  lies in $\cS_b$ is equivalent to the condition that
for each class $\zeta' \in H^1(L/\Q,A[2])$ with $\zeta'\notin \{0,\zeta\}$, there is a place $v\in \Sigma$ such that $Y_{\zeta',b_v}(\Q_v)=\emptyset$.

Phrased in this way, we see that the existence of an element $b$ satisfying part (1) of Condition \ref{cond_E_new} is a version of Condition A of \cite{MR2183385}.
\end{remark}

\begin{lemma}\label{lem:FbL_nonempty}
  Let $b=(b_v)_v\in \idelesmodsquaresset{\Q}{\Sigma}$ and suppose that we have $\zeta\in \cS_b$.
  Then the family denoted $\FbL$ in Theorem \ref{thm:intro_distr} is non-empty.
\end{lemma}
\begin{proof}
For $v\in \Sigma$, denote by $\psi_{b_v}$ the quadratic character $G_{\Q_v}\to \F_2$ associated to $b_v$.
By \cite[Proposition 5.2]{MR4934780}, it suffices to show that for every
quadratic character $\chi\colon \Gal(L/\Q)\to \F_2$, we have 
$\sum_{v\in \Sigma}\inv_v(\psi_{b_v}\cup \chi)=0$.

Fix a character $\chi\colon \Gal(L/\Q)\to \F_2$.  Since $L$ is defined as the fixed field
of the kernel of $\zeta$, the map $\Gal(L/\Q)\to A[2]$ sending $\sigma\in \Gal(L/\Q)$ to
$\zeta(\sigma)$ is an injection. By non-degeneracy of the Weil pairing on $A[2]$,
there is $x\in A[2]$ such that $\chi=x \cup \zeta$. We then have 
\begin{eqnarray*}
\sum_{v\in \Sigma}\inv_v(\psi_{b_v}\cup \chi)=\sum_{v\in \Sigma}\inv_v\big((\psi_{b_v}\cup x)\cup \zeta\big)=
\sum_{v\in \Sigma}\inv_v(\delta_{b_v,v}(x)\cup \zeta)+\sum_{v\in \Sigma}\inv_v(\delta_v(x) \cup \zeta),
\end{eqnarray*}
where the second equality follows from Lemma \ref{lem:twisted_boundary_map}.
Let $v\in \Sigma$. By the assumption that $Y_{\zeta,b_v}(\Q_v)\neq \emptyset$,
both $\delta_{b_v,v}(x)$ and $\res_v(\zeta)$ lie in $\sS_{b_v,v}$.
Since $\sS_{b_v,v}$ is isotropic for the local Tate pairing, we have
$\inv_v(\delta_{b_v,v}(x)\cup \zeta)=0$. Thus we have
$\sum_{v\in \Sigma}\inv_v(\delta_{b_v,v}(x)\cup \zeta)=0$. Further, by reciprocity for the Brauer group, we have 
\[\sum_{v\in \Sigma}\inv_v(\delta_v(x) \cup \zeta)=\sum_{v\notin \Sigma}\inv_v(\delta_v(x) \cup \zeta)=0,\]
where the final equality follows from the fact that, for each $v\notin \Sigma$,
both $\delta_v(x)$ and $\res_v(\zeta)$ are unramified. This proves the lemma. 
\end{proof}
 
\begin{theorem} \label{thm:kummer_starting} 
Suppose that Conditions \ref{assumption:simple_gamma_mod} and \ref{cond_E_new} are satisfied. 
Then there is a positive constant $c_0$ such that we have
\[\#\{d\in \Z  ~~:~~d\textup{ squarefree},  ~|d|<X,~ \Sel_2(A_d/\Q) =\left \langle \zeta, \delta_d(A[2]) \right \rangle \} \sim  c_0X(\log X)^{\frac{1}{[L:\Q]}-1}.\]
\end{theorem}

\begin{proof}
By Lemma \ref{lem:FbL_nonempty}, the family $\FbL$ is non-empty.
In the presence of Condition \ref{cond_E_new}, the case $r=0$ of Theorem \ref{thm:intro_distr}
shows that the proportion of $d\in \FbL$ with $\dim\Sel_2(A_d/\Q)=2g+1$ is equal to  $\alpha(0)>0.41$.
In particular, this proportion is positive. Condition \ref{cond_E_new}  guarantees that, for any
$d\in \FbL$, we have $\zeta \in \Sel_2(A_d/\Q)$. Moreover, for any $d\in \FbL$ which is divisible
by at least one prime $p\notin \Sigma$, it follows from the cases $v=p$ of Lemmas
\ref{lem:kummer_image_basic}  and \ref{prop:KummerIntersection}(2) that $\dim \delta_d(A[2])=2g$
and $\delta_d(A[2])\cap \{0,\zeta\}=0$. Consequently, we have $\dim\Sel_2(A_d/\Q)\geq 2g+1$, with equality if and only if  $\Sel_2(A_d/\Q) =\left \langle \zeta, \delta_d(A[2]) \right \rangle$. Combining this discussion with \Cref{prop:asymptotic_cF}, we see that there is a positive constant $c_b$ such that 
\[\#\{d\in \FbL  ~~:~~  ~|d|<X,~ \Sel_2(A_d/\Q) =\left \langle \zeta, \delta_d(A[2]) \right \rangle \} \sim  c_bX(\log X)^{\frac{1}{[L:\Q]}-1}.\]
On the other hand, if $d\in \Z$ is a squarefree integer such that $\zeta \in \Sel_2(A_d/\Q)$,
then it follows from  \Cref{prop:KummerIntersection}(2) that any prime divisor $p$ of $d$ with
$p\notin \Sigma$ splits completely in $L/\Q$. Consequently, we have $d\in \cF_{L}^{b'}$ for some
$b'\in  \prod_{v\in \Sigma}\Q_v^{\times}/\Q_v^{\times 2}$.
The family $\cF_{L}^{b'}$ has the same asymptotics as the family $\cF_{L}^{b}$,
possibly with a different positive leading constant.
Theorem \ref{thm:intro_distr} then shows that the limit governing the proportion of $d$ in $\cF_{L}^{b'}$
for which  $\Sel_2(A_d/\Q) =\left \langle \zeta, \delta_d(A[2]) \right \rangle$
exists (depending on the value of $b'$ it may be zero). Summing the contribution to the
left hand side of the statement from each of the finitely many possibilities for $b'$ gives the result. 
\end{proof}

Let $\bA_{\Q}$ be the ring of ad\`eles of $\Q$, and for an integer $d$, recall 
from the introduction that $Y_{\zeta,d}(\bA_\Q)^{\Br(Y_{\zeta,d})}$
denotes the Brauer--Manin set. Recall that we have
$Y_{\zeta,d}(\Q)\subset Y_{\zeta,d}(\bA_\Q)^{\Br(Y_{\zeta,d})}\subset Y_{\zeta,d}(\bA_\Q)$,
and that if we have $Y_{\zeta,d}(\bA_\Q)^{\Br(Y_{\zeta,d})}\neq\emptyset$,
then we say that there is no Brauer--Manin obstruction to the existence of rational points
on $Y_{\zeta,d}$.
We now prove \Cref{thm:kummer_intro} from the introduction.

\begin{proposition}\label{prop:noBrauerManinObs}
Suppose that Conditions \ref{assumption:simple_gamma_mod} and \ref{cond_E_new} are satisfied. 
Then we have
\[\#\{d\in \Z  ~~:~~d\textup{ squarefree},  ~|d|<X,~  Y_{\zeta,d}(\bA_\Q)^{\Br(Y_{\zeta,d})}\neq\emptyset\} \asymp  X(\log X)^{\frac{1}{[L:\Q]}-1}.\]
\end{proposition}


\begin{proof}
Suppose that $d$ is a squarefree integer such that $\Sel_2(A_d/\Q)=\left \langle \zeta, \delta_d(A[2])\right \rangle$.
Denote by $\sha(A_d/\Q)_{\nd}$ the quotient of the Shafarevich--Tate group of $A_d$ by its maximal divisible subgroup.
Since  $\Sel_2(A_d/\Q)/\delta_d(A[2])$ surjects onto $\Sha(A_d/\Q)[2]$, we have $\dim \Sha(A_d/\Q)[2]\leq 1$.
On the other hand, it follows from \cite{MR2915483}*{Proposition 2.12(a)} that $\dim \Sha(A_d/\Q)_{\nd}[2]$ is even.
Thus we have $\Sha(A_d/\Q)_{\nd}[2]=0$. In particular, the class of $\zeta$ in $\Sha(A_d/\Q)$ is divisible,
hence in the kernel of the Cassels--Tate pairing.
By \cite[Proposition 18]{MR4095422} we have, in the notation of that result, $Y_{\zeta,d}(\bA_{\Q})^{\Br_{1/2}}\neq \emptyset$.
\cite[Theorem 2]{MR4095422} then shows that $Y_{\zeta,d}(\bA_{\Q})^{\Br(Y_{\zeta,d})}\neq \emptyset$, as desired.
\end{proof}

\begin{corollary}\label{cor:loc_soluble_fibres}
  Assume the hypotheses of Proposition \ref{prop:noBrauerManinObs}, and assume, in addition,
  that the $2$-primary part of the Shafarevich--Tate group of each quadratic twist of $A$ is finite. 
  Then we have
  \[
    \#\{d\in \Z  ~~:~~d\textup{ squarefree},  ~|d|<X,~  Y_{\zeta,d}(\Q)\neq\emptyset\} \asymp  X(\log X)^{\frac{1}{[L:\Q]}-1}.
  \]
  Moreover, when squarefree integers $d$ for which $Y_{\zeta,d}$ is everywhere
  locally soluble are ordered by $|d|$, a positive proportion satisfy $Y_{\zeta,d}(\Q)\neq \emptyset$.
\end{corollary}
\begin{proof}
  As in the final paragraph of the proof of Theorem \ref{thm:kummer_starting}, it follows from \Cref{prop:asymptotic_cF} that
  the number of squarefree integers $d$ such that $|d|<X$ and $Y_{\zeta,d}$ is everywhere locally soluble, is $\ll X(\log X)^{\frac{1}{[L:\Q]}-1}$.
  By \cite{Manin}, the Brauer--Manin obstruction accounts for all failures of the Hasse
  principle for torsors under abelian varieties with finite Shafarevich--Tate group.
  In fact, the proof of that result shows that to obtain the conclusion
  in the special case of $2$-coverings, one need only assume the finiteness
  of the $2$-primary part of the Shafarevich--Tate group.
  The result now follows from combining Theorem~\ref{thm:kummer_starting} and Proposition \ref{prop:noBrauerManinObs}.
\end{proof}

Per \cite[Section 6]{MR3519097}, we can associate to $\zeta$ a \textit{generalized Kummer variety} $\sX_\zeta/\Q$. Explicitly, denote by $\widetilde{Y}_\zeta$ the blow up of $Y_\zeta$ at the fixed points of the involution $\iota$. Then $\iota$ extends to an involution $\widetilde{\iota}$ of $\widetilde{Y}_\zeta$, and the quotient $\sX_\zeta=\widetilde{Y}_\zeta/\widetilde{\iota}$ is smooth and is the sought Kummer variety.  For each squarefree integer $d$, when we repeat this construction with $Y_{\zeta,d}$ in place of $Y_\zeta$, we obtain the same Kummer variety. Thus for each $d$, we have a natural degree $2$ morphism $f_d\colon \widetilde{Y}_{\zeta,d} \to \sX_\zeta$. For each point $x\in \sX_\zeta(\Q)$, there is a squarefree integer $d$ such that $x$ lifts to $\widetilde{Y}_{\zeta,d}(\Q)$. Thus we have  
\[\sX_{\zeta}(\Q)=\bigcup_{d\in \Z \textup{ sq. free}}f_d\big(\widetilde{Y}_{\zeta,d}(\Q)\big).\]

\begin{corollary} \label{cor:Hasse_Kummer_first}
Suppose that Conditions \ref{assumption:simple_gamma_mod} and \ref{cond_E_new} are satisfied. 
Assume further that the $2$-primary part of the Shafarevich--Tate group of each quadratic twist of $A$ is finite.  Then we have
\[\#\{d\in \Z  ~~:~~d\textup{ squarefree},  ~|d|<X,~  f_d(\widetilde{Y}_{\zeta,d}\big(\Q)\big)\neq \emptyset\} \asymp  X(\log X)^{\frac{1}{[L:\Q]}-1}.\]
In particular, we have $\sX_\zeta(\Q)\neq \emptyset$. 
\end{corollary}

\begin{proof}
Since $\zeta \neq 0$, we have $\widetilde{Y}_{\zeta,d}(\Q)\neq \emptyset$ if and only if $Y_{\zeta,d}(\Q)\neq \emptyset.$ The result now follows from Corollary \ref{cor:loc_soluble_fibres}. 
\end{proof}

\begin{remark}
Condition \ref{cond_E_new} implies, in particular, that $\sX_{\zeta}$ is everywhere locally soluble. Indeed, as explained in Remark \ref{rmk:condition_E}, for places $v\in \Sigma$ it is built into the definition of Condition \ref{cond_E_new} that $\sX_{\zeta}(\Q_v)\neq \emptyset$. For a place $v\notin \Sigma$, it follows from the fact that $\zeta$ is unramified at $v$ and $A$ has good reduction at $v$ that $Y_{\zeta}$, hence  also $\sX_{\zeta}$, has a $\Q_v$-point.  
\end{remark}

Assuming finiteness of  relevant Shafarevich--Tate groups, a consequence of Corollary \ref{cor:Hasse_Kummer_first} is that, when  Conditions \ref{assumption:simple_gamma_mod} and \ref{cond_E_new} are satisfied,  the Kummer variety $\sX_\zeta$ satisfies the Hasse principle. Results of this type appear in several places in the literature \cite{MR2183385,MR3519097,MR3932783,MR4934780}. Of those, the works of Skorobogatov--Swinnerton-Dyer \cite{MR2183385} and Harpaz \cite{MR3932783} treat Kummer varieties associated to abelian varieties with full rational $2$-torsion.  Like Corollary \ref{cor:Hasse_Kummer_first}, these results are also conditional on finiteness of relevant (parts of) Shafarevich--Tate groups, and similarly establish the existence of rational points on $\sX_\zeta$ by deducing the existence of points on the $2$-coverings $Y_{\zeta,d}$. However, they establish only the existence of a $d$ for which $Y_{\zeta,d}(\Q)\neq \emptyset$,  while Corollary \ref{cor:Hasse_Kummer_first} gives precise information about the number of such $d$.

To the best of our knowledge, even the existence result of Corollary \ref{cor:Hasse_Kummer_first} is new. As in Remark \ref{rmk:condition_E}, with the parity condition (2) removed,  Condition \ref{cond_E_new} is a direct analogue of  \cite[Condition E]{MR2183385}. On the other hand, even accounting for the difference in setup, Condition \ref{assumption:simple_gamma_mod} appears at first sight to be quite different to both the remaining condition in \cite[Theorem 1]{MR2183385}  and the single condition in \cite[Theorem 2.8]{MR2183385} (as in \cite[Remark 2.10]{MR3932783}, the absence of a version of \cite[Condition E]{MR2183385}  in the work of Harpaz is due to the presence of a second descent step). 
However, we show in Proposition \ref{cond_Z_prop} below that variants of those conditions in fact imply Condition \ref{assumption:simple_gamma_mod}. Combined with Example \ref{kummer_hyp_ex} below, Proposition \ref{cond_Z_prop}  also serves to provide a large supply of abelian varieties for which Condition \ref{assumption:simple_gamma_mod} holds.

\begin{notation}
For a place $v$ of $\Q$, denote by $\Phi_v$ the group of connected components of the geometric special fibre of the N\'{e}ron model of $A$ at $v$.
Recall from  \cite[Definition 2.1, Remark 2.2]{MR3932783} that a \textit{$2$-structure} $M$ for $A$ is a set of $2g$ odd places such that $A$
has semistable reduction of toric rank $1$ at each $v\in M$, and such that the natural map  $A[2] \to \oplus_{v\in M}\Phi_v/2\Phi_v$ is an isomorphism.
By \cite[Remark 2.2]{MR3932783}, for each $v\in M$, we have $\Phi_v[2^\infty]\cong \Z/2\Z$. For such $v$, by non-degeneracy of the Weil pairing
$e\colon A[2]\times A[2] \to \F_2$, there is a unique $P_v\in A[2]$ such that the reduction map $A[2] \to \Phi_v/2\Phi_v\cong \F_2$  is given by
$x \mapsto e(x,P_v)$. 
\end{notation}

\begin{definition}
We say that a $2$-structure $M$ is \textit{strict} if, for any two distinct elements $v,v'$ of $M$, the Weil pairing between $P_v$ and $P_{v'}$ is non-trivial, i.e. $e(P_v,P_{v'})=1$.
\end{definition}

\begin{example} \label{kummer_hyp_ex}
Let $g$ be a positive integer and let $a_1,...,a_{2g+1}\in \Z$ be pairwise distinct. Let $f(x)=\prod_{i=1}^{2g+1}(x-a_i)$. Let $y^2=f(x)$ be the corresponding genus $g$ hyperelliptic curve, and let $A$ be its Jacobian. Let $M=\{p_1,...,p_{2g}\}$ be a set of odd primes such that, for each $i\in \{1,...,2g\}$, we have $\ord_{p_i}(a_i-a_{2g+1})=1$ and  $\ord_{p_i}(a_j-a_{j'})=0$ for all distinct  $j,j'\in \{1,...,2g+1\}$ with $\{j,j'\}\neq \{i,2g+1\}$. Arguing as in  \cite[Proof of Theorem 1.3 assuming Theorem 2.8]{MR3932783}, we see that $M$ is a $2$-structure for $A$, and that for each $p_i\in M$, the point $P_{p_i}\in A[2]$ is given by the class of the divisor $(a_i,0)-(a_{2g+1},0)$. From the description of the Weil pairing on $A[2]$ given in \cite[Section 5.2]{MR2964027}, we see that this $2$-structure is strict. 
\end{example}

\begin{proposition} \label{cond_Z_prop}
Suppose that $A$ has a strict $2$-structure $M$ over which $\zeta$ is unramified (i.e. such that $\res_v\zeta$ is unramified for all $v\in M$). Then:
\begin{enumerate}[leftmargin=*, label=\upshape{(\arabic*)}]
  \item\label{item:condsimplegamma} Condition \ref{assumption:simple_gamma_mod} is satisfied;
  \item\label{item:condEnew} if there is $b\in \prod_{v\in \Sigma}\Q_v^{\times}/\Q_v^{\times 2}$ satisfying part (1) of Condition \ref{cond_E_new}, then Condition \ref{cond_E_new} is satisfied.  
\end{enumerate}
\end{proposition}

\begin{proof}
  \ref{item:condsimplegamma}. Let $v\in M$ and let $\sigma_v$ be any element of $G_\Q$ that topologically generates the tame inertia group at $v$.
Recall from Definition \ref{def:definition_of_gamma} the map $\gamma\colon G_\Q\to \End(A[2])$.
We claim that we have $\gamma(\sigma_v)=\phi_{P_v}$, where $\phi_{P_v}\in \End(A[2])$ sends each $x\in A[2]$ to $e(x,P_v)P_v$
(cf. Proposition \ref{transvections_proposition}). Indeed, fix $x\in A[2]$. If we have $e(x,P_v)=0$, then the image
of $x$ is $2$-divisible in $\Phi_v$, hence also in $A(\Q_v^{\ur})$, where $\Q_v^{\ur}$ denotes the maximal unramified extension of $\Q_v$. Thus we have $\gamma(\sigma_v)(x)=\delta(x)(\sigma_v)=0$. Conversely, if $e(x,P_v)=1$, then $x$ has non-trivial image in $\Phi_v/2\Phi_v$, and  by \cite[Lemma 3.6]{MR3932783} we have $\delta(x)(\sigma_v)=P_v$. That is, $\gamma(\sigma_v)(x)=P_v$. This proves the claim. 

To complete the proof, we begin by noting that for each $v\in M$, the image of $\sigma_v$ in $\Gal(\Q(A[4])/\Q)$ lies in $\Gal(L(A[4])/L)$.
Indeed, since $\res_v\zeta$ is unramified, we have $\zeta(\sigma_v)=0$, so that $\sigma_v$ is contained in $\ker(\zeta)=G_L$.
By the claim, the image $\Gamma\leq \End(A[2])$ of $\Gal(L(A[4])/L)$ under $\gamma$
contains the transvections $\{\phi_{P_v}\}_{v\in M}$. Since $M$ is strict,
Proposition \ref{transvections_proposition} implies that the Lie algebra generated by $\Gamma$
contains all self-adjoint endomorphisms, which implies that Condition \ref{assumption:simple_gamma_mod} is satisfied.

\ref{item:condEnew}. Cf.  \cite[Lemmas 5 and 8]{MR2183385}. Suppose $b=(b_v)_{v\in \Sigma}$ satisfies part  (1) of Condition \ref{cond_E_new}. Fix a choice of $w\in M$, noting that $w\in \Sigma$. We claim that we can find $b_w'\in \Z_w^{\times}/\Z_w^{\times 2}$ such that the tuple $b'$, obtained from $b$ by replacing the $w$-component with $b_w'$, again satisfies part (1)  of Condition \ref{cond_E_new}. Indeed, first suppose that $\res_w(\zeta)=0$. Then every quadratic twist of $Y_{\zeta}$ has a $\Q_v$-point. Moreover, $w$ splits completely in $L/\Q$, hence for any element   $\zeta'\in H^1(L/\Q,A[2])$, we have $\res_w(\zeta')=0$ also. If $\zeta'\notin \{0,\zeta\}$, then  $\zeta'\notin \cS_b$ by assumption, so there is $v\in \Sigma$ such that $Y_{\zeta',b_v}(\Q_v)= \emptyset$ (cf. Remark \ref{rmk:condition_E}).  Since $\res_w(\zeta')=0$, we cannot have $v=w$. Thus we can replace $b_w$ with \textit{any} $b_w' \in \Z_w^{\times}/\Z_w^{\times 2}$, and part  (1)  of Condition \ref{cond_E_new} continue to hold. Next, suppose that $\res_w(\zeta)\neq 0$. In this case, since $Y_{\zeta,b_w}(\Q_w)\neq \emptyset$, it follows from the first claim in the proof of \cite[Lemma 3.16]{MR3932783} that we already have $b_w\in   \Z_w^{\times}/\Z_w^{\times 2}$.  

Now fix $b=(b_v)_{v\in \Sigma}$ satisfying part  (1)  of Condition \ref{cond_E_new}. As is possible by the claim, we suppose that $b_w\in \Z_w^{\times}/\Z_w^{\times 2}$, where $w\in M$ is fixed as above. By assumption, $A_{b_w}$ is the quadratic twist of $A$ by an unramified (possibly trivial) character.   Let $u$ be the unique non-trivial class in  $\Z_w^{\times}/\Z_w^{\times 2}$. Since $\zeta\in \sS_{b_w,w}$ is unramified at $w$,  it follows from \cite[Lemma 3.25]{MR3932783} that $\zeta$ lies in both  $\sS_{b_w,w}$ and  $\sS_{ub_w,w}$. Moreover, given any $\zeta'\in H^1(L/\Q,A[2])$, we have that $\zeta'$ is unramified at $w$ also, hence by the same lemma we have $\zeta' \in  \sS_{b_w,w}$ if and only if  $\zeta'\in \sS_{ub_w,w}$. Letting $b'$ denote the tuple obtained from $b$ by replacing the $w$-component with $ub_w$, we see that both $b$ and $b'$ satisfy part  (1)   of Condition \ref{cond_E_new}. We claim that exactly one of $b$ and $b'$ satisfies part (2) of Condition \ref{cond_E_new}. For this, it suffices to show that the invariants $\kappa_w(b_w)$ and $\kappa_w(ub_w)$ (the contributions from the place $w$ to the parity) are distinct. Now $\{b_w,ub_w\}=\{1,u\}$, and we trivially have $\kappa_w(1)=0$.  On the other hand, $\kappa_w(u)$ is, by definition, precisely the dimension  modulo $2$ of the vector space denoted $\overline{W}_w$ in  \cite[Lemma 3.25]{MR3932783}. It is shown there that this dimension is equal to $1$.  
\end{proof}

 \begin{corollary} \label{cor:kummer_finite}
Suppose that $A$ has a strict $2$-structure over which $\zeta$ is unramified. Suppose further that there exists $b\in \prod_{v\in\Sigma}\Q_v^{\times}/\Q_v^{\times 2}$ satisfying part  (1)  of Condition \ref{cond_E_new}.
 
Assume  that the $2$-primary part of the Shafarevich--Tate group of each quadratic twist of $A$ is finite. Then we have
\[\#\{d\in \Z  ~~:~~d\textup{ squarefree},  ~|d|<X,~  f_d(\widetilde{Y}_{\zeta,d}\big(\Q)\big)\neq \emptyset\} \asymp  X(\log X)^{\frac{1}{[L:\Q]}-1}.\]
In particular, we have $\sX_\zeta(\Q)\neq \emptyset$. 
\end{corollary}

\begin{proof}
Combine Proposition \ref{cond_Z_prop} with Corollary \ref{cor:Hasse_Kummer_first}.
\end{proof}

 \begin{remark} \label{rmk:final_kummer}
As alluded to above, the existence part of Corollary \ref{cor:kummer_finite} is closely related to \cite[Theorem 1]{MR2183385}. These results are not directly comparable. While \cite[Theorem 1]{MR2183385} considers Kummer varieties associated to a product of two elliptic curves over a general number field, Condition \ref{assumption:simple_gamma_mod} ensures that any abelian variety satisfying the conditions of Corollary \ref{cor:kummer_finite} is simple. Nevertheless, the similarity between the statements of these results should be clear. While the existence part of  Corollary \ref{cor:kummer_finite} can very likely be established by the methods of \cite{MR3932783} (even without the innovative second descent step, cf. \cite[Remark 2.10]{MR3932783}), we include it to highlight that the analytic methods underpinning the proof Theorem \ref{thm:intro_distr} are capable of establishing results on the Hasse principle for Kummer varieties of an appreciable quality to those in the literature. 

In fact, we would speculate that much more may be within reach. Indeed, the $L=\Q$ case of Theorem \ref{thm:intro_distr} represents only a small part of the breakthrough results of Smith \cite{Smith1,Smith2}. Adapting Smith's results in full to Frobenian families of twists  such as those appearing in Theorem \ref{thm:intro_distr}  would likely allow one to prove a version of Corollary \ref{cor:Hasse_Kummer_first} that applies over general number fields, allows very general $2$-torsion structures, and incorporates information from higher descents.  
 \end{remark}




\bibliographystyle{plain}

\large\bibliography{references}

\end{document}